\documentclass[a4paper,11pt,leqno
]{article}
\usepackage{enumitem}
\usepackage[utf8]{inputenc}
\usepackage[T1]{fontenc}
\usepackage{amsthm}
\usepackage[font={footnotesize}]{caption}
\usepackage{amssymb}
\usepackage{mathrsfs}
\usepackage{amsmath}
\usepackage{amsfonts}
\usepackage{mathtools}
\usepackage{graphicx}
\usepackage{float}
\usepackage{dsfont}
\usepackage[a4paper]{geometry}
\usepackage{url}
\usepackage{csquotes}
\usepackage{IEEEtrantools}
\usepackage[toc,title,page]{appendix}
\usepackage{apptools}

\usepackage{constants}
\usepackage{tikz-cd}
\usepackage{multicol}
\usepackage{parcolumns}
\usepackage{stackrel}
\usepackage{comment}

\usepackage
{hyperref}

\newtheorem{The}{Theorem}[section]
\newtheorem{Lemme}[The]{Lemma}
\newtheorem{Prop}[The]{Proposition}
\newtheorem{Cor}[The]{Corollary}
\theoremstyle{definition}

\theoremstyle{remark}
\newtheorem{Rk}[The]{Remark}

\AtAppendix{\counterwithin{The}{section}}

\title{\normalsize 
\textbf{
CLUSTER CAPACITY FUNCTIONALS AND ISOMORPHISM THEOREMS FOR GAUSSIAN FREE FIELDS
}}
\author{}
\date{}
\makeatletter
\newcommand{\Ed}{-}
\newcommand{\E}{\mathbb{E}}

\newcommand{\R}{\mathbb{R}}
\newcommand{\Z}{\mathbb{Z}}
\newcommand{\N}{\mathbb{N}}

\newcommand{\F}{\mathcal{F}}
\newcommand{\G}{\mathcal{G}}
\newcommand{\K}{\mathcal{K}}
\newcommand{\A}{\mathcal{A}}
\newcommand{\W}{\mathcal{W}}
\renewcommand{\P}{\mathbb{P}}
\newcommand{\eps}{\varepsilon}

\newcommand{\I}{{\cal I}}
\newcommand{\V}{{\cal V}}
\renewcommand{\phi}{\varphi}
\renewcommand{\tilde}{\widetilde}
\renewcommand{\hat}{\widehat}
\renewcommand{\epsilon}{\varepsilon}
\renewcommand\theequation{\thesection.\arabic{equation}}

\@addtoreset{equation}{section}

\newcommand{\tend}[2]{\displaystyle\mathop{\longrightarrow}_{#1\rightarrow#2}}
\newconstantfamily{c}{symbol=c}
\newcommand{\diff}{\,\mathrm{d}}
\newcommand{\h}{\mathit{\mathbf{h}}}
\makeatother

\usepackage{color}
\definecolor{Red}{rgb}{1,0,0} 
\definecolor{Blue}{rgb}{0,0,1}
\definecolor{Olive}{rgb}{0.41,0.55,0.13}
\definecolor{Yarok}{rgb}{0,0.5,0}
\definecolor{Green}{rgb}{0,1,0}
\definecolor{MGreen}{rgb}{0,0.8,0}
\definecolor{DGreen}{rgb}{0,0.55,0}
\definecolor{Yellow}{rgb}{1,1,0}
\definecolor{Cyan}{rgb}{0,1,1}
\definecolor{Magenta}{rgb}{1,0,1}
\definecolor{Orange}{rgb}{1,.5,0}
\definecolor{Violet}{rgb}{.5,0,.5}
\definecolor{Purple}{rgb}{.75,0,.25}
\definecolor{Brown}{rgb}{.75,.5,.25}
\definecolor{Grey}{rgb}{.7,.7,.7}
\definecolor{Black}{rgb}{0,0,0}
\def\red{\color{Red}}
\def\blue{\color{Blue}}

\newcommand{\AP}[1]{{\blue{[\small AP: #1]}}}
\newcommand{\PF}[1]{{\red{[\small PF: #1]}}}

\usepackage{titletoc}

\dottedcontents{section}[4em]{}{2.9em}{0.7pc}
\dottedcontents{subsection}[0em]{}{3.3em}{1pc}
\newenvironment{myitemize}
{ \begin{itemize}
    \setlength{\itemsep}{1pt}
    \setlength{\parskip}{1pt}
    \setlength{\parsep}{1pt}     }
{ \end{itemize}                  }

\usepackage{titlesec}
\titleformat{\subsection}[runin]{\normalfont\bfseries}{\thesubsection.}{.5em}{}[.]\titlespacing{\subsection}{0pt}{2ex plus .1ex minus .2ex}{.8em}
\titleformat{\subsubsection}[runin]{\normalfont\itshape}{\thesubsubsection.}{.3em}{}[.]\titlespacing{\subsubsection}{0pt}{1ex plus .1ex minus .2ex}{.5em}
\titleformat{\paragraph}[runin]{\normalfont\itshape}{\theparagraph.}{.3em}{}[.]\titlespacing{\paragraph}{0pt}{1ex plus .1ex minus .2ex}{.5em}

\begin{document}




\maketitle

\begin{center}
\vspace{-1.3cm}
Alexander Drewitz$^1$, Alexis Pr\'evost$^2$ and Pierre-Fran\c cois Rodriguez$^3$ 
\end{center}

\vspace{-0.1cm}
\begin{abstract}
\centering
\begin{minipage}{0.90\textwidth}
\vspace{0.3cm}
    We investigate 
    level sets of the Gaussian free field on continuous transient metric graphs $\widetilde{\mathcal G}$ and study the capacity of its level set clusters.
    We prove, without any further assumption on the base graph $\mathcal{G}$, that the capacity of sign clusters on $\widetilde{\mathcal G}$ is finite almost surely. This leads to a new and effective criterion to determine whether the sign clusters of the free field on $\widetilde{\mathcal G}$ are bounded or not. It also elucidates why the critical parameter for percolation of level sets on~$\widetilde{\mathcal G}$ vanishes in most instances in the massless case and establishes the continuity of this phase transition in a wide range of cases, including all vertex-transitive graphs. When the sign clusters on $\widetilde{\mathcal G}$ do not percolate, we further determine by means of isomorphism theory the exact law of the capacity of compact clusters at any height. Specifically, we derive this law from an extension of Sznitman's refinement of Lupu's recent isomorphism theorem relating the free field and random interlacements, proved along the way, and which holds under the sole assumption that sign clusters on $\widetilde{\mathcal G}$ are bounded. Finally, we show that the law of the cluster capacity functionals obtained in this way actually characterizes the isomorphism theorem, i.e.~the two are equivalent.

\end{minipage}
\end{abstract}

\vspace{0.6cm}
\begin{minipage}{0.85\textwidth}
{\small
}
\end{minipage}

\vspace{3cm}
\begin{flushleft}

\noindent\rule{5cm}{0.4pt} \hfill December 2021 \\
\bigskip
\begin{multicols}{2}

$^1$Universit\"at zu K\"oln\\
Department of Mathematics\\
and Computer Science\\
Division of Mathematics  \\
Weyertal 86--90 \\
50931 K\"oln, Germany. \\
\url{adrewitz@uni-koeln.de}\\[2em]

$^2$University of Cambridge\\
Faculty of Mathematics \\
Wilberforce Road\\
Cambridge CB3 0WA, United Kingdom\\
\url{ap2195@cam.ac.uk}\\[2em]

\columnbreak
\thispagestyle{empty}
\bigskip
\medskip
\hfill$^3$Imperial College London\\
\hfill Department of Mathematics\\
\hfill London SW7 2AZ \\
\hfill United Kingdom\\
\hfill \url{p.rodriguez@imperial.ac.uk} 
\end{multicols}
\end{flushleft}

\newpage
\newpage
\section{Introduction}
\label{sec:intro}
In this article, we consider the Gaussian free field $\varphi$ on the cable system $\widetilde{\mathcal{G}}$ associated to an arbitrary transient weighted graph $\mathcal{G}$; see the discussion around \eqref{deflambdarho} below for the precise setup. Cable processes have increasingly proved an insightful object of study, as shown for instance in the recent articles \cite{MR3502602}, \cite{MR3492939}, \cite{LW18}, \cite{DrePreRod2},  \cite{DiWi} and \cite{werner2020clusters}. In the present work, we investigate a well-chosen observable, the capacity of finite clusters in the excursion set $E^{\geq h}$ of $\varphi$ above height~$h\in \mathbb{R}$, see \eqref{eq:capobservable} below. This quantity features prominently in our article~\cite{DrePreRodCritExp}. Our main result, stated below in Theorem \ref{T:main} -- see also Section \ref{sec:results} for a more exhaustive discussion -- underlines the central nature of this observable and unveils some of its deeper ramifications.

To wit, our findings imply for instance that the cluster capacity observable at height $h=0$ is finite almost surely, for \textit{any} transient graph $\mathcal{G}$, see Theorem~\ref{T:main},\ref{maint1} (our setup allows for a killing measure, including the degenerate case of Dirichlet boundary conditions, which will play an important role below). This immediately leads to a much improved understanding of why the height $h=0$ tends to be critical for the percolation problem $\{ E^{\geq h }: h \in \R\}$ in the massless case, i.e.\ in the absence of killing, and more generally when $\h_{\text{kill}}<1$ (see \eqref{defh0} below). A simple criterion, see \eqref{capcondition} on p.\pageref{capcondition} and Theorem~\ref{T:main},\ref{maint1}, which covers an extensive number of cases, can then be used to check if the sign clusters of $\varphi$ percolate or not. 

For instance, see Corollary~\ref{C:noperc}, as a consequence of this criterion, our results yield that the sign clusters of $\varphi$ on \textit{any} vertex-transitive graph with no killing are bounded and thus establish the phase transition of $\{ E^{\geq h }: h \in \R\}$ as being second order. Corresponding results hold for the loop soup $\mathcal{L}_{1/2}$, see  Corollary~\ref{Cor:loopsoups}; see also the discussion following Theorem \ref{T:main} regarding the current state of affairs. 
 
When the sign clusters of $\varphi$ are bounded -- which holds e.g.~when \eqref{capcondition} holds -- we are able to identify the distribution of the cluster capacity observable at any level $h \in \R$, see Theorem~\ref{T:main},\ref{maint3} below. This law is explicitly characterized by \eqref{eq:laplacecap}, introduced on p.\pageref{eq:laplacecap} (see~also \eqref{eq:hdensity} for the corresponding density). Moreover, we show that this information is \textit{equivalent} to the `strong Ray-Knight-type' isomorphism recently derived in \cite{MR3492939} (refining \cite{MR3502602}, see also \eqref{eqcouplingintergffszn} on p.\pageref{eqcouplingintergffszn}) under slightly stronger assumptions than those to follow. This identity relates the free field itself with the local times of random interlacements on $\widetilde{\mathcal{G}}$. Thus, we effectively obtain a characterization of an isomorphism theorem (in the non-interacting case) in terms of the free field alone. In fact, for massless graphs (or even if $\h_{\text{kill}}<1$) our results imply under $\hyperref[eq:laplacecap]{(\emph{\textnormal{Law}}_0)}$ the dichotomy $\tilde{h}_* \in{\{0,\infty\}}$, where $\tilde{h}_*$ refers to the corresponding critical level; cf.~Theorem~\ref{T:main},\ref{maint4}. We further refer to the forthcoming article \cite{Pre1} for sharpness and limitations to the validity of these results. The identity \eqref{eq:laplacecap} is derived in \cite{DrePreRodCritExp} by means of differential formulas, and has important consequences regarding the (near-)critical regime for level sets of $\varphi$ on $\tilde{\mathcal{G}}$; see \cite{DrePreRodCritExp} regarding these matters.

\bigskip
We now introduce our setup and refer to Section~\ref{s:usefulresults} for details. We consider a transient weighted graph $ \G= (\overline{G},\bar{\lambda},\bar{\kappa}),$ where $\overline{G}$ is a finite or countably infinite set, $\bar{\lambda}_{x,y}\in{[0,\infty)}$, $x,y\in\overline{G},$ are non-negative weights satisfying $\bar{\lambda}_{x,y} =\bar{\lambda}_{y,x} \geq 0$ and $\bar{\lambda}_{x,x}=0$ for all $x,y\in \overline{G}.$ Furthermore, $\bar{\kappa}_x\in{[0,\infty]}$, $x\in\overline{G},$ is a killing measure, possibly infinite. To deal with the latter in a convenient way, given $ \G= (\overline{G},\bar{\lambda},\bar{\kappa}),$ we introduce the triplet $(G,\lambda, \kappa)$, to which we will mostly refer throughout the article, by setting $(G,\lambda, \kappa)= (\overline{G}^M,\bar{\lambda}^M,\bar{\kappa}^M)$, the latter being defined in \eqref{eq:defGfinite}, with $M$ a certain set of `mid-points' given by \eqref{eq:0bcreduction1new}. In particular, this definition entails that $(G,\lambda, \kappa)= (\overline{G},\bar{\lambda},\bar{\kappa})$ whenever $\bar{\kappa}_x < \infty$ for all $ x\in \overline{G}$. 
Otherwise $(G,\lambda, \kappa)$ is obtained by suitable `enhancement' of~$ \G$ (exploiting network equivalence). As a result, the killing measure $\kappa$ is finite everywhere, i.e.~$\kappa_x < \infty$ for all $x \in G$. 

We always tacitly assume that the induced graph $(G,E)$ with edge set $E =\{ \{x,y \}: x,y \in G,\,  \lambda_{x,y}>0\}$ is connected and locally finite. We write $x\sim y$ when $\{x,y\}\in{E},$ and we define
\begin{equation}
\label{deflambdarho}
\begin{split}
  &\lambda_x=\kappa_x+\sum_{y\in{G}}\lambda_{x,y},\ \rho_x=\frac{1}{2\kappa_x} \text{ for }x\in{G} \text{  and  } \rho_{x,y}=\frac{1}{2\lambda_{x,y}}\text{ for }x\sim y\in{G}  
   \end{split}
\end{equation}
(with $\rho_x=\infty$ when $\kappa_x=0$
). One naturally associates to $ \G$ a continuous version $\tilde{ \G},$ the corresponding cable system or metric graph, obtained by replacing each edge $e =\{x,y \} \in E$ by an open interval $I_e$ of length $\rho_{x,y}$, glued to $G$ through its endpoints $x$ and $y.$ One further attaches to each vertex $x\in G$ an additional interval $I_x$ isometric to $[0,\rho_x),$ glued to $x$ through $0$ (we refer to Section~\ref{S:I_x} and Remark \ref{R:lawh},\ref{R:removecables2} for their raison-d'\^etre).

One then defines (e.g.\ in terms of its associated Dirichlet form, see \eqref{eq:quadraticform} and \eqref{Dirichlet} below for details) a diffusion process $({X}_t)_{t \geq 0 }$ on $\tilde{\G} \cup \{\Delta\}$, where $\Delta$ denotes an (absorbing) cemetery state, which can be viewed as Brownian motion on the cable system.  The process $X$ induces a pure jump process $Z=(Z_t)_{t\geq0}$ on $G\cup \{\Delta\}$, which we refer to as its\textit{ trace} (or \textit{print})\textit{ on $G$}, see \eqref{traceonG}, associated to a corresponding trace form. The induced process $Z$ has the law of the continuous time Markov chain that jumps from $x \in G$ to $y \in G$ at rate $\lambda_{x,y}$ and is killed at rate $\kappa_x.$ Similarly, the trace of $X$ on $\{x\in{\overline{G}}:\,\bar{\kappa}_x<\infty\}$ has the law of the continuous time Markov chain on $\overline{G}$ that jumps from $x \in \overline{G}$ to $y \in \overline{G}$ at rate $\bar{\lambda}_{x,y}$ and is killed at rate $\bar{\kappa}_x.$
We write $P_x$ for the canonical law of $X_{\cdot}$ with starting point $x \in \tilde{\G}$, and occasionally $P_x^{\tilde{\G}}$ in place of $P_x$ to stress the dependence on the datum $\tilde{\G}$. We say that $X_{\cdot}$ is killed if $X_{\cdot}$ exits $\tilde{\G}$ via $I_x$ for some $x\in{G}$ with $\kappa_x>0$ (which is equivalent to $Z$ being killed, i.e.~entering $\Delta$). Accordingly, we define 
\begin{equation}
\label{defh0}
\h_{\text{kill}}(x)\stackrel{\text{def.}}{=}P_x(X_{\cdot} \text{ is killed}), \text{ for all }x\in\tilde{\G}.
\end{equation}
Moreover, we say that $\h_{\text{kill}}<1$ if $\h_{\text{kill}}(x)<1$ for all $x \in \tilde{\G},$ or equivalently if $\h_{\text{kill}}(x)<1$ for some $x\in \tilde{\G}$ 
(recall that $(G,E)$ is assumed to be a connected graph). An important family of graphs satisfying $\h_{\text{kill}}<1$ are \textit{massless} graphs with $\bar{\kappa} =\kappa \equiv 0$, or equivalently $\h_{\text{kill}}(\cdot)=0$.
 
Our results deal with the graph $\G$ and its associated metric graph $\tilde{\G}$, when $\G$ is transient; that is, when the Markov chain $Z$ is transient, which we tacitly assume from now on. 
In particular, the graph $\G$ may be finite when $\kappa \not\equiv 0.$ We then define the Gaussian free field on~$\tilde\G$, whose canonical law $\P^G$ (occasionally denoted as $\P^G_{\tilde{\G}}$), defined on the space $C(\tilde{\G},\R)$ endowed with the $\sigma$-algebra generated by the coordinate maps $\phi_x,$ $x\in{\tilde{\G}},$ is such that
\begin{equation}
\label{defGFF}
\text{under $\P^G,$ $(\phi_x)_{x\in{\tilde{\G}}}$ is a centered Gaussian field with covariance function $g(\cdot,\cdot)$.}
\end{equation}
Here, $g(\cdot,\cdot)$ refers to the Green density of $X_{\cdot}$ with respect to the Lebesgue measure $m$ on $\widetilde{\mathcal{G}}$, see~\eqref{Greendef}. %
The restriction of this process to $G$ has the same law as the usual Gaussian free field on $\G$ associated to the discrete Markov chain $Z$.

\bigskip
We now describe our main results, which deal with the excursion sets $E^{\geq h}\stackrel{\text{def.}}{=}\{y\in{\tilde{\G}}:\phi_y\geq h\}$ of $\varphi$, for varying height $h\in\R.$ We endow $\tilde{ \G}$ with the (geodesic) distance $d(\cdot,\cdot)$ such that all intervals $I_e$, $e\in E$, and $I_x$, when $\rho_x< \infty$, have length one (rather than $\rho_{e}$ and $\rho_x,$ respectively). 
Albeit not essential, we assume for convenience that $d$ also assigns length one to $I_x$ when $\rho_x= \infty$ (by means of some strictly increasing bijection $[0,1) \to [0,\infty)$).
The clusters, i.e.~maximal connected components, of $E^{\geq h}$, are defined as
\begin{align} \label{Ehx0}
\begin{split}
&{E}^{\geq h}(x_0)\stackrel{\text{def.}}{=} \big \{y\in{\tilde{\G}}:\,x_0\leftrightarrow y\text{ in }E^{\geq h} \big \}, \text{ for $x_0\in{\tilde{\G}}$, $h\in \R;$}
\end{split}
\end{align}
here, for measurable ${A}\subset\tilde{\G}$ and $x,y\in{\tilde{\G}}$, we write 
$\{x\leftrightarrow y \text{ in } {A}\}$ if there exists a (continuous) path from $x$ to $y$ in ${A},$ and we say that ${A}$ is connected in $\tilde{\G}$ if $z\leftrightarrow z'$ in ${A}$ for all $z,z'\in{{A}}.$ A central role in this work will be played by the cluster capacity functional
\begin{equation}
\label{eq:capobservable}
\mathrm{{\rm cap}}({E}^{\geq h}(x_0)), \text{ for $h\in \R$, $x_0\in \tilde{\G}$;}
\end{equation}
We refer to \eqref{defcap} and \eqref{defcapinfinity} below for the definition of $\text{cap}(A)$, the electrostatic capacity of $A$, for arbitrary closed, possibly unbounded subsets $A$ of $\tilde{\G}$. For instance, in case $A\subset G$ is finite (or more generally if $A'\subset \tilde{\G}$ is compact and $\partial A' = A$), then $\text{cap}(A)$ (and $\text{cap}(A')$) coincide with the usual capacity of the set $A$ for the discrete chain $Z$.

One of our interests is on the percolative properties of the set $E^{\geq h}$ (with respect to $d$). We introduce the corresponding critical parameter
\begin{equation}
\label{defh*bou}
\begin{split}
 \widetilde h_* 
&=\inf \big \{ h \in \R :  \text{ for all $x_0\in{\tilde{\G}}$, }\, \P^G(E^{\geq h}(x_0) \text{ is unbounded})=0 \big \}
 \end{split}
\end{equation}
(with the convention $\inf\varnothing= \infty$; note that $\widetilde h_*$ is equivalently defined as the smallest level $h$ such that $\P^G$-a.s.~$E^{\geq h}$ contains no unbounded connected component). A fortiori, \eqref{defh*bou} entails that for each $h< \widetilde{h}_*,$ with positive $\P^G$-probability the discrete set $E^{\geq h}\cap G$ contains a percolating connected component in the usual sense (i.e., the component is unbounded with respect to the graph distance on $(G,E)$). In other words, the corresponding critical parameter $h_*$ (see for instance (1.8) in \cite{DrePreRod2} for its definition) satisfies $h_* \geq  \widetilde h_* $. 
Other natural definitions of critical parameters associated to the sets $\{ E^{\geq h} , h \in \R\}$ exist and will be of interest, see \eqref{defh*com} and \eqref{defh*cap} below. They correspond to several natural ways of measuring the `magnitude' of clusters in $E^{\geq h}$, and \eqref{eq:capobservable} reflects one such choice, based on capacity as a measure of size. 

We now briefly introduce the process of random interlacements on $\tilde{\G},$ see \cite{MR2680403}, \cite{MR3308116} and \cite{MR2525105}, to the extent necessary to formulate our main findings; further details are provided in Section~\ref{subsec:RI}. The interlacement process will play a prominent role in the present context, due to recent isomorphisms, see \cite{MR3502602}, \cite{MR3492939} and \eqref{eqcouplingintergffszn} below, relating it to $\varphi$ in a very explicit fashion. Under a suitable probability measure $\P^I,$ for each $u>0,$ random interlacements at level $u$ on $\tilde{\mathcal{G}}$ constitute a Poisson point process $\omega_u$ with intensity $u\nu_{\tilde{\G}},$  where $\nu_{\tilde{\G}}$ is a measure on doubly non-compact 
trajectories modulo time-shift (when $\kappa \not\equiv 0,$ these trajectories may be killed by the measure $\kappa$ before escaping to infinity, i.e., they may `exit $\tilde{\G}$ via $I_x$' for some $x\in{G}$ with $\kappa_x>0$; see \eqref{defQK} and \eqref{definter} for the precise definition of $\nu_{\tilde{\G}}$). We denote by $(\ell_{x,u})_{x\in{\tilde{\G}}}$ the continuous field of local times associated to $\omega_u,$ i.e.\ the sum of the local time densities relative to the Lebesgue measure on $\tilde{\G}$ of all the trajectories in $\omega_u.$ We then define the interlacement set as $\I^u = \{x \in \tilde{\G}: \ell_{x,u} > 0 \}$, a random open subset of $\tilde{\G}$. Without any further assumptions on ${\mathcal{G}}$, it can be shown that for all $u>0,$
\begin{equation}
    \label{usualiso}
    \Big(\ell_{x,u}+\frac{1}{2}\phi_x^2\Big)_{x\in{\tilde{\G}}}\text{ has the same law under }\P^G\otimes\P^I\text{ as }\Big(\frac12(\phi_x+\sqrt{2u})^2\Big)_{x\in{\tilde{\G}}}\text{ under }\P^G;
\end{equation}
see \cite{MR2892408} for the original derivation of this result on the (discrete) base graph graph $\mathcal{G}$ in case $\kappa\equiv0$, based on the generalized second Ray-Knight theorem of \cite{MR1813843}; see also Proposition~6.3 of \cite{MR3502602} and (1.27)--(1.30) in \cite{MR3492939} for extensions to $\tilde{\G}$. We refer to Remark~\ref{R:nomassconversion} below regarding a justification for the validity of \eqref{usualiso} in the present setup, which is more general. As first observed in \cite{MR3502602}, the isomorphism \eqref{usualiso} implies a stochastic domination of each connected component of $\mathcal{I}^u$ by a level-set cluster of $\varphi$, which straightforwardly  yields (recall \eqref{defh0}) that
\begin{equation}
\label{ifhkill<1thenh_*>0}
\text{ if }\h_{\text{kill}}<1,\text{ then }\tilde{h}_*\geq0,
\end{equation}
see the paragraph following \eqref{couplingusualiso} below for details. 
The reverse inequality $\tilde{h}_*\leq0$ is an entirely different matter and has so far only been verified in a handful of cases (see below Theorem \ref{T:main} for a list). Part of our main result addresses this issue.

Under additional assumptions, refining the link between $\mathcal{I}^u$ and level-sets of $\varphi$ described above \eqref{ifhkill<1thenh_*>0}, the identity \eqref{usualiso} can be considerably strengthened. Indeed, Theorem 2.4 in \cite{MR3492939} asserts that, if 
\begin{align}
&\label{eq:0bounded}\tag{Sign} \P^G\text{-a.s., }E^{\geq0}\text{ only contains bounded connected components,}
\end{align}
and $g|_{G \times G}$ is uniformly bounded on the diagonal, see also (1.42) in \cite{MR3492939} for a slightly weaker condition (but see below; our results will imply that this latter condition is in fact unnecessary), 
then
\begin{align}
&\label{eqcouplingintergffszn}\tag{Isom}
	\begin{array}{l}
	\big(\phi_x 1_{x\notin{\mathcal{C}_u}}+\sqrt{\phi_x^2+2\ell_{x,u}}\,  1_{x\in{\mathcal{C}_u}}\big)_{x\in{\tilde{\G}}}\text{ has the same law} \\[0.8em]
	\text{under }{\P}^{I}\otimes\P^G \text{ as }\big(\phi_x+\sqrt{2u}\big)_{x\in{\tilde{\G}}}\text{ under }\P^G,\text{ for all }u\geq0,
	\end{array}
\end{align}
where $\mathcal{C}_u$ denotes the closure of the union of the connected components of those sign clusters $\{x\in{\tilde{\G}}:|\phi_x|>0\}$ that intersect the interlacement set $\I^u.$ In particular, noting that $\ell_{x,u}=0$ if $x \notin \mathcal{C}_u$,  \eqref{eqcouplingintergffszn} is seen to yield \eqref{usualiso} upon taking squares. In practice, the main obstacle to deducing the identity \eqref{eqcouplingintergffszn} is showing that \eqref{eq:0bounded} holds (cf.~the discussion following Theorem~\ref{T:main}).

Our main result investigates the newly introduced capacity observable \eqref{eq:capobservable} and explores the links between this quantity, the value of the critical parameter $\tilde{h}_*$ in \eqref{defh*bou} and the validity of the identity \eqref{eqcouplingintergffszn}. A natural structural property that will appear in this context is the (weak) condition that
\begin{align} 
&\label{capcondition}
\tag{Cap}
\mathrm{cap}(A)=\infty\text{ for all ($d$-)unbounded, closed, connected sets }A\subset \widetilde{\mathcal{G}}
\end{align}
(see \eqref{capconditiondis} for an equivalent formulation in terms of the base graph $\G$ and below \eqref{eq:capobservable} for the definition of $\text{cap}(\cdot)$ in the present context). One can for instance show that \eqref{capcondition} is verified whenever the Green function $g|_{G \times G}$ is uniformly bounded on the diagonal, see Lemma \ref{equivcapcondition} below (cf.~also \eqref{gbounded} for a slightly more general condition). In particular, \eqref{capcondition} holds on any vertex-transitive graph.

We now present a succinct version of our main result. It entails several findings which are discussed in Section \ref{sec:results} in a more comprehensive form. For later reference we introduce the condition
\begin{align}
	&\label{eq:laplacecap}\tag{$\text{Law}_h$}
	\E^G\big[\exp\big(-u\, \mathrm{cap}\big({E}^{\geq h}(x_0)\big)\big) 1_{\phi_{x_0}\geq h}\big]=\P^G\big(\phi_{x_0}\geq \sqrt{2u+h^2}\big)\text{ for all }u\geq0, \,  x_0 \in \widetilde{\mathcal{G}}; 
\end{align}
note that the Laplace transform in \eqref{eq:laplacecap} can be equivalently described in terms of an associated density $\rho_h$, which is explicit, see \eqref{eq:hdensity} and Lemma \ref{equivalenceforthelaw} below. 

\begin{The}
\label{T:main}
	Let $\G$ be a transient weighted graph. 
	Then:
	\begin{enumerate}[label={\arabic*)}]
		\item \label{maint1}
		$\P^G$-a.s., the random variable $\mathrm{cap}(E^{\geq 0}(x_0))$ is finite for all $x_0\in{\tilde{\G}}.$ In particular, the condition \eqref{capcondition} implies \eqref{eq:0bounded} (see Theorem \ref{mainresult} and Corollary \ref{maincor} for details).
		\item\label{maint3}
		The following implications hold true (cf.\ also Fig. \ref{fig1_equiv} below):
		\begin{equation*}
		\begin{tikzcd}[row sep=large]
		\tilde{h}_*\leq0
		\stackrel{\text{Cor. \ref{percoath*}}}{\Longleftrightarrow}  
		\eqref{eq:0bounded}
		\stackrel{\text{Thm. \ref{mainresultcap}}}{\Longrightarrow} 				
		\hyperref[eq:laplacecap]{(\emph{\textnormal{Law}}_0)}
		\stackrel{\text{Thm. \ref{couplingintergff}}}{\Longleftrightarrow} 
\eqref{eqcouplingintergffszn}
		\stackrel{\text{Thm. \ref{couplingintergff}}}{\Longleftrightarrow} 
	\eqref{eq:laplacecap}_{h \geq 0}.
		\end{tikzcd}
		\end{equation*}
		In particular, in view of \eqref{ifhkill<1thenh_*>0}, if $\G$ is a transient weighted graph such that $\h_{\textnormal{kill}}<1$ and \eqref{capcondition} is fulfilled, then $\tilde{h}_*=0$ and the law of $\mathrm{cap}({E}^{\geq h}(x_0))$ is characterized by \eqref{eq:laplacecap}, for $h \geq 0$ (equivalently, \eqref{eqcouplingintergffszn} holds).
		
\item\label{maint4} 

 If \hyperref[eq:laplacecap]{\emph{($\text{Law}_0$)}} holds but \eqref{eq:0bounded} does not hold, then $\tilde{h}_*=\infty$ (see Corollary~\ref{dichotomy} for details).

In particular, in view of \eqref{ifhkill<1thenh_*>0}, if \hyperref[eq:laplacecap]{\emph{($\text{Law}_0$)}} holds and $\h_{\textnormal{kill}}<1,$ then $\tilde{h}_* \in{\{0,\infty\}}$. 
	\end{enumerate}
	\end{The}
To appreciate the strength of Theorem~\ref{T:main}, we highlight one particular consequence, which follows directly from items~\ref{maint1} and~\ref{maint3} above together with Corollary~\ref{equivcapcondition},\ref{equivcapcondition_b} below.

\begin{Cor}[No percolation at criticality] \label{C:noperc} Let $\mathcal{G}$ be a vertex-transitive, massless, transient weighted graph. Then ($\tilde{h}_*=0$ and) the clusters of $E^{\geq 0}$ are $\P$-a.s.~bounded.
\end{Cor}

We further refer to Corollary \ref{Cor:loopsoups} below for interesting consequences of Theorem \ref{T:main} regarding loop soups, and to \cite{DrePreRodCritExp} regarding the (near-)critical picture associated to the (continuous) phase transition exhibited by Corollary~\ref{C:noperc}.

We now elaborate on the results of Theorem \ref{T:main} in due detail and give some ideas concerning their proofs. 
In part \ref{maint1} of Theorem \ref{T:main}, the finiteness of the capacity functional \eqref{eq:capobservable} at height $h=0$ -- which, remarkably, holds without any further assumption on $\mathcal{G}$ -- can loosely be regarded as an indication that the sign clusters of the Gaussian free field on $\tilde{\G}$ do not percolate, at least when measured in terms of capacity, cf.~also \eqref{defh*cap} and Theorem~\ref{mainresult} below. Condition \eqref{capcondition} formalizes this intuition, since it directly implies that closed connected sets have finite capacity if and only if they are bounded. Thus,  if \eqref{capcondition} holds true, so does \eqref{eq:0bounded}, which in turn directly entails $\tilde{h}_*\leq0,$ see \eqref{defh*bou}. The condition \eqref{capcondition} is moreover usually easy to verify, since it depends only on the structure of the graph $\G,$ and not on the Gaussian free field. As alluded to above, the inequality $\tilde{h}_*\leq0$ had previously only been proved on a certain number of graphs with $\kappa\equiv0$, which all verify condition \eqref{capcondition}, namely:
\begin{myitemize}
    \item $\Z^d,$ $d\geq3,$ with unit weights, see Theorem 1 and Proposition 5.5 in \cite{MR3502602}. This proof could actually be easily extended to all amenable, vertex-transitive graphs, and such graphs verify \eqref{capcondition}, see Lemma \ref{equivcapcondition},\ref{equivcapcondition_b}.
    \item The $(d+1)$-regular tree $\mathbb{T}_d$, $d\geq2,$ with unit weights, see Proposition 4.1 in \cite{MR3492939}. It is easy to prove that these graphs verify \eqref{capcondition}, using Lemma \ref{equivcapcondition},\ref{capconditionontrees}, the fact that $e_{K,\mathbb{T}^d}(x) \geq c(d)$ (which holds uniformly over connected finite subsets $K \subset \mathbb{T}_d$ and $x \in \partial K$), along with the isoperimetric bound $|\partial K| \geq c'(d)|K|$  (see for instance \cite{MR3616731}, p.80).
    \item Any tree $
    \mathbb T$ with unit weights such that $\{x\in\mathbb {T}:\,R_x^{\infty}>A\}$ only has bounded components for some $A>0,$ where $R_x^{\infty}$ is the effective resistance between $x$ and infinity for the descendants of $x,$ see Proposition 2.2 in \cite{MR3765885}. These graphs verify \eqref{capcondition} by Lemma \ref{equivcapcondition},\ref{capconditionontrees}.
    \item Any transient graph with controlled weights (see e.g.\ condition $(p_0)$ in \cite{DrePreRod2}), such that the volume of balls have polynomial growth and the Green function decreases polyonomially fast, see Proposition~5.2 in \cite{DrePreRod2}. These graphs verify \eqref{capcondition}, see Lemma~3.2 in~\cite{DrePreRod2}.
\end{myitemize}
Hence, Theorem \ref{T:main} 
subsumes and generalizes all these previous results, and it covers many new cases, such as \textit{all} vertex-transitive graphs, see Lemma \ref{equivcapcondition},\ref{equivcapcondition_b} below. What is more, without assuming that \eqref{capcondition} is fulfilled, it is possible to construct a graph $\G$ such that $\tilde{h}_*\le 0$ fails to hold, see Proposition \ref*{Pre1:h*infinity} in 
\cite{Pre1}. One can also easily find examples of graphs such that \eqref{eq:0bounded} is verified, while \eqref{capcondition} is not, see Remark \ref{R:mainresults1},\ref{signwithoutcap}, or Proposition \ref*{Pre1:Z20counterexample} in \cite{Pre1} for more details. A further, very interesting question is whether there exist examples of graphs $\mathcal{G}$ not satisfying $\hyperref[eq:laplacecap]{(\emph{\textnormal{Law}}_0)}$, or any of the other equivalent conditions appearing in Theorem \ref{T:main},\ref{maint3}.

A stepping stone for the proof of Theorem \ref{T:main},\ref{maint1} (and, as will soon turn out, of Part~\ref{maint3} as well) is the observation that the identity \eqref{eqcouplingintergffszn}, if assumed to hold, implies $\eqref{eq:laplacecap}_{h \geq 0}$, see Proposition~\ref{couplingimplytheorem} and Lemma \ref{isomequivisom'} below. Crucially, this observation can be applied immediately when $\G$ is a finite (transient) graph, for \eqref{eqcouplingintergffszn} is then a direct consequence of the isomorphism between loop soups and the Gaussian free field, see \cite{MR2815763} and \cite{MR3502602}, that we recall in \eqref{eqcouplingloopsgff}. We refer to Lemma \ref{Theforfinite}, proved in Appendix \ref{App:isom} using similar ideas as in the proof of Theorem 8 in \cite{LuSaTa}, for corresponding details. 

Equipped with \eqref{eqcouplingintergffszn}, and thus \eqref{eq:laplacecap}$_{h\geq0},$ on finite transient graphs we then  approximate the Gaussian free field on any infinite transient graph $\G$ by the Gaussian free field on a sequence of finite transient graphs $\G_n$ increasing to $\G$ as $n\to \infty$, see \eqref{eq:def_approx} and Lemma \ref{LemmaapproGFF}. The fact that our setup allows for $0$-boundary conditions (i.e.\ $\bar{\kappa}_x =\infty$ for some $x \in G$) is central for this purpose.
The capacity functional \eqref{eq:capobservable} has certain desirable monotonicity properties under this approximation, see \eqref{capGngeqcapG}, and  Theorem \ref{T:main},\ref{maint1} corresponds to the information that survives in the limit $n \to \infty$ without further assumptions on $\G$.

Let us now comment on Part~\ref{maint3} of Theorem \ref{T:main} and its proof. Figure \ref{fig1_equiv} illustrates  the various implications involved in its statement in a more explicit fashion and will hopefully provide some useful guidance for the reader.
\begin{figure}[h]
		\begin{equation*}
\begin{array}{c}
\tilde{h}_*\leq 0 \,
\stackrel[a)]{\text{Cor.\ \ref{percoath*}}}{\Longleftrightarrow}  \,
\eqref{eq:0bounded}  \,
\stackrel[b)]{\text{Thm.\ \ref{mainresultcap}}}{\Longrightarrow}  \,
\eqref{eq:laplacecap}_{h \geq 0}
\\ \\ \text{ and }\\ \\
\hyperref[eq:laplacecap]{({\text{Law}}_0)}  \,
\stackrel[c)]{\text{Lem.\ \ref{couplingisalwaystrue}}}
\Longrightarrow  \,
\eqref{eqcouplingintergffszn}  \,
\stackrel[d)]{\text{Prop.\ \ref{couplingimplytheorem}}}
\Longrightarrow  \,
\eqref{eq:laplacecap}_{h \geq 0}.
\end{array}
\end{equation*}
		\vspace{-0.3cm}
\caption{The detailed chain of implications constituting Theorem \ref{T:main},\ref{maint3}. The implications in the second line immediately yield the equivalence of $\hyperref[eq:laplacecap]{({\text{Law}}_0)}$, \eqref{eqcouplingintergffszn}
 and $\eqref{eq:laplacecap}_{h \geq 0}.
$}
\label{fig1_equiv}
\end{figure}

The equivalence a) in Figure \ref{fig1_equiv} entails that 
if $\tilde{h}_*=0,$ then the level sets of the GFF never percolate at the critical point $h=0,$ 
 even if \eqref{capcondition} (which imply \eqref{eq:0bounded}) is not verified. We comment on its proof at the very end of this discussion. Implication~b) represents the desired improvement over the argument delineated above yielding Theorem~\ref{T:main},\ref{maint1}, by which the full information \eqref{eq:laplacecap}$_{h\geq0}$ survives in the limit as $n \to \infty$ under the assumption that the sign clusters of $\varphi$ are bounded (which holds e.g.\ under condition \eqref{capcondition}). In fact, when \eqref{capcondition} is satisfied, we also provide an explicit formula for the law of the capacity of clusters above negative levels,  see Theorem~\ref{mainresultcap} for further details; see also Remark \ref{R:isomisom},4), Lemma \ref{h-hsamelaw} and Remark \ref{approimplylawcap},\ref{stronglawnegative} regarding the (related) symmetry properties relating compact clusters in $E^{\geq h}$ and $E^{\geq-h}$, for arbitrary $h>0$.

The exact formula \eqref{eq:laplacecap}$_{h\geq0}$ describing the law of the capacity functional \eqref{eq:capobservable} is of course instrumental and witnesses a certain degree of integrability of the model $\{ E^{\geq h }: h \in \R\}$. For instance, one can immediately deduce from it (see \eqref{eq:hdensity}) that the capacity of critical clusters has heavy tails satisfying
\begin{equation}
\label{eq:cap_tails}
\P^G\big(  \mathrm{cap}\big({E}^{\geq 0}(x_0)\big) \geq r \big) \sim \big(\pi^2g(x_0,x_0)r\big)^{-1/2}, \text{ as }r \to \infty.
\end{equation}
Further to \eqref{eq:cap_tails}, one can use \eqref{eq:laplacecap}$_{h\geq0}$ to directly deduce bounds on 
various quantities of interest related to the (near-)critical behavior for the percolation of $\{ E^{\geq h }: h \in \R\}$, see \cite{DrePreRodCritExp}. The approach using differential formulas developed therein actually leads to an independent proof of the implication b), along with extended results valid on any transient graph $\G$, see Theorem 1.1 in \cite{DrePreRodCritExp}. Incidentally, an explicit formula for the probability of the event $\{x\longleftrightarrow y$ in $E^{\geq0}\}$ has also been obtained in Proposition 5.2 of \cite{MR3502602}, and  was a key ingredient for all previous proofs of the inequality $\tilde{h}_*\leq0.$

We now turn to the equivalences c) and d) in the second line of Figure \ref{fig1_equiv}. The direct (i.e.~right) implications appearing there already imply the equivalences. The direct implication in d) is another application of our initial observation, Proposition~\ref{couplingimplytheorem}, applied above in the context of Theorem \ref{T:main},\ref{maint1} for finite graphs only, but remaining valid in infinite volume.

Remarkably, the direct implication in c) asserts that it is sufficient to know that the law of the capacity of the sign clusters is given by \hyperref[eq:laplacecap]{($\text{Law}_0$)} in order to deduce the strong version \eqref{eqcouplingintergffszn} of the isomorphism theorem. In particular, together with b), this implies that \eqref{eqcouplingintergffszn} holds whenever  \eqref{eq:0bounded} is verified, which generalizes Theorem 2.4 of \cite{MR3492939} that required stronger assumptions, cf.~the above discussion leading to \eqref{eqcouplingintergffszn}. 

Extending the setting in which the identity \eqref{eqcouplingintergffszn} is valid is also interesting as this relation has already been useful in \cite{MR3492939} and \cite{MR3765885} to compare the critical parameter for the percolation of random interlacements and the Gaussian free field on discrete trees, and in \cite{DrePreRod2} to prove strong percolation for the level sets of the discrete Gaussian free field at a positive level on a large class graphs, for instance $\Z^d,$ $d\geq3,$ or various fractal graphs. It is not always easy to check that the conditions (1.32) and (1.34), or (1.42), of Theorem 2.4 in \cite{MR3492939} are exactly verified, see the proof of Corollary 5.3 in \cite{DrePreRod2} which sparked our interest, and it can thus be interesting to replace them by the weaker condition \eqref{capcondition}, which is easier to verify.

The proof of c) requires deriving a full-fledged isomorphism theorem relating random interlacements and the Gaussian free field on an adequate class of graphs, assuming the identity \hyperref[eq:laplacecap]{($\text{Law}_0$)} alone.
In order to prove \eqref{eqcouplingintergffszn}, we employ an approximation scheme, starting from a finite-volume setup. The scheme is similar in spirit to the previously used approximation for $\varphi$, but more involved, as it requires approximating random interlacements on infinite graphs by random interlacements on finite graphs, see Lemma \ref{approximationprop}. Combining the approximations for the free field and the interlacement process, we then obtain \eqref{eqcouplingintergffszn} if 
\hyperref[eq:laplacecap]{($\text{Law}_0$)} is fulfilled, see Lemma~\ref{couplingisalwaystrue}.

Moreover, our proof of \eqref{eqcouplingintergffszn}, which relies on taking a suitable limit rather than proceeding directly in infinite volume and using the Markov property as in \cite{MR3492939}, immediately lets us derive a signed version of the isomorphism for random interlacements on discrete graphs, taking advantage of the equivalent discrete isomorphism for the loop soup, \eqref{eqcouplingloopsgffdis}. As a by-product of the proof, we thus obtain a version of the isomorphism \eqref{eqcouplingintergffszn} for the discrete graph $\G$ in Theorem~\ref{couplingintergff}, see \eqref{eqcouplingintergffdis}, similar to the version of the second Ray-Knight theorem from Theorem~8 in \cite{LuSaTa}.

Finally, the isomorphism \eqref{eqcouplingintergffszn} has another interesting consequence, stated in Theorem~\ref{T:main},\ref{maint4} and Corollary~\ref{dichotomy}: if \hyperref[eq:laplacecap]{{($\text{Law}_0$)}} holds but \eqref{eq:0bounded} does not hold, then $\tilde{h}_*=\infty.$ This can be regarded as a partial converse to the implication \eqref{eq:0bounded} $\Longrightarrow$ \hyperref[eq:laplacecap]{{($\text{Law}_0$)}} from part~\ref{maint3}, which leads to a dichotomy for the value of $\tilde{h}_*$ in case $\h_{\textnormal{kill}}<1$. In particular, if $\G$ is a graph such that $\tilde{h}_*\leq0,$ then $E^{\geq h}$ is $\P^G$-a.s.\ bounded for all $h>0,$ and thus \eqref{eq:laplacecap} holds for all $h>0,$ see Theorem~\ref{mainresultcap}. Taking the limit as $h\searrow0,$ one can then prove that \hyperref[eq:laplacecap]{($\text{Law}_0$)}, and thus \eqref{eqcouplingintergffszn}, hold. Since $\tilde{h}_*\neq\infty,$ this means that \eqref{eq:0bounded} must hold, and thus we also obtain Theorem \ref{T:main},\ref{maint3},a) (see Figure 1). 

\medskip

We now explain how this article is organized. Section \ref{s:usefulresults} recalls the main objects of interest, the diffusion $X,$ the Gaussian free field, and random interlacements on the cable system in the present (broad) setup. It also supplies suitable notions of equilibrium measure and capacity on $\tilde{\G}$, see Lemma \ref{GA}, \eqref{defequilibriumcable} and \eqref{defcap}. 

Section \ref{sec:results} contains the detailed versions of all our findings, which together imply Theorem \ref{T:main}, and that we  prove in the rest of the article. The central results are the three Theorems \ref{mainresult}, \ref{mainresultcap} and \ref{couplingintergff}, along with their respective corollaries.

Section \ref{sec:prep} gathers various key preliminary results, notably Proposition~\ref{couplingimplytheorem}, which derives  \eqref{eq:laplacecap}$_{h\geq0}$ as a consequence of \eqref{eqcouplingintergffszn} (or more precisely, an equivalent but more handy formulation \eqref{eqcouplingintergff} introduced in Section \ref{sec:results}). It also contains the approximation scheme for $\varphi$, see Lemma \ref{LemmaapproGFF}, as well as the isomorphism \eqref{eqcouplingintergffszn} on finite graphs, see Lemma \ref{Theforfinite}. These results are the ingredients of various arguments in the sequel. 

First, Section \ref{sec:mainviainter} is devoted to the proof of Theorems \ref{mainresult} and \ref{mainresultcap}, which roughly correspond to Theorem \ref{T:main},\ref{maint1}, 
and \ref{maint3},b) in Figure 1, but contain more detailed results. Their proof quickly follows from the preparatory work done in Section \ref{sec:prep}.

Section \ref{sec:iso} is then concerned with the proof of the isomorphism between random interlacements and the Gaussian free field \eqref{eqcouplingintergffszn} under the condition \hyperref[eq:laplacecap]{($\text{Law}_0$)}, and to its consequences, Corollaries \ref{dichotomy} and \ref{percoath*}. At the technical level, an important role is played by the approximation of random interlacements on a graph $\G,$ by random interlacements on a sequence of graphs increasing to $\G,$ see Lemmas \ref{limitKn} and \ref{approximationprop}. Some concluding remarks and open questions are gathered at the end of that section.

Throughout the article, we will sometimes add $\tilde{\G}$ as a subscript to the notation to stress the underlying graph $\G$ that we consider. For the reader's orientation, we note that the conditions \eqref{eq:0bounded}, \eqref{eq:laplacecap} and \eqref{eqcouplingintergffszn} are all introduced above Theorem \ref{T:main}, and that the condition \eqref{eqcouplingintergff} is introduced above Theorem \ref{couplingintergff}.

\medskip

\textit{Acknowledgements.} Part of this work was carried out while the research of P-FR was supported by the ERC-Grant CriBLaM. AD and AP thank I.H.E.S.\ and Hugo Duminil-Copin for their hospitality at these early stages. The research of AD is supported by the Deutsche Forschungsgemeinschaft (DFG) grant DR 1096/1-1, that of AP by the Engineering and Physical Sciences Research Council (EPSRC) grant EP/R022615/1 and Isaac Newton Trust (INT) grant G101121. We thank Tom Hutchcroft for his comments about condition \eqref{capcondition}, which partially stimulated Appendix~\ref{subsec:capandkappa=0}. We thank A.-S.\ Sznitman for pointing out the short proof of \eqref{eq:capinfinity} given in \eqref{eq:theta_short}, as well as an anonymous referee.

\section{Preliminaries and useful results}
\label{s:usefulresults}
We return to the framework described around \eqref{deflambdarho}, consisting of a transient weighted graph $ \G$, the induced triplet $(G,\lambda,\kappa)$ satisfying $\kappa_x<\infty$ for all $x\in{{G}}$ and the associated cable system $\widetilde{\mathcal{G}}$. We now define the various objects attached to this setup. 
We first sketch a construction of the canonical diffusion $X$ on $\widetilde{\mathcal{G}}$ and of its trace  on suitable subsets $F$ of $G$ from the associated Dirichlet form in Section~\ref{subsec:diff}. In Section \ref{subsec:pottheory} we introduce several aspects of potential theory on $\widetilde{\mathcal{G}}$ in this general framework, which can be conveniently defined probabilistically by `enhancements', exploiting instances of network equivalence on the base graph $\mathcal{G}$, see Lemma~\ref{GA} below. 
We then briefly discuss the cables $I_x$ (Section \ref{S:I_x}) and their role in taking suitable graph limits, recall the Gaussian free field $\varphi$ and its Markovian decomposition (Section \ref{subsec:GFF}), and supply the definition of random interlacements 
in the present context (Section \ref{subsec:RI}). 

\medskip


Recall the definition of the cable system $\tilde{\G}$: first, each edge $e=\{x,y\}\in E$ is replaced by an open interval $I_e,$ isometric to $(0,\rho_{x,y}),$ see \eqref{deflambdarho}. In addition, an open interval $I_x$ of length $\rho_x(=\frac{1}{2\kappa_x})$ (possibly unbounded) is attached to each vertex $x$ of ${G}$. The cable system $\tilde{\G}$ is then obtained by glueing together the intervals $I_e,$ $e\in{E},$ to $G$ through their respective endpoints, and by glueing one endpoint of $I_x,$ $x\in{G},$ to $x.$ Note that ${G}$ can be naturally viewed as a subset of $\tilde{\G}.$ The elements of ${G}$ will still be called \textit{vertices} and the intervals $I_e,$ $e\in{E},$ and $I_x,$ $x\in{{G}},$ will be referred to as the \textit{edges} of $\tilde{\G}.$ 

The canonical distance on each ${I}_e,$ $e\in{E},$ and ${I}_x,$ $x\in{G},$ is denoted by $\rho_{\tilde{\G}}(\cdot, \cdot).$ Note that $\rho_{\tilde{\G}}(x,y)$ is only defined if $x$ and $y$ are on the same edge. In a slight abuse of notation, for any edge $e=\{x,y\}\in{E}$ and any $t\in{[0,\rho_{x,y}}],$ we denote by $x+t\cdot I_e=y+(\rho_{x,y}-t)\cdot I_e$ the point of ${I}_e$ at ($\rho_{\tilde{\G}}$-)distance $t$ from $x,$ and for any vertex $x\in{G}$ and $t\in{[0,\rho_{x}}),$ by $x+t\cdot I_x$ the point of ${I}_x$ at distance $t$ from $x.$ We also consider the distance $d$ on $\tilde{\G},$ cf.\ above \eqref{Ehx0}, which is such that $d(x,y),$ $x,y\in{\tilde{\G}}$, is the minimal length of a continuous path between $x$ and $y,$ when changing the length of each $I_e,$ $e\in{E\cup G}$ from $\rho_e$ to $1.$ In particular, the restriction of $d(\cdot,\cdot) $ to $G\times G$ is just the graph distance $d_{\G}$ on $\G.$ We consider $(\tilde{\G},d)$ as a metric space, and for $A\subset \tilde{\G}$ we define $\partial A$ as the boundary of $A$ in $\tilde{\G}$ for $d$. Finally throughout the article, we say that a set $K\subset\tilde{\G}$ is compact if it is compact for the distance $d.$

\subsection{The canonical diffusion on the cable system}
	\label{subsec:diff}
We define the set of forward trajectories $W^+_{\tilde{\G}}$ as the set of functions $w^+:[0,\infty)\rightarrow\tilde{\G}\cup\{\Delta\},$ where $\Delta$ is a cemetery point (not in $\tilde{\G}$), for which there exists $\zeta\in{[0,\infty]}$ such that $w^+_{|{[0,\zeta)}}\in{C([0,\zeta),\tilde{\G})}$ and, when $\zeta<\infty,$ $w^+(t)=\Delta$ for all $t\geq\zeta.$ For each $t\geq0$ we denote by $X_t$ the projection at time $t,$ i.e.\ $X_t(w^+)=w^+(t)$ for all $w^+\in{W^+_{\tilde{\G}}},$ and by $\mathcal{W}^+_{\tilde{\G}}$ the $\sigma$-algebra on $W^+_{\tilde{\G}}$ generated by $X_t,$ $t\geq0.$ By $m$ we denote the Lebesgue measure on $\tilde{\G},$ which can be informally described as the sum of the Lebesgue measures on each $I_e,$ $e\in{E},$ and $I_x,$ $x\in{G},$ with the normalization $m(I_e)=\rho_e$ and $m(I_x)= \rho_x$ (with, say, mass $1$ associated to each sub-interval of Euclidean length $1$). We proceed to define a diffusion on $\tilde{\G},$ which we will characterize through its associated Dirichlet form. In order to define the latter, introduce for measurable $f:\tilde{\G}\rightarrow\R$,  
\begin{equation}
\label{eq:quadraticform}
    (f,f)_{m}\stackrel{\text{def.}}{=}\sum_{e\in{E\cup{G}}}\int_{I_e}f^2\, \mathrm{d}m_{|I_e},
\end{equation}
the corresponding Hilbert space $L^2(\tilde{\G},m)\stackrel{\text{def.}}{=}\{f:\tilde{\G}\rightarrow\R \text{ measurable}; \,(f,f)_{m}<\infty\}$ (modulo the usual equivalence relation) and $(f,g)_{m}$ the associated quadratic form on $L^2(\tilde{\G},m)$ obtained via polarization. Let $C_0(\tilde{\G})$ be the closure for the $\|\cdot\|_\infty$-norm of the set of continuous functions with compact support on $\tilde{\G}$ and let $D(\tilde{\G},m)\subset L^2(\tilde{\G},m)$ be the space of functions $f\in{C_0(\tilde{\G})}$ such that $f_{|I_e}\in{W^{1,2}(I_e,m_{|I_e}})$ for all $e\in{E\cup{G}}$ and 
\begin{equation*}
    \sum_{e\in E\cup{G}}\|f_{|I_e}\|_{W^{1,2}(I_e,m_{|I_e})}^2<\infty,
\end{equation*}
where $W^{1,2}(I_e,m_{|I_e})$ denotes the respective Sobolev space on $I_e.$ We now define the Dirichlet form on $L^2(\tilde{\G},m)$ (in which $D({\tilde{\G}},m)$ is densely embedded), 
\begin{equation}
\label{Dirichlet}
    \mathcal{E}_{\tilde{\G}}(f,g)\stackrel{\text{def.}}{=}\frac12(f',g')_m\text{ for all }f,g\in{D({\tilde{\G}},m)}.
\end{equation}
By Theorem 7.2.2.\ in \cite{MR2778606}, one associates to each $x\in{\tilde{\G}}$ an $m$-symmetric diffusion starting in $x$ with state space $\tilde{\G}\cup\{\Delta \}$ to the Dirichlet form $\mathcal{E}_{\tilde{\G}}.$ We denote by $P_x\, (=P_x^{\tilde{\G}})$ its law on $(W_{\tilde{\G}}^+,\mathcal{W}_{\tilde{\G}}^+)$ and also define, for any non-negative measure $\mu$ on $\tilde{\G}$ with countable support $\text{supp}(\mu)$, the measures
\begin{equation}
\label{probmeasure}
P_\mu\stackrel{\text{def.}}{=}\sum_{x\in{\text{supp}(\mu)}}\mu_xP_x.
\end{equation}  
Note that $\zeta = \inf \{ t \geq 0 : X_t =\Delta\}$ is either $\infty,$ or the first time $X$ blows up (i.e., $X$ escapes all $d$-bounded sets) 
or gets killed (i.e., exits $ \tilde{\G}$ through some $I_x$ with $\kappa_x>0$). Informally, one can obtain a diffusion with law $P_x$ as follows: first, one runs a Brownian motion starting at $x$ on $I_e,$ with $x\in{I_e},$  $e\in{E\cup{G}},$ until a vertex $y$ is reached. Then one chooses uniformly at random an edge or vertex $v$ among $\{ y\} \cup \{\{y,z\}: z\sim y\}$ and runs a Brownian excursion on $I_v$ until a vertex is reached; this procedure is iterated until either the process blows up or the open end of the interval $I_x$ is reached for some $x\in{{G}},$ in which case the process is killed at that time. We refer to Section 2 of \cite{DrePreRod} or \cite{MR3502602} for a more formal description of this construction on $\Z^d,$ $d\geq3.$ 

We now briefly review how to take \textit{traces} of the process $X$ on suitable subsets $F$ of $\tilde{\G}$. One can show, analogously to Section 2 of \cite{MR3502602}, that the process $X$ under $P^{\tilde{\G}}_x$ allows for a space-time continuous family of local times $(\ell_y(t))_{y\in{\tilde{\G}},t\geq0}.$ Therefore, using that $P^{\tilde{\G}}_x$ lives on the canonical space $(W_{\tilde{\G}}^+,\mathcal{W}_{\tilde{\G}}^+),$ for all sets $F \subset \tilde\G$ of the form $F = \bigcup_{e \in F_1} \overline{I}_e \cup \bigcup_{x \in F_2} \{x \}$,
where $F_1\subset E\cup{G}$ and $F_2\subset{G}$ are arbitrary, we can define the time change
\begin{equation*}
    \tau_t^F\stackrel{\text{def.}}{=}\inf\Big\{s>0:\,\int_0^s 1_{\{X_u\in{  \bigcup_{e \in F_1} I_e}\}}\diff u+\sum_{y\in{F_2}}\ell_y(s)>t\Big\}\text{ for all }t\geq0\text{ and }w^+\in{W^+_{\tilde{\G}}}.
\end{equation*}
Here, we use the convention $\inf\varnothing=\zeta$ 
and denote the trace of $X$ on ${F}$ by $X^F=(X_{\tau_t^F})_{t\geq0},$ with the convention $X_\infty = \Delta$, which corresponds to a time changed process with respect to a positive continuous additive functional (PCAF), see (A.2.36) and below in \cite{MR2778606} for instance. As a first application of this definition, letting 
\begin{equation}
    \label{traceonG}
    Z\stackrel{\text{def.}}{=} X^{G} \text{ (the trace of $X$ on $G$)}\quad 
    \end{equation} 
it follows from Theorem 6.2.1.\ in \cite{MR2778606} that for all $x\in{{G}}$ the law of $Z$ under $P_x^{\tilde{\G}}$ is that of the continuous time Markov chain that jumps from $x \in G$ to $y \in G$ at rate $\lambda_{x,y}$ and is killed at rate $\kappa_x.$ Furthermore, the local times $(\ell_y(\zeta))_{y\in G}$ of $X$ after being killed have the same law under $P_x^{\tilde{\G}}$ as the total occupation times of that jump process (after being killed), see for instance (1.97) and (2.80) in \cite{MR2932978}.  We also denote by $(\hat{Z}_n)_{n\in\N}$ the discrete time skeleton of $Z,$ i.e.\ the sequence of elements of $G$ visited by the process $Z$, with the convention that $\hat{Z}_n = \Delta$ for all large enough $n$ if $Z$ gets killed.

\subsection{Elements of potential theory on $\widetilde{\G}$}

Our next goal is to supply workable notions of equilibrium measure and capacity on $\widetilde{\G}$, for arbitrary closed (and in particular compact) subsets of $\widetilde{\G}$, as necessary in order to investigate observables like $\mathrm{cap}(E^{\geq h}(x_0))$ (cf.\ Theorem \ref{T:main}). We first define the Green function of an open set $U\subset\tilde\G$ by
\begin{equation}
\label{Greendef}
g_{U}(x,y)=E_x[\ell_y(T_U)]\text{ for all } x,y\in{\tilde{\G}},
\end{equation}
where $E_x$ denotes expectation with respect to $P_x= P^{\tilde{\G}}_x$ and $T_U=\inf\{t\geq0:X_t\notin{U}\}$ is the first exit time of $U,$ with the convention $\inf\varnothing=\zeta.$ We simply write $g= g_{\tilde{\G}}$ for the usual Green function on $\tilde{\G}.$  

\label{subsec:pottheory}
We now introduce the notions of equilibrium measure and capacity on $\tilde{\G}$ by `enhancements', see Lemma \ref{GA} below. This will allow to directly reformulate the equilibrium problem in a discrete setup and to thereby import the respective standard versions of these notions on transient graphs, see \eqref{defequilibriumcable}, \eqref{defcap} and \eqref{defcapinfinity} below. In particular, this approach immediately provides several useful identities, e.g.\ relating exit distributions for the diffusion $X$ with the corresponding equilibrium measure, cf.\ \eqref{exitequi} and \eqref{consistencyequilibrium}.  

On the (transient) graph $(G,\lambda, \kappa)$ associated to $\mathcal{G}$, for all finite $A\subset G$ the equilibrium measure and capacity of $A$ are defined by 
\begin{equation}
\label{defeAcap}
    e_{A,\G}(x)\stackrel{\mathrm{def.}}{=}\lambda_xP_{x}(\tilde{H}_{A}(\hat{Z})=\infty) 1_{A}(x) \text{ for all }x\in{{G}}, \quad \text{and} \quad\mathrm{{\rm cap}}_{\mathcal{G}}(A)\stackrel{\mathrm{def.}}{=}\sum_{x\in{A}}e_{A,\G}(x),
\end{equation}
where $\tilde{H}_{A}(\hat{Z}) \stackrel{\mathrm{def.}}{=}\inf\{n\geq1,\ \hat{Z}_n\in{A}\},$ with $\inf\varnothing=\infty,$ is the first return time to $A$ for the discrete time random walk $\hat{Z}$ on $\G$, cf.\ below \eqref{traceonG}. The following observation is key.

\begin{Lemme}[Enhancements]
\label{GA}
For all countable sets $A\subset\tilde{\G}$ without accumulation point in $\tilde{\G},$ there exists a unique graph $\G^{A}=(G^A,\lambda^A,\kappa^A)$ with vertex set $G^A=A\cup G,$ such that
\begin{align}
&\text{(with a slight abuse of notation), $\tilde{\G}$ is a subset of $\tilde{\G}^A$, the cable system of $\G^{A}$;} \label{eq:GAsubsetG}\\
& \text{for all }x\in{{G}^A,}\text{ the laws of the traces } X^{G^A}=(X_{\tau^{{G}^{A}}_t})_{t\geq0}\text{ 
under }P^{\tilde{\G}}_x\text{ and }P^{\tilde{\G}^{A}}_x coincide; \label{lawGA}
\end{align}
\end{Lemme}

\begin{proof}
We first introduce the weights $\lambda^A$ and the killing measure $\kappa^A$. For each $e=\{x_0,x_1 \}\in E$, let $A\cap I_e=\{z_1(e),\dots, z_{n-1}(e)\}$, where $n = n(e) \geq 1$ is such that $n-1=|A\cap I_e|$ and the $z_{k}(e)$'s are labeled by order of appearance as one traverses the (open) edge $I_e$ from, say, $x_0$ to $x_1$ (the underlying choice of orientation of $e$ will not affect the definition of $\lambda^A, \kappa^A$ in \eqref{eq:enh1} below). For later convenience, we set $z_0(e)=x_0$ and $z_n(e)=x_1$, and drop the argument $e$ in the sequel whenever no risk of confusion arises. Similarly, for $x \in G$, we enumerate $A\cap I_x=\{z_1(x),\dots, z_{n-1}(x)\}$ (with $n=n(x)\in{\N\cup\{\infty\}}$ such that $n-1=|A\cap I_{x}|$ if $|A\cap I_{x}|<\infty,$ and $n=\infty$ otherwise) according to increasing distance from $x$, and set $z_0(x)=x$. We then define, for $z, z' \in G^A = G\cup A$, 
\begin{equation}
\label{eq:enh1}
\begin{split}
\lambda_{z,z'}^{A}&=\begin{cases}
\frac{1}{2\rho_{\tilde{\G}}(z,z')},& \text{if } \{z,z'\}=\{z_{k-1}(v),z_k(v)\} \text{} \text{ for some }v \in E\cup G \text{ and } k \geq 1,\\
0,&\text{otherwise,}
\end{cases}\\
\kappa_z^{A}&=\begin{cases}
\frac{\kappa_{x}}{1-2\kappa_{x}\rho_{\tilde{\G}}(x,z)},&\text{if $x= z_{n-1}(x)$ for some $x \in G$ (with $n=n(x)<\infty$)},
\\0,&\text{otherwise}.
\end{cases}
\end{split}
\end{equation}
Thus, each edge $e\in E$ is replaced by a linear chain of $n=n(e)$ edges $\{ z_{k-1}, z_k \}$, $1\leq k \leq n$, with weights $\lambda_{ z_{k-1}, z_k}^{A}$, and similarly a chain of $n(x)-1$ edges is attached to each $x \in G$, with killing $\kappa_{ z_{n-1}(x)}^{A}$ at its `dangling' end. By \eqref{eq:enh1} and \eqref{deflambdarho}, for all $e=\{ x_0,x_1\} \in E$ and $x \in G$,
\begin{equation}
\label{eq:enh1.1}
\begin{split}
&\sum_{k=1}^{n(e)} \rho^A_{z_{k-1},z_k}= \sum_{k=1}^{n(e)} \rho_{\tilde{\G}}(z_{k-1},z_k)  =\rho_{\tilde{\G}}(x_0,x_1)=\rho_{x_0,x_1},\\
&\sum_{k=1}^{n(x)} \rho^A_{z_{k-1},z_k} + \frac1{2\kappa_{ z_{n(x)-1}}^{A}} =  \sum_{k=1}^{n(x)} \rho_{\tilde{\G}}(z_{k-1},z_k)  + \frac1{2\kappa_x}- \rho_{\tilde{\G}}(x,z_{n(x)-1})=\rho_x, \text{ if $n(x)<\infty$,}\\
&\sum_{k=1}^{\infty} \rho^A_{z_{k-1},z_k} =  \sum_{k=1}^{\infty} \rho_{\tilde{\G}}(z_{k-1},z_k)=\rho_x, \text{ if $n(x)=\infty.$}
\end{split}
\end{equation}
Therefore, $\tilde{\G}$ can be identified with the set $\tilde{\G}^A \setminus I$, where $\tilde{\G}^A$ is the cable system associated to $(G^A,\lambda^A,\kappa^A)$ and $I = I_1\cup I_2\cup I_3$, where 
\[
I_1= \bigcup_{e \in E} \bigcup_{ k=1}^{n(e)-1} I_{z_k(e)}, \quad I_2= \bigcup_{x \in G,n(x)<\infty} \bigcup_{ k=1}^{n(x)-2} I_{z_k(x)}
\quad \text{ and } \quad I_3 =\bigcup_{x \in G,n(x)=\infty} \bigcup_{ k=1}^{\infty} I_{z_k(x)}.
\]
By a similar reasoning as detailed below around \eqref{dirge}, it then follows that for all $x\in{\tilde{\G}}$ (viewed as a subset of $\tilde{\G}^A$), the law of the trace of $X$ on $\tilde{\G}$ under $P^{\tilde{\G}^{A}}_x$ is $P^{\tilde{\G}}_x.$ In view of \eqref{traceonG}, the claim \eqref{lawGA} then follows.  
\end{proof}

By slightly adapting the above arguments, one defines the graph $(\overline{G}^M,\bar{\lambda}^M,\bar{\kappa}^M)$ alluded to at the beginning of Section~\ref{sec:intro}, see above \eqref{deflambdarho}, as follows. Given $\G=(\overline{G},\bar{\lambda},\bar{\kappa})$, possibly with $\bar{\kappa}_x =\infty$ for some $x \in \overline{G}$, let
\begin{equation}
\label{eq:0bcreduction1new}
M\stackrel{\text{def.}}{=}\{ a : \text{ midpoint of $I_e$ for some $e \in E_{\bar{\kappa}}$}\}
\end{equation}
where $   E_{\bar{\kappa}}=\big \{\{x,y\}:\,x,y \in \overline{G},\,\bar{\lambda}_{x,y}>0,\bar{\kappa}_x=\infty\text{ and }\bar{\kappa}_y<\infty \big\}$ and
$I_e$ is an interval isomorphic to the open interval $(0,1/(2\bar{\lambda}_{x,y}))$ glued at $0$ to $y$, with boundary $\{x,y\}$.
 Now, by a small extension of Lemma \ref{GA}, one constructs from $\G=(\overline{G},\bar{\lambda},\bar{\kappa})$ the graph 
\begin{equation}
\label{eq:defGfinite}
(G,\lambda,\kappa) \stackrel{\text{def.}}{=}(\overline{G}^M,\overline{\lambda}^M,\overline{\kappa}^M)\text{ with }\overline{G}^M=\{ x \in \overline{G}: \bar{\kappa}_x < \infty\} \cup M\text{ and }M\text{ as in }\eqref{eq:0bcreduction1new},
\end{equation}
by treating $I_e$ for $e=\{ x,y\} \in E_{\bar{\kappa}}$ with $\bar{\kappa}_y <\infty$ in the same manner as $I_y$ in \eqref{eq:enh1} (whence $\lambda_{y,a}=\bar{\lambda}_{y,a}^M=2\bar{\lambda}_{y,x},$ $\kappa_y=\bar{\kappa}^M_y=0$ and $\kappa_a=\bar{\kappa}_a^M = 2\bar{\lambda}_{y,x}$ for $a\in M$ the midpoint of $I_{x,y}$), and keeping the same weights and killing measures for the other vertices. Plainly, $(\overline{G}^M,\overline{\lambda}^M,\overline{\kappa}^M)$ satisfies $\overline{\kappa}^M< \infty$. Similarly as below \eqref{traceonG}, it follows from Theorem 6.2.1.\ in \cite{MR2778606} that the law of the trace of $X$ (under
$P_{\cdot}^{\tilde{\G}}$) on $\{x\in{\overline{G}}:\,\bar{\kappa}_x<\infty\}$ is that of the continuous time Markov chain on $\overline{G}$ that jumps from $x \in \overline{G}$ to $y \in \overline{G}$ at rate $\bar{\lambda}_{x,y}$ and is killed at rate $\bar{\kappa}_x,$ hence justifying our choice of $(G,\lambda,\kappa)$ as in \eqref{eq:defGfinite} to define the cable system $\tilde{\G}.$ Note also that $(G,\lambda,\kappa)=\G$ when $\bar{\kappa}<\infty$ since $E_{\bar{\kappa}}=\varnothing$ in that case.

The following remark turns out handy in a couple of instances in this article.

\begin{Rk}[Generating any given cable system from a graph without killing] \label{R:nomassconversion} 
As an application of Lemma \ref{GA}, given $(G,\lambda, \kappa)$ and the corresponding cable system $\tilde{\G}$, one can naturally associate $\tilde{\G}$ to a triplet $(G', \lambda', \kappa')$ with $\kappa' \equiv 0$. To do so, one considers, for each $I_x$ with $\kappa_x \in (0,\infty)$ a sequence $z_n(x)$, $n \geq 0,$ converging to the open end of $I_x$ (note that such a sequence does not have an accumulation point in $\tilde{\G}$). Then, with $A= \{ z_n(x): n \geq 0, \, x\in G \text{ s.t. }\kappa_x \in (0,\infty)\}$, one defines $G'=G^A$ and $\lambda'=\lambda^A$ as given by Lemma~\ref{GA} (note that $\kappa^A \equiv 0$ by \eqref{eq:enh1}). By \eqref{eq:GAsubsetG}, one has that $\tilde{\G}\subset \tilde{\G}^A$ and $\tilde{\G}$ is in fact obtained from $\tilde{\G}^A$ by removing all (unbounded) cables $I_x$, $x\in A$. In particular, combining this observation with the isomorphism \cite{MR2892408}, which holds on $(G',\lambda')$, one readily infers that \eqref{usualiso} holds for $\tilde{\G}$.
\end{Rk}

We now extend the definition of the equilibrium measure from \eqref{defeAcap} to the cable graph setting. When $K$ is a compact subset of $\tilde{\G},$ we define its \textit{exterior} boundary
\begin{equation}
\label{defpartialext}
    \hat{\partial}K=\left\{x\in{ K}:\,P_x\left(X_{L_K}=x,L_K>0\right)>0\right\},
\end{equation} 
where $L_K=\sup\{t>0:X_t\in{K}\}$ is the last exit time of $K,$ with the convention $\sup\varnothing=0.$ Note that $ \hat{\partial}K$ is finite since $K$ is bounded and $I_e$ contains at most two points of $ \hat{\partial}K$ for all $e\in{E\cup G}.$ Consider now any sets $K,\hat{K},A\subset \tilde{\G}$ such that
\begin{equation}
\label{defKhatA}
\text{$K$ is compact, $\hat{K}$ finite, $A$ has no accumulation point and }\hat{\partial} K\subset \hat{K}\subset (K\cap G^A).
\end{equation}
 For all $x,y\in{A},$ by \eqref{lawGA} as well as (1.56) in \cite{MR2932978} (and its straightforward adaptation to infinite transient weighted graphs; this also applies to subsequent references to \cite{MR2932978}) applied to the graph $\G^{A},$  noting that $L_{\hat{K}}=L_K$ a.s.\ and $\{L_{\hat{K}} >0 , X_{L_{\hat{K}}}=x\}=\{\overline{L}_{\hat{K},A} >0 , X^{G^A}_{\overline{L}_{\hat{K},A}^-}=x\}$ where $\overline{L}_{\hat{K},A}$ is the last exit time of $\hat{K}$ for $X^{G^A},$ the trace of $X$ (under $P_x^{\tilde{\G}}$) on $G^A,$ and $X^{G^A}_{\overline{L}_{\hat{K},A}^-}$ is the last vertex of $\hat{K}$ visited by $X^{G^A}$ before time $\overline{L}_{\hat{K},A},$
\begin{equation}
\label{exitequi1}
    P^{\tilde{\G}}_y(  L_K >0 , X_{L_{K}}=x)=g(y,x)e_{\hat{K},\G^{A}}(x).
\end{equation}
We now define the equilibrium measure of $K$ in $\tilde{\G}$ by 
\begin{equation}
\label{defequilibriumcable}
e_{K,\tilde{\G}}(x)\stackrel{\text{def.}}=e_{ \hat{\partial} K,\G^{ \hat{\partial} K}}(x)  1_{\{x\in{ \hat{\partial} K}\}},
\end{equation}
with $\G^{ \hat{\partial} K}$ as supplied by Lemma~\ref{GA} and the (discrete) equilibrium measure on the right-hand side as defined in \eqref{defeAcap}. For $K,\hat{K}$ and $A$ as in \eqref{defKhatA}, we then have that
\begin{equation}
\label{consistencyequilibrium}
e_{\hat{K},\G^{A}}(x)=e_{K,\tilde{\G}}(x)\text{ for all }x\in{A}.
\end{equation}
Indeed, \eqref{consistencyequilibrium} follows from \eqref{exitequi1} when $x\in{\hat{\partial}K},$ and both terms of \eqref{consistencyequilibrium} are equal to $0$ when $x\in{A\setminus\hat{\partial}K}$ by \eqref{exitequi1} and \eqref{defequilibriumcable}. In particular if $K\subset G,$ by \eqref{consistencyequilibrium} with $\hat{K}=K$ and $A=\varnothing,$ the definition \eqref{defequilibriumcable} of the equilibrium measure on the cable system coincides with the definition of the equilibrium measure from \eqref{defeAcap}. Moreover, \eqref{consistencyequilibrium} can be used to obtain a description of the equilibrium measure purely in terms of the diffusion $X,$ instead of using the equilibrium measure on the discrete graph $\G^{\hat{\partial}K}$ as in \eqref{defequilibriumcable}. Indeed, denoting by $B_{\rho}(x,\eps)$ the ball centered at $x\in{\tilde{\G}}$ with radius $\eps\geq0$ for the distance $\rho$ introduced above Section \ref{subsec:diff}, which is well defined for small enough $\eps,$ one has
\begin{equation}
\label{eq:equiviaX}
e_{K,\tilde{\G}}(x)=\lim_{\eps\rightarrow0}\frac{d_x}{2\eps}P_x(L_K<H_{\partial B_{\rho}(x,\eps)})\text{ for all }x\in{\hat{\partial}K},
\end{equation}
where $d_x$ is the degree of $x$ if $x\in{G},$ and $d_x=2$ otherwise. In order to prove \eqref{eq:equiviaX}, one uses \eqref{consistencyequilibrium} with $A=\partial B_{\rho}(x,\eps)\cup \hat{\partial}K$ and $\hat{K}=A\cap K,$ and \eqref{defeAcap}, noting that $\lambda_x^A= d_x/(2\eps)$ by \eqref{eq:enh1} and that $\tilde{H}_K(X^{G^A})=\infty$ if and only if $L_K<H_{\partial B_{\rho}(x,\eps)}$ for $\eps$ small enough. Actually, the equality \eqref{eq:equiviaX} thus still holds when removing the limit as $\eps\rightarrow0,$ for small enough $\eps.$ Moreover, we obtain from \eqref{exitequi1} and \eqref{consistencyequilibrium} that
\begin{equation}
\label{exitequi}
    P^{\tilde{\G}}_y(  L_K >0 , X_{L_{K}}=x)=g(y,x)e_{K,\tilde{\G}}(x), \text{ for all $x,y\in{\tilde{\G}}$.}
\end{equation}
The identity \eqref{exitequi} is reminiscent of the equilibrium measure for the usual Brownian motion (on $\R^d$, with suitable killing when $d=1,2$), see for instance Proposition 3.3 in \cite{MR1717054}. In fact, \eqref{exitequi} (or \eqref{eq:equiviaX}) could be used instead of \eqref{consistencyequilibrium} as defining $e_{K,\tilde{\G}}(\cdot)$.

The capacity of a compact set $K\subset \tilde{\G}$ is defined as the total mass of the equilibrium measure, 
\begin{equation}
\label{defcap}
\mathrm{{\rm cap}}_{\tilde{\G}}(K)\stackrel{\mathrm{def.}}{=}\sum_{x\in{ \hat{\partial} K}}e_{K,\tilde{\G}}(x).
\end{equation}
When there is no risk of ambiguity, we will simply write $e_K$, $\mathrm{{\rm cap}}(K)$ instead of $e_{K,\tilde{\G}}$, $\mathrm{{\rm cap}}_{\tilde{\G}}(K)$.

 Using  \eqref{lawGA}, \eqref{defequilibriumcable}, and \eqref{consistencyequilibrium}, we can now extend a variety of useful results on equilibrium measures from the discrete case to $\tilde{\G}.$  
By (an adaptation of) \cite[(1.57)]{MR2932978}, one easily shows the following characterization of the capacity in terms of a variational problem as 
\begin{align}\label{variational}
\begin{split}
\mathrm{{\rm cap}}(K)=\Big(\inf_{\mu}\sum_{x,y\in{\hat{K}}}g(x,y)\mu(x)\mu(y)\Big)^{-1},
\end{split}
\end{align}
for $K,\hat{K}\subset  \tilde{\G}$ as in \eqref{defKhatA} with $A=\hat{K},$
where the infimum is over all probability measures $\mu$ on $\hat{K},$ see e.g.\ Proposition 1.9 in \cite{MR2932978}.
In view of \eqref{consistencyequilibrium}, when $K\subset K'$ are two compacts of $\tilde{\G},$  using (1.59) in \cite{MR2932978}, one obtains the `sweeping identity' 
\begin{equation}
\label{swappinglemma}
P_{e_{K'}}(X_{H_K}=x,H_K<\zeta)=e_K(x)\text{ for all }x\in\tilde{\G},
\end{equation}
where $H_K=\inf\{t\geq0:\,X_t\in{K}\},$ with the convention $\inf\varnothing=\zeta.$
In particular, summing \eqref{swappinglemma} over $x\in{{\partial K}}$ yields the monotonicity property
\begin{equation}
\label{capincreasing}
\mathrm{{\rm cap}}(K)\leq\mathrm{{\rm cap}}(K'), \text{ for $K\subset K'$ compacts of $\tilde{\G}.$}
\end{equation} 

 We now proceed to extend the notion of capacity to closed (not necessarily bounded) sets with finitely many components, cf.\ \eqref{limitcap} below, which will turn out helpful in the proof of Lemma \ref{LemmaapproGFF} below. For any measurable function $f:\tilde{\G}\rightarrow\R$ and $K$ a compact subset of $\tilde{\G},$ the harmonic extension $\eta^f_K$ of $f$ on $K$ is defined as
\begin{equation}
\label{harmoK}
    \eta_K^f(x)\stackrel{\mathrm{def.}}{=}\sum_{y\in{\partial K}}P_{x}(X_{{H}_{K}}=y,H_{K}<\zeta)f(y) \quad \text{ for all }x\in{\tilde{\G}}.
\end{equation}
Note that the sum in \eqref{harmoK} is well defined since for each $x\in{\tilde{\G}}$ the set $\partial_{x}K\stackrel{\mathrm{def.}}{=}\{y\in{\partial K}:\,P_{x}(X_{{H}_{K}}=y,H_{K}<\zeta)>0\}$ contains at most two points per edge of $\tilde{\G}$ intersecting $K,$ and hence is finite. In the sequel, a decreasing sequence of compacts $(K_n)_{n\in\N}$ is said to decrease to a compact $K$ if $K=\bigcap_{n\in \mathbb{N}} K_n$.
Moreover, in a slight abuse of notation, we say that an increasing sequence of compacts $(K_n)_{n\in\N}$ increases to a compact $K$ if $K$ is the closure of $\bigcup_{n\in \mathbb{N}} K_n$ (later on, this notion permits to assert for instance that if ${E}^{\geq h}(x_0)$ is compact, cf.\ \eqref{Ehx0}, the clusters ${E}^{\geq h'}(x_0)$ increase to ${E}^{\geq h}(x_0)$ as $h' \searrow h$). The following convergence result for harmonic extensions will be useful.

\begin{Lemme}
\label{approxharmo}
Let $f:\tilde{\G}\rightarrow\R$ be a continuous function 
and $K_n$, $n\in\N$, as well as $K$ be compact subsets of $\tilde{\G}$ such that $(K_n)_{n\in\N}$ increases or decreases to $K.$ Then for all $x\in{\tilde{\G}}$,
\begin{equation}
\label{eq:harmextconv}
    \eta_{K_n}^f(x)\tend{n}{\infty}\eta_{K}^f(x).
\end{equation}
\end{Lemme}
\begin{proof}
Fix some $x\in{\tilde{\G}}.$ For all $y\in{\partial_x K},$ let  $A_{n}^y=\{z\in{\partial_x{K_n}}:\,d(z,y)\leq d(z,y')\text{ for all }y'\in{\partial_x K}\}.$ Then $\max_{z\in{A_n^y}}d(z,y)\tend{n}{\infty}0$ for all $y\in{\partial_x K},$ and there exists an integer $N$ such that for all $n\geq N,$ the set $(A_{n}^y)_{y\in\partial_x K}$ is a partition of $\partial_x K_n.$ By \eqref{harmoK}, for all $x\in{\tilde{\G}}$ and $n\geq N$,
\begin{equation*}
    \eta_{K}^f(x)-\eta_{K_n}^f(x)=\sum_{y\in{\partial_x K}}\Big(P_x(X_{H_{K}}=y,H_{K}<\zeta)f(y)-\sum_{z\in{A_{n}^y}}P_x(X_{H_{K_n}}=z,H_{K_n}<\zeta)f(z)\Big).
\end{equation*}
By continuity, for any $\eps>0$ there exists $N'\geq N$ such that for all $n\geq N',$ $y\in{\partial_x K}$ and $z\in{A_n^y}$ we have $|f(y)-f(z)|\leq\eps.$ Therefore, for all $x\in{\tilde{\G}}$ and $n\geq N',$
\begin{equation*}
    |\eta_{K}^f(x)-\eta_{K_n}^f(x)|\leq \eps+\sum_{y\in \partial_x K}|f(y)| \cdot \big|P_x(X_{H_{K}}=y,H_{K}<\zeta)-P_x(X_{H_{K_n}}\in{{A_{n}^y}},H_{K_n}<\zeta)\big|.
\end{equation*}
Since for all $x\in{\tilde{\G}}$ and $y\in{\partial_x K}$ the absolute value of the difference on the right-hand side is bounded by
\begin{align*}
 &P_x(X_{H_{K}}=y,X_{H_{K_n}}\notin A_{n}^y,H_{K}<\zeta,H_{K_n}<\zeta)+P_x(H_{K}<\zeta,H_{K_n}=\zeta)
    \\&\qquad +P_x(X_{H_{K}}\neq y,X_{H_{K_n}}\in A_{n}^y,H_{K}<\zeta,H_{K_n}<\zeta)+P_x(H_{K}=\zeta,H_{K_n}<\zeta)
\end{align*}
and each of these terms tends to $0$ as $n \to \infty$, \eqref{eq:harmextconv} follows as $f$ is uniformly bounded on compacts.
\end{proof}
An interesting and immediate consequence of Lemma \ref{approxharmo} and \eqref{swappinglemma} is the following: if 
 $K_n$, $n\in\N$, and $K$ are compacts of $\tilde{\G}$ such that $(K_n)_{n\in\N}$ increases or decreases to $K,$ consider the quantity $\sum_{x\in \partial K}e_K(x)\eta_{K_n}^1(x)$ in case the $K_n$ are increasing and $\sum_{x\in \partial K}e_{K_1}(x)\eta_{K_n}^1(x)$ in case the $K_n$ are decreasing, respectively (which both equal $\text{cap}(K_n)$ by virtue of \eqref{swappinglemma}). We can then take $n \to \infty$ while applying \eqref{eq:harmextconv} with $f=1$ to obtain that
\begin{equation}
    \label{limitcap}
    \lim\limits_{n\rightarrow\infty}\mathrm{{\rm cap}}(K_n)=\mathrm{{\rm cap}}(K).
\end{equation}
Hence, we can extend the definition of the capacity to any closed set $A\subset\tilde{\G}$ by setting
\begin{equation}
\label{defcapinfinity}
\mathrm{{\rm cap}}(A)=\lim\limits_{n\rightarrow\infty}\mathrm{{\rm cap}}(A\cap K_n),
\end{equation}
where $(K_n)_{n\in\N}$ is any increasing sequence of compacts of $\tilde{\G}$ exhausting $\tilde{\G}.$ This limit exists and does not depend on the choice of the sequence $(K_n)_{n\in\N}$ by \eqref{capincreasing}, and it is consistent with the existing definition of capacity for compacts, cf.\ \eqref{defcap}, by means of \eqref{limitcap}.

\subsection{Varying killing measure and the cables $I_x$}
\label{S:I_x}

In the sequel, it will repeatedly be useful to compare the diffusion $X$ on $\tilde{\G}$ for varying killing measure. In particular, this comprises `infinite-volume' limits, in which all but finitely many $x\in \overline{G}$ initially satisfy $\bar{\kappa}_{x}=\infty$, and $\bar{\kappa}$ is sequentially reduced, see \eqref{eq:def_approx} below. Consider the family of graphs $( {\G}_{\bar{\kappa}})_{\bar{\kappa}}$, where ${\G}_{\bar{\kappa}}=(\overline{G}, \bar{\lambda}, \bar{\kappa})$, for fixed $\overline{G}$ and $\bar{\lambda}$ and varying killing measure $\bar{\kappa} \in [0,\infty]^{\overline{G}}$. Let $\tilde{\G}_{\bar{\kappa}}$ be the cable system associated to ${\G}_{\bar{\kappa}}$ (cf.\ below \eqref{deflambdarho}). In view of \eqref{eq:0bcreduction1new}, \eqref{eq:defGfinite}, one can interpret
\begin{equation}
\label{eq:graphinclusion}
\tilde{\G}_{\bar{\kappa}'}\subset \tilde{\G}_{\bar{\kappa}} \quad \text{ if } \bar{\kappa}' \geq \bar{\kappa},
\end{equation}
where $\bar{\kappa}' \geq \bar{\kappa}$ means $\bar{\kappa}'_x \geq \bar{\kappa}_x$ for all $x \in \overline{G}.$ We then set, under $P^{\tilde{\G}_{\bar{\kappa}}}_x,$ $x\in{\tilde{\G}_{\bar{\kappa}'}}(\subset\tilde{\G}_{\bar{\kappa}}),$
\begin{equation}
\label{eq:zeta}
X_t^{\bar{\kappa}'}=
\begin{cases}
X_t, \text{ if } t< \zeta_{\bar{\kappa}'}\\
\Delta, \text{ if } t\geq \zeta_{\bar{\kappa}'} 
\end{cases}
\text{ where } \zeta_{\bar{\kappa}'} = \inf \{ t \geq 0 : X_t \notin  \tilde{\G}_{\bar{\kappa}'}  \}.
\end{equation}
By Theorem 4.4.2.\ in \cite{MR2778606}, the Dirichlet form associated to $X_t^{\bar{\kappa}'}$ is $\mathcal{E}_{\tilde{\G}_{\bar{\kappa}'}},$ and so
\begin{equation}
\label{lawkappa'vskappa}
    \text{the law of $X_t^{\bar{\kappa}'}$ under $P^{\tilde{\G}_{\bar{\kappa}}}_x$ is $P^{\tilde{\G}_{\bar{\kappa}'}}_x$ for all $x\in{\tilde{\G}_{\bar{\kappa}^{'}}}.$}
\end{equation}

We now briefly compare the above setup to existing definitions of the metric graph $\tilde{\G}$ and its associated diffusion $X$, which do not usually involve attaching cables $I_x$ to the vertices $x \in G$ (see e.g.\ Section~5 of \cite{MR2570665}, Section~2 of \cite{MR3152724} or Section~2 of \cite{MR3502602}).
Upon considering a suitable trace process in the present context, see \eqref{dirge} below, these two descriptions are essentially equivalent and in particular, they lead to the same notion of capacity for most sets of interest. Most important to our investigations is the feature that the cables $I_x$ provide natural embeddings as $\kappa$ varies, see \eqref{eq:graphinclusion}--\eqref{eq:zeta} above. This will be useful for approximation purposes, see \eqref{eq:def_approx} and Lemmas \ref{LemmaapproGFF} and \ref{approximationprop} below, as well as to derive \eqref{eq:laplacecap} and \eqref{eqcouplingintergffszn} in the case $\kappa\neq 0$.
We define $\tilde{\G}^{\Ed}$ as the closed subset of $\tilde{\G}$ consisting of the closure of the union of the intervals $I_e,$ $e\in{E}$, (or, in other words, the subset of $\tilde{\G}$ obtained upon removing the intervals $I_x$, $x \in G$) and denote by $X^{\tilde{\G}^{\Ed}}$ the trace on $\tilde{\G}^{\Ed}$ of $X.$  
 One can prove by Theorem 6.2.1.\ in \cite{MR2778606} that the Dirichlet form on $L^2_{\Ed}(\tilde{\G}^{\Ed},m_{|\tilde{\G}^{\Ed}})= \{f\in{L^2(\tilde{\G}^{\Ed},m_{|\tilde{\G}^{\Ed}})}:\sum_{x\in{{G}}}\kappa_xf(x)^2<\infty\}$ associated to $X^{\tilde{\G}^{\Ed}}$ is
\begin{equation}
\label{dirge}
    \mathcal{E}_{\tilde{\G}^{\Ed}}(f,g)\stackrel{\text{def.}}{=}\frac12(f',g')_{m_{|\tilde{\G}^{\Ed}}}+\sum_{x\in{G}}\kappa_xf(x)g(x) \text{ for }f,g\in{D({\tilde{\G}^{\Ed}},m_{|\tilde{\G}^{\Ed}})\cap L^2_{\Ed}({\tilde{\G}^{\Ed}},m_{|\tilde{\G}^{\Ed}})},
\end{equation}
where we recall that the space $D$ had been introduced below \eqref{eq:quadraticform}. If $\kappa\equiv0$ on $\mathcal{G}$, the process $X^{\tilde{\G}^{\Ed}}$ thus corresponds to the usual diffusion on the cable system $\tilde{\G}^{\Ed}.$ If $\kappa\geq0$  on $\mathcal{G}$ (i.e.\ $\tilde{\G}=\tilde{\G}_{\kappa}$), it follows from Theorems 6.1.1.\ and A.2.11.\ in \cite{MR2778606} that $X^{\tilde{\G}^{\Ed}}$ has the same law under $P^{\tilde{\G}}_x$ as the diffusion $X^{\tilde{\G}^{\Ed}_0}$ under $P^{\tilde{\G}_{0}}_x$ (where $\tilde{\G}_{0}= \tilde{\G}_{\kappa \equiv 0}$) killed at time
$    \zeta_{\kappa}^{\Ed}=\inf\{t<\zeta_0:\,\sum_{x\in{{G}}}\ell_x(t)\kappa_x\geq\xi\}$,
where $\xi$ is an independent exponential variable with parameter $1$ (with the convention $\inf\varnothing=\zeta_0$). The latter is the process studied e.g.\ in Section~2 of~\cite{MR3502602}. Moreover, the trace of $X^{\tilde{\G}^{\Ed}}$ (under $P^{\tilde{\G}}_x$) on $G$ has the same law as $Z$, hence the local times $(\ell_y(t))_{y\in{\tilde{\G}^{\Ed}},t\geq0}$ have the same law under $P^{\tilde{\G}}_x$ as those of the process $X^{\tilde{\G}_0^{\Ed}}$ (killed at time $\zeta_{\kappa}^{\Ed}$) under $P^{\tilde{\G}_0}_x,$ i.e. the local times of the process introduced in \cite{MR3502602}. 

Consequently, for compact $K\subset\tilde{\G}^{\Ed}$ one could have defined a notion $\mathrm{{\rm cap}}_{\tilde{\G}^{\Ed}}(K)$ similarly as in \eqref{defequilibriumcable} and \eqref{defcap}, but starting from the process $X^{\tilde{\G}^{\Ed}}$ and considering suitable enhancements of $\tilde{\G}^{\Ed}$, resulting in $\mathrm{{\rm cap}}_{\tilde{\G}^{\Ed}}(K)= \mathrm{{\rm cap}}_{\tilde{\G}}(K)$ for all $K\subset\tilde{\G}^{\Ed}$. This can be further strengthened when $\kappa\equiv0$, as asserted in the following lemma, which records the capacity of the cables $I_x$ for later purposes.

\begin{Lemme}
\label{capged}
   For all $x\in{{G}},$ the following dichotomy holds:
\begin{equation}
    \label{capIx}
    \text{ if }\kappa_x>0,\text{ then }\mathrm{{\rm cap}}(\overline{I}_x)=\infty,\text{ and if }\kappa_x=0,\text{ then }\mathrm{{\rm cap}}(\overline{I}_x)=\mathrm{{\rm cap}}(\{x\}).
\end{equation}
Moreover, if $\kappa\equiv0,$ then for all connected and closed sets $A\subset\tilde{\G}$ such that $A\cap\tilde{\G}^{\Ed}\neq\varnothing,$ one has $\mathrm{{\rm cap}}_{\tilde{\G}}(A)=\mathrm{{\rm cap}}_{\tilde{\G}^{\Ed}}(A\cap\tilde{\G}^{\Ed}).$ 
\end{Lemme}
\begin{proof}
   We first show \eqref{capIx}. If $\kappa_x>0,$ then for all $t\in{(0,\rho_x)},$ writing $y_t=x+(\rho_x-t)\cdot I_x$ (see the beginning of Section \ref{s:usefulresults} for notation), 
    we see by \eqref{eq:enh1} that $\kappa_{y_t}^{\{{y_t}\}}=\frac{1}{2t}.$ Let $I_x^t=\{ x+s\cdot I_x : 0\leq s \leq \rho_x -t\}$. 
    Then   by \eqref{defequilibriumcable}      $e_{I_x^t}(y_t)=\lambda_{y_t}^{\{y_t\}}P_{y_t}^{\G^{\{{y_t}\}}}(\tilde{H}_{\{x,y_t\}}=\infty)=\kappa^{\{y_t\}}_{y_t}
    ,$ and so we see that $\mathrm{{\rm cap}}(I_x^t)\geq e_{I_x^t}(y_t)=\frac{1}{2t}.$ Hence, by \eqref{defcapinfinity}, we obtain $\mathrm{{\rm cap}}(\overline{I}_x)=\infty$  as $t \downarrow 0$. 
    
    If $\kappa_x=0,$ then keeping the same notation, we have for all $t\in{(0,\infty)}$ that $P_{y_t}^{\G^{\{{y_t}\}}}(\tilde{H}_{I_x^t}=\infty)=0,$ since $X$ behaves like a Brownian motion on $I_x$ and hence always return to $I_x^t$ in finite time. Moreover $P_{x}^{\G^{\{{y_t}\}}}(\tilde{H}_{I_x^t}=\infty)=P_{x}^{\G^{\{{y_t}\}}}(\tilde{H}_{\{x\}}=\infty).$ Therefore by \eqref{defequilibriumcable}, we get $\mathrm{{\rm cap}}(I_x^t)=e_{I_x^t}(x)+0=e_{\{x\}}(x)=\mathrm{{\rm cap}}(\{x\}),$ and by \eqref{defcapinfinity} we obtain that $\mathrm{{\rm cap}}(\overline{I}_x)=\mathrm{{\rm cap}}(\{x\})$.

    Suppose now that $\kappa\equiv0,$ and let $K\subset\tilde{\G}$ be a connected and compact set such that $K\cap\tilde{\G}^{\Ed}\neq\varnothing.$ Then since $X$ cannot be killed via $I_x$ for all $x\in{G},$ we have $\hat{\partial} (K\cap\tilde{\G}^{\Ed})=\hat{\partial}K$ and  for all $x\in{\hat{\partial} K}$
    \begin{equation*}
        e_{K\cap\tilde{\G}^{\Ed}}(x)=\lambda_x^{\hat{\partial} K}P^{\tilde{\G}^{\hat{\partial} K}}_x(\tilde{H}_{K\cap\tilde{\G}^{\Ed}}=\infty)=\lambda_x^{\hat{\partial} K}P^{\tilde{\G}^{\hat{\partial} K}}_x(\tilde{H}_K=\infty)=e_{K}(x),
    \end{equation*}
    from which the claim follows for such $K$, and for arbitrary closed connected sets by means of \eqref{defcapinfinity}.
\end{proof}

\begin{Rk} \label{R:removecables1}
The second part of Lemma \ref{R:removecables1} implies that, when $\kappa\equiv0,$ one can consider $\tilde{\G}^-$ instead of $\tilde{\G}$ and all our results, for instance \eqref{eqcouplingintergffszn} or \eqref{eq:laplacecap} for $\mathrm{cap}_{\tilde{\G}^-}(E^{\geq h}_-(x_0))$ instead of $\mathrm{cap}_{\tilde{\G}}(E^{\geq h}(x_0)),$ hold under the same conditions, where $E^{\geq h}_-(x_0)=E^{\geq h}(x_0)\cap\tilde{\G}^-$ is the connected component of $x_0$ in $\{x\in{\tilde{\G}^-}:\,\phi_x\geq h\}.$  Note that this is not true anymore when $\kappa \not\equiv 0$. Indeed for instance one has by \eqref{capIx} that ${\rm{cap}}_{\tilde{\mathcal{G}}}(\overline{I}_x)=\infty,$ yet, $\mathrm{{\rm cap}}_{\tilde{\G}^{\Ed}}(\overline{I}_x\cap\tilde{\G}^{\Ed})= {\rm{cap}}_{\mathcal{G}}(\{x\}) \leq \lambda_x < \infty.$ Therefore, one cannot simply replace $\mathrm{cap}_{\tilde{\G}}(E^{\geq h}(x_0))$ by $\mathrm{cap}_{\tilde{\G}^-}(E^{\geq h}_-(x_0))$ in \eqref{eq:laplacecap}, and, when considering $\tilde{\G}^-$ instead of $\tilde{\G},$ one has to change the isomorphism \eqref{eqcouplingintergffszn} to take into account the influence of the trajectories in the random interlacement process entirely included in one of the cables $I_x,$ $x\in{G}$ with $\kappa_x>0,$ possibly hitting the sign clusters, see Remark \ref{R:isomisom},\ref{eqcouplingintergffG-} for details.
\end{Rk}


\subsection{The Gaussian free field}
\label{subsec:GFF}
We now collect a few important properties of the Gaussian free field $(\phi_x)_{x\in{\tilde{\G}}}$ on the cable system~$\tilde\G$ defined in \eqref{defGFF}.  
We first recall its strong spatial Markov property and refer to Section~1 of \cite{MR3492939} for details. For any open set $O\subset\tilde{\G},$ we consider the $\sigma$-algebra $\mathcal{A}_O=\sigma({\phi}_x,\,x\in{O}),$ and for any compact $K\subset\tilde{\G}$ we define $\mathcal{A}_K^+=\bigcap_{\eps>0}\A_{K^\eps}$, where $K^\eps$ is the open $\eps$-ball around $K$ for the distance $d.$ We say that $\K$ is a compatible random compact subset of $\tilde{\G}$ if $\K$ is a compact subset of $\tilde{\G}$ with finitely many connected components and $\{\K\subset O\}\in{\A_O}$ for any open set $O\subset\tilde{\G}.$ We then define
\begin{equation}
\label{eq:defAK^+}
    \begin{split}
    \mathcal{A}_{\K}^+=\big\{A\in{\A_{\tilde{\G}}}:\,A\cap\{\K\subset K\}\in{\A_K^+}\ &\text{for all }K\subset\tilde{\G} \text{ which are compact}\\
    &\text{and the closure of their respective interiors}\big \}.
    \end{split}
\end{equation}
The Markov property now states that for any compatible random compact $\K,$
\begin{equation}
\label{Markov}
    \text{conditionally on }\mathcal{A}_{\K}^+,\ (\phi_x)_{x\in{\tilde{\G}}}\text{ is a Gaussian field with mean }\eta_{\K}^{\phi}\text{ and covariance }g_{\K^c},
\end{equation}
where $\eta_{\K}^{\phi}$ was defined in \eqref{harmoK} and $g_{\K^c}$ in \eqref{Greendef}. An application of the Markov property is that, conditionally on $(\phi_x)_{x\in{{G}}},$ if $e=\{y,z\}\in{E},$ the law of $(\phi_x)_{x\in{I_e}},$ is that of a Brownian bridge of length $\rho_e$ between $\phi_y$ and $\phi_z$ of a Brownian motion with variance $2$ at time $1,$ and these Brownian bridges are independent as $e$ varies. Similarly, conditionally on $(\phi_x)_{x\in{{G}}},$ one can describe the law of $(\phi_x)_{x\in{I_y}},$ as that of a Brownian bridge of length $\rho_y$ between $\phi_y$ and $0$ of a Brownian motion with variance $2$ at time $1$ if $\kappa_y>0,$ and as that of a Brownian motion starting in $\phi_y$ with variance $2$ at time $1$ if $\kappa_y=0,$ and all these Brownian bridges and Brownian motions are independent.  We refer to Section 2 of \cite{DrePreRod} for a proof of this result on $\Z^d,$ $d\geq3,$ which can easily be adapted to any transient graph. In particular, we have that
\begin{equation}
\label{phiedgeind}
\begin{array}{l}
    \text{conditionally on }(\phi_x)_{x\in{{G}}},\text{ the random fields }(\phi_x)_{x\in{I_e}}, e\in{E\cup{G},\text{ are}} 
    \\\text{independent, and for all }e\in{E\cup{G}}, \text{ the field }(\phi_{x})_{x\in{I_e}}\text{ only depends on }\phi_{|e},
    \end{array}
\end{equation}
where $\phi_{|e}=(\phi_x,\phi_y)$ if $e=\{x,y\}\in{E}$ and $\phi_{|e}=\phi_x$ if $e=x\in{{G}}.$ Moreover, using the exact formula for the distribution of the maximum of a Brownian bridge, see e.g.~\cite{MR1912205}, Chapter~IV.26, one knows that for all $e\in{E\cup{G}}$
\begin{equation}
\label{phionIx}
  \P^G(|\phi_z|>0\text{ for all }z\in{I_e}\,|\,\phi_{|e})=\big(1- p_e^{\G}(\phi)\big) 1_{e\in{E}},
\end{equation}
where for all $e=\{x,y\}\in{E}$ and $f:{G}\rightarrow\R,$
\begin{equation}
\label{defpef}
     p_e^{\G}(f)\stackrel{\mathrm{def.}}{=}p_e^{\G}(f,0)=\begin{cases}
    \exp\big(-2\lambda_{x,y}f(x)f(y)\big),&\text{ if }f(x)f(y)\geq0,
    \\1,&\text{ otherwise.}
    \end{cases}
\end{equation} 
A useful notation $ p_e^{\G}(f,g)$ will later be introduced and include \eqref{defpef} as a special case when $g=0$, see \eqref{defpe} below. 

\subsection{Random interlacements}
\label{subsec:RI}
We now briefly introduce random interlacements on the cable system $\tilde{\G}.$ We define the set of doubly infinite trajectories $W_{\tilde{\G}}$ as the set of functions $w:\R\rightarrow\tilde{\G}\cup{\Delta},$ for which there exist $-\infty\leq\zeta^-<\zeta^+\leq\infty$ such that $w_{|(\zeta^-,\zeta^+)}\in{C((\zeta^-,\zeta^+),\tilde{\G})}$ and $w(t)=\Delta$ for all $t\notin{(\zeta^-,\zeta^+)}.$ For each $w\in{W_{\tilde{\G}}},$ we also define $ p_{\tilde \G}^*(w)= p^*(w)$ as the equivalence class of $w$ modulo time shift; here, $w$ and $w'$ are equal modulo time shift if there exists $t_0\in\R$ such that $w(t+t_0)=w(t)$ for all $t\in\R,$ and $W_{\tilde{\G}}^*=\{p^*(w):\,w\in{W_{\tilde{\G}}}\}.$ Let $\mathcal{W}_{\tilde{\G}}$ be the $\sigma$-algebra on $W_{\tilde{\G}}$ generated by the coordinate functions, and $\mathcal{W}_{\tilde{\G}}^*=\{A\subset W_{\tilde{\G}}^*:(p^*)^{-1}(A)\in{\mathcal{W}_{\tilde{\G}}}\}.$ For each compact $K$ of $\tilde{\G},$ we denote by $W_{K,\tilde{\G}}^0$ the set of trajectories $w\in{W_{\tilde{\G}}}$ with $H_K(w)=0,$ where $H_K(w)=\inf\{t\in\R:\,w(t)\in{K}\},$ with the convention $\inf\varnothing=\zeta^+,$ and $W_{K,\tilde{\G}}^*=\{ w^*\in W_{\tilde{\G}}^* : (p^*)^{-1}(w^*)\cap W_{K,\tilde{\G}}^0\neq\varnothing\}$. For $w\in{W_{\tilde{\G}}},$ we define the forward part of $w$ as $(w(t))_{t\geq0}$ and the backward part of $w$ as $(w(-t))_{t\geq0},$ which are both elements of $W_{\tilde{\G}}^+,$ see above \eqref{eq:quadraticform}. For $w^*\in{W_{K,\tilde{\G}}^*}$ we define the forward (resp.~backward) part of $w^*$ on hitting $K$ as the forward (resp.~backward) part of the unique trajectory in $(p^*)^{-1}(\{w^*\})\cap W_{K,\tilde{\G}}^0.$ 

The intensity measure underlying random interlacements on $\tilde{\G}$ is defined as follows.
For a set  $A\in{\mathcal{W}_{\tilde{\G}}}$ we write $A^{\pm}\stackrel{\text{def.}}{=}\{(w(\pm t))_{t\geq0}:\,w\in{A}\}$, 
whence $A^+,A^-\in{\mathcal{W}^+_{\tilde{\G}}}.$ The set of all $A\in{\W_{\tilde{\G}}}$ with $A\subset{{W}_{K,\tilde{\G}}^0},$ such that $A$ is equal to the set of $w\in{W_{K,\tilde{\G}}^0}$ whose forward part is in $A^+$ and whose backward part is in $A^-,$ is denoted by $\mathcal{W}_{K,\tilde{\G}}^0.$ We then observe that $\mathcal{W}_{K,\tilde{\G}}^0$ and $\{A\in{\mathcal{W}_{\tilde{\G}}}:\,W_{K,\tilde{\G}}^0\cap A=\varnothing\}$ generate $\mathcal{W}_{\tilde{\G}}.$ 
Recalling the definition of the last exit time $L_K$ and the exterior boundary $\hat{\partial}K$ from \eqref{defpartialext} and below, for all $x\in{ \hat{\partial} K}$  let
\begin{equation} \label{eq:PKG}
P^{ K,\tilde{\G}}_x \equiv P^{K}_x \text{ be the law of }(X_{t+L_K})_{t\geq0} \text{ under }P_x(\cdot\,|\, L_K >0 , X_{L_{K}}=x).
\end{equation}
We now define a measure $Q_{K,\tilde{\G}}$ on $\mathcal{W}_{\tilde{\G}},$ whose restriction to $\mathcal{W}_{K,\tilde{\G}}^0$ is given by
\begin{equation}
\label{defQK}
Q_{K,\tilde{\G}}(A)=\sum_{x\in{ \hat{\partial} K}}e_K(x)P^{\tilde{\G}}_x(X \in {A^+})P^{K,\tilde{\G}}_x(X\in {A^-}), \quad A \in \mathcal{W}_{K,\tilde{\G}}^0,
\end{equation}
and such that $Q_{K,\tilde{\G}}(A)=0$ for all $A\in{\mathcal{W}_{\tilde{\G}}}$ with $A\cap W_{K,\tilde{\G}}^0=\varnothing.$  It is essentially folklore by now that there exists a unique measure $\nu_{\tilde{\G}}$ on $W^*_{\tilde{\G}},$ such that for all compacts $K\subset\tilde{\G},$
\begin{equation}
    \label{definter}
    \nu_{\tilde{\G}}(A^*)=Q_{K,\tilde{\G}}\big((p^*)^{-1}(A^*)\big)\text{ for all }A^*\in{\mathcal{W}_{\tilde{\G}}^*},\, A^*\subset W_{K,\tilde{\G}}^*.
\end{equation}
We will not give a proof of the existence of the measure $\nu_{\tilde{\G}};$ instead, we refer to \cite{MR2525105} for a proof of the existence of such a measure on the discrete graph $\G$ when $\kappa\equiv0,$ and to \cite{MR3502602} for the setting of the cable system associated to ${\Z}^d,$ $d\geq3.$ Indeed, one can easily adapt these proofs to obtain a measure $\nu_{\tilde{\G}}$ such that \eqref{definter} holds for all compacts $K$ of $\tilde{\G}$ with $ \hat{\partial} K\subset{G},$
also in the case $\kappa \not\equiv 0$ (see also Remark~\ref{R:nomassconversion}). Considering now the case of arbitrary compact subsets $K$ of $\tilde{\G},$ one can thus construct a measure $\nu_{\tilde{\G}^{ \hat{\partial} K}}$ such that \eqref{definter} holds for $\nu_{\tilde{\G}^{ \hat{\partial} K}}$ and $K.$ Using the fact that $P^{\tilde{\G}}_x$ is the law of the trace of $X$ on $\tilde{\G}$ under $P_x^{\tilde{\G}^{ \hat{\partial} K}},$  one easily deduces that $\nu_{\tilde{\G}}$ is the `trace on $\tilde{\G}$' of $\nu_{\tilde{\G}^{ \hat{\partial} K}},$ so that \eqref{definter} also holds for $\nu_{\tilde{\G}}$ and $K.$ Alternatively, a direct proof of \eqref{definter} on the cable system is also presented in Theorem \ref*{Pre1:nuexists} of \cite{Pre1}.

The random interlacement process ${\omega}$ is a Poisson point process on $W^*_{\tilde{\G}}\times(0,\infty)$ under the probability $\P^{I}_{\tilde{\G}}$ with intensity measure $\nu_{\tilde{\G}}\otimes\lambda,$ where $\lambda$ is the Lebesgue measure on $(0,\infty).$ When $\kappa \not\equiv 0,$ the forward and backward parts of the trajectories can be killed before blowing up; in our setup this is realized by either part of the trajectory exiting $\tilde{\G}$ to $\Delta$ via $I_x$ for some $x\in{G}$ with $\kappa_x>0$. We also denote by ${\omega}_u$ the point process which consist of the trajectories in $\omega$ with label less than $u,$ by $(\ell_{x,u})_{x\in{\tilde{\G}}}$ the continuous field of local times relative to $m$ on $\tilde{\G}$ of $\omega_u$ and by $\I^u=\{x\in{\tilde{\G}}:\ \ell_{x,u}>0\}$ the interlacement set at level $u.$ The set $\I^u$ is characterized by the following identity: for any measurable set $A\subset{\tilde{\G}},$ 
\begin{equation}
\label{defIu}
\P^I_{\tilde{\G}}(\I^u\cap A=\varnothing)=\exp\left(-u\, \mathrm{{\rm cap}}(\overline{A})\right)
\end{equation} 
(note that the set $\I^u$ is open, so it intersects $A$ if and only if it intersects $\overline{A}$). The trace $\hat{\omega}_u$ of $\omega_u$ on $G$ has the same law under $\P^I_{\tilde{\G}}$ as the usual discrete random interlacement process, see \cite{MR2525105} in the case $\kappa\equiv0.$  If $\kappa \not\equiv 0,$ a trajectory in $\hat{\omega}_u$ can start or end at a fixed point $x\in{G},$ and in this case we say that this trajectory is \emph{killed} at $x.$ We also define $\I_E^u\subset E\cup G$ to be the set of edges in $E$ crossed by at least one single trajectory in $\hat{\omega}_u,$ union with the set of vertices at which a trajectory in $\hat{\omega}_u$ is killed. 
In the case $\lambda_{x,y}=\frac{T}{T+1}$ for all $x,y\in{E}$ and $\kappa_x=\frac{\text{deg}(x)}{T+1}$ for all $x\in{G},$ $T>0,$ the discrete random interlacement process $\hat{\omega}_u$ corresponds to the model of `finitary random interlacements' studied in \cite{MR3962876}. 
	  In view of Remark~\ref{R:nomassconversion}, this actually fits within the framework of \cite{MR2525105} upon suitable enhancement of $\G$.

 The law of $\omega_u$ can also be described as follows: for any compact $K$ of $\tilde{\G}$, the law of the forward trajectories in $\omega_u$ hitting $K$ is a Poisson point process with intensity $uP_{e_K}^{\tilde{\G}}$ which can be constructed from a Poisson point process of discrete trajectories with intensity $uP_{e_K}^{\tilde{\G}^{ \hat{\partial} K}}(\widehat{Z} \in \cdot)$ by adding Brownian excursions on the edges. Hence, $\omega_u$ can be constructed from $\hat{\omega}_u$ by adding independent Brownian excursion on the edges, see \cite{MR3502602} for details. In particular,
\begin{equation}
\label{omegaedgeind}
    \begin{array}{l}
    \text{conditionally on }\hat{\omega}_u,\text{ the random variables }(\ell_{x,u})_{x\in{I_e}},e\in{E\cup{G}}, \text{ are}
    \\\text{independent, and for all }e\in{E\cup{G}},(\ell_{x,u})_{x\in{I_e}}\text{ only depends on }\hat{\omega}_{u,e},
    \end{array}
\end{equation}
where $\hat{\omega}_{u,e}$ is the set of trajectories in $\hat{\omega}_u$ hitting $e.$  When there is no risk of ambiguity, we  abbreviate $\P^I=\P^I_{\tilde{\G}},$ and $\nu=\nu_{\tilde{\G}}.$

\section{Main results}
\label{sec:results}
In this section, we state our main results, Theorems \ref{mainresult}, \ref{mainresultcap} and \ref{couplingintergff}, and explore their consequences. Put together, these results in particular imply Theorem \ref{T:main}, see the end of this section for the short proof, but in fact they provide more detailed results. Theorem \ref{mainresult}, together with its Corollary \ref{maincor}, roughly corresponds to \ref{maint1} in Theorem \ref{T:main}. Theorem \ref{mainresultcap} investigates the properties of the cluster capacity observable. In particular, it establishes that, when bounded almost surely,  the cluster $E^{\geq h}(x_0)$ has a capacity described by \eqref{eq:laplacecap}. Theorem \ref{couplingintergff} then broadly speaking relates 
 $\text{\eqref{eq:laplacecap}}_{h\geq 0}$ and the identity \eqref{eqcouplingintergffszn} between random interlacements and the Gaussian free field on $\tilde{\G}$. In doing so, it also supplies new instances of \eqref{eqcouplingintergffszn}, see Remark \ref{R:isomisom},1), along with a version on the discrete base graph $\G$, see \eqref{eqcouplingintergffdis}. Finally, some further interesting consequences are put together in Corollaries \ref{dichotomy} and \ref{percoath*}.
 
We now lay the ground for our first main result, Theorem \ref{mainresult}. Its true meaning becomes transparent upon defining, next to $\tilde{h}_*$ (see \eqref{defh*bou}) two further critical parameters. As will soon become clear, the conditions $\kappa\equiv0$ or \eqref{capcondition} appearing in Theorem \ref{mainresult} will cause various of these parameters to coincide, leading to streamlined results. We first introduce
\begin{equation}
\label{defh*com}
\tilde{h}_*^{{\rm com}} =\inf \big\{ h \in \R :  \text{ for all $x_0\in{\tilde{\G}}$, }\, \P^G(E^{\geq h}(x_0) \text{ is non-compact})=0 \big \}
\end{equation}
(recall that compactness is with respect to the graph distance $d$).
Every compact set is ($d$-)bounded, so we always have $\tilde{h}_*^{{\rm com}} \geq \tilde{h}_*.$ The third critical parameter, involving the capacity of clusters in $E^{\geq h},$ is
\begin{equation}
\label{defh*cap}
\tilde{h}_*^{{\rm cap}} =\inf \big \{ h \in \R :  \text{ for all $x_0\in{\tilde{\G}}$, }\, \P^G(\mathrm{cap}(E^{\geq h}(x_0))=\infty)=0 \big \},
\end{equation}
see \eqref{defcapinfinity} for the definition of capacity in this context. Note that \eqref{defh*cap} is well-defined due to the monotonicity of $\text{cap}(\cdot)$, see \eqref{capincreasing}, which extends to arbitrary closed sets on account of \eqref{defcapinfinity}. Every compact set has finite capacity, so $\tilde{h}_*^{{\rm com}} \geq \tilde{h}_*^{{\rm cap}} ,$ and we therefore have that
\begin{equation}
\label{capcombougen}
\text{on any transient graph, }\tilde{h}_*^{{\rm com}} \geq \tilde{h}_*^{{\rm cap}} \text{ and }\tilde{h}_*^{{\rm com}} \geq \tilde{h}_*.
\end{equation}
On any graph such that $\kappa\equiv0$ or \eqref{capcondition} is verified, the situation becomes simpler, due to the following basic result. Its proof can be omitted at first reading.

\begin{Lemme}
\label{1stpartofmaincor} $(h\in\R$, $x_0\in{\tilde{\G}})$.

\medskip
\noindent $\P^G$-a.s., if either $h\geq0,$ $\mathrm{{\rm cap}}(E^{\geq h}(x_0))<\infty$ or $\kappa\equiv0$ on ${G},$ then $E^{\geq h}(x_0)$ is compact if and only if it is bounded.
\end{Lemme}
\begin{proof}
Observe that by definition, a connected set $K$ is compact if and only if it is a closed and bounded subset of $\tilde{\G}$ such that $I_x\cap K$ is a connected compact subset of $I_x$ for all $x\in{{G}}.$ Therefore, if the level set $E^{\geq h}(x_0)$ of $x_0$ is compact, then it is bounded. Hence, we only have to show the reverse implication, and we assume from now on that $E^{\geq h}(x_0)$ is bounded. First note that, as explained below \eqref{Markov}, if $\kappa_x=0,$ since $\phi$ on $I_x$ conditioned on $\phi_x$ has the same law as a Brownian motion starting in $\phi_x$ with variance $2$ at time $1,$ we have that $I_x\cap E^{\geq h}(x_0)$ is $\P^G$-a.s.\ a connected compact of $I_x.$ Therefore $E^{\geq h}(x_0)$ is a.s.\ compact if $\kappa\equiv0.$ If $\kappa_x>0$ we have by \eqref{capIx} applied to the graph $\G^{\{x+t\cdot I_x\}}$ (cf.~Lemma \ref{GA} for notation) that $\mathrm{{\rm cap}}(I_x^t)=\infty$, where $I_x^t=\{ x+s\cdot I_x : t \leq s < \rho_x\}$. If $\mathrm{{\rm cap}}(E^{\geq h}(x_0))<\infty,$ by \eqref{capincreasing} we obtain $ I_x^t\not\subset E^{\geq h}(x_0),$ that is $I_x\cap E^{\geq h}(x_0)$ is a connected compact of $I_x,$ and so $E^{\geq h}(x_0)$ is compact. Finally, if $\kappa_x>0$ and $h\geq0,$ as explained below \eqref{Markov}, since $\phi$ on $I_x$ conditioned on $\phi_x$ has the same law as a Brownian bridge of finite length between $\phi_x$ and $0$ of a Brownian motion with variance $2$ at time $1,$ $I_x\cap E^{\geq h}(x_0)$ is a.s.\ a connected compact of $I_x,$ and so $E^{\geq h}(x_0)$ is a.s.\ compact.
\end{proof}

Lemma \ref{1stpartofmaincor} has two immediate consequences. On the one hand, 
in view of \eqref{defh*bou}, \eqref{defh*com} and by \eqref{capcombougen}, Lemma \ref{1stpartofmaincor} (applied in the case $\kappa\equiv 0$) yields that
\begin{equation}
\label{capcomboukappa}
\text{ if }\G\text{ is a transient graph with }\kappa\equiv 0\text{, then }\tilde{h}_*^{{\rm com}} = \tilde{h}_*\geq \tilde{h}_*^{{\rm cap}} .
\end{equation}
We refer to Remark \ref*{Pre1:endremark},\ref*{Pre1:counterexampleh*cap<h*}) in 
\cite{Pre1} for an example of a graph for which the inequality in \eqref{capcomboukappa} is strict. On the other hand, if condition \eqref{capcondition} is fulfilled, then every connected closed set with finite capacity is bounded, and so $\tilde{h}_*^{{\rm cap}} \geq \tilde{h}_*$ by \eqref{defh*bou} and \eqref{defh*cap}. But by Lemma \ref{1stpartofmaincor}, for all $x_0\in{\tilde{\G}},$ if $\mathrm{{\rm cap}}(E^{\geq h}(x_0))<\infty,$ then $E^{\geq h}(x_0)$ is also compact, and so $\tilde{h}_*^{{\rm cap}} \geq \tilde{h}_*^{{\rm com}} .$ Thus, we obtain that
\begin{equation}
\label{capcomboucap}
\text{ if }\G\text{ is a transient graph verifying  \eqref{capcondition} then, }\tilde{h}_*^{{\rm com}} =\tilde{h}_*^{{\rm cap}} \geq \tilde{h}_*.
\end{equation}
In particular, if $\G$ satisfies \eqref{capcondition} and $\kappa\equiv 0,$ then from \eqref{capcomboucap} and \eqref{capcomboukappa} it is clear that the three critical parameters $\tilde{h}_*^{{\rm com}} ,$ $\tilde{h}_*$ and $\tilde{h}_*^{{\rm cap}} $ coincide; hence, in this case, in order to 
prove that they are equal to zero,
it is sufficient to show that one of them is non-negative while another one is non-positive. Our first main result provides such a statement, without any further assumption on $\mathcal{G}$ (recall our setup from above \eqref{deflambdarho}).


\begin{The}
	\label{mainresult}
	Let $\G$ be a transient weighted graph. For each $x_0\in\tilde{\G}$ and $h\geq0,$ the random variable $\mathrm{{\rm cap}}({E}^{\geq h}(x_0))$ is $\P^G$-a.s.\ finite, and for each $h<0$ the level set $E^{\geq h}(x_0)$ of $x_0$ is non-compact with positive probability. 
\end{The}
The proof of Theorem \ref{mainresult} appears over the next two sections. Note that the fact that $E^{\geq h}(x_0)$ is non-compact with positive probability for all $h<0$ could alternatively be obtained from the Markov property \eqref{Markov} similarly as in \cite{MR914444}, see also the Appendix of \cite{MR3765885} for details, or from the isomorphism \eqref{usualiso}, see \eqref{ifhkill<1thenh_*>0} and above. Here, we will obtain it as a direct consequence of our methods. In particular, Theorem \ref{mainresult} implies $\tilde{h}_*^{{\rm cap}} \leq0$ and $\tilde{h}_*^{{\rm com}} \geq0.$ Thus, together with Lemma \ref{1stpartofmaincor}, \eqref{capcomboukappa} and \eqref{capcomboucap}, Theorem \ref{mainresult} has the following immediate

\begin{Cor}
	\label{maincor} Let $\G$ be a transient weighted graph. 
	\begin{enumerate}[label={\arabic*)}]
	\item
	If $\G$ satisfies \eqref{capcondition},
	then \eqref{eq:0bounded} holds and $\tilde{h}_*^{{\rm com}} =\tilde{h}_*^{{\rm cap}} =0 \, ( \geq \tilde{h}_*).$ 
	\item
	If $\kappa\equiv0,$ then for each $h<0,$ the level set $E^{\geq h}(x_0)$ of $x_0$ is unbounded with positive probability; hence $(\tilde{h}_*^{{\rm com}} =)\, \tilde{h}_*\geq0.$ 
	\end{enumerate}
Therefore, if $\G$ satisfies \eqref{capcondition} and $\kappa\equiv 0,$ then $\tilde{h}_*=\tilde{h}_*^{{\rm com}} =\tilde{h}_*^{{\rm cap}} =0.$  
\end{Cor}

Notice that Theorem \ref{mainresult} and Corollary \ref{maincor} immediately imply item \ref{maint1} of Theorem~\ref{T:main}. We now comment on Theorem \ref{mainresult} and Corollary \ref{maincor}, and first elaborate on the condition \eqref{capcondition}, which is central 
 in obtaining $\tilde{h}_*=0$. Further comments on Theorem \ref{mainresult} and Corollary~\ref{maincor} are collected below in Remark \ref{R:mainresults1}. 

 The following lemma supplies a large class of graphs for which \eqref{capcondition} holds. In particular, by means of this lemma, Corollary \ref{maincor} generalizes all previously known results about $\tilde{h}_*=0$ (see below Theorem \ref{T:main} for a list). We highlight item~$2)$ of Lemma \ref{equivcapcondition}, comprising the condition \eqref{gbounded} which is sufficient for \eqref{capcondition} but stated only in terms of the Green function on $\G,$ and thus can be easier to verify. It implies for instance that any vertex-transitive graph verifies \eqref{capcondition}. 
Part $3)$ below accounts for the trees studied in \cite{MR3765885} and shows that Proposition 2.2 in \cite{MR3765885} can be seen as direct consequence of Corollary \ref{maincor},1); see also the discussion following Theorem \ref{T:main}.

\begin{Lemme}[Criteria for \eqref{capcondition}]
\label{equivcapcondition}
$\quad$

\begin{enumerate}[label={\arabic*)}]
\itemsep-0.5em 
\item \label{equivcapcondition_a}
Condition \eqref{capcondition} holds true if and only if
\begin{equation}
    \label{capconditiondis}
        \mathrm{{\rm cap}}(A)=\infty\text{ for all infinite and connected sets }A\subset {G}.
\end{equation}
\item \label{equivcapcondition_b} If
\vspace{-0.5em}
\begin{align}
&   \label{gbounded}  \begin{array}{l}
    \text{there exists }g_0<\infty\text{ such that }\{x\in{{G}}:\,g(x,x)>g_0\}
    \\\text{has no unbounded connected component}
    \end{array}
\end{align}
then condition \eqref{capcondition} is verified for $\G.$ In particular, if $\G$ is vertex-transitive, \eqref{capcondition} holds.\\[-0.5em]
\item \label{capconditionontrees}
Let $\mathbb{T}$ be a transient tree with zero killing measure and unit weights and denote by $R_x^{\infty}$ the effective resistance between $x$ and $\infty$ in $\mathbb{T}_x$, the sub-tree of $\mathbb{T}$ consisting only of $x$ and its descendents (relative to a base point $x_0 \in \mathbb{T}$). If $\{x\in{\mathbb{T}}:\,R_x^{\infty}>A\}$ only has bounded connected components for some $A>0,$ then \eqref{capcondition} is verified.
\end{enumerate}
\end{Lemme}

Lemma \ref{equivcapcondition} is proved in Appendix \ref{subsec:capandkappa=0}. We proceed to make further comments around Theorem~\ref{mainresult}, Corollary~\ref{maincor} and Lemma \ref{equivcapcondition}.

\begin{Rk}\label{R:mainresults1}
\begin{enumerate}[label={\arabic*)}]
\item\label{finitegraphs} In order to develop an intuition for the results of Theorem \ref{mainresult} and Corollary~\ref{maincor}, consider the case where $\G$ is a finite transient graph. 
Recall that for $x\in{G}$ such that $\kappa_x>0$ (such $x$ necessarily exists when $\G$ is finite and transient) the field $\phi$ on $I_x,$ conditionally on $\phi_x,$ has the same law as a Brownian bridge of length $\rho_x<\infty$ between $\phi_x$ and $0$ of a Brownian motion with variance $2$ at time $1,$ see the discussion below \eqref{Markov}. Therefore, for all $h<0,$  we have that $\P^G(\phi_y\geq h\text{ for all }y\in{I_x})>0,$ and since $I_x$ is non-compact, we obtain $\tilde{h}_*^{{\rm com}} \geq0.$ Now similarly if $h\geq0,$ then $\P^G(\phi_y\geq h\text{ for all }y\in{I_x})=0$ for all $x\in{G},$ and since $G$ is finite, it follows that $\tilde{h}_*^{{\rm com}} \leq0.$ Since \eqref{capcondition} is trivially verified on finite graphs, we thus have by \eqref{capcomboucap} that $\tilde{h}_*^{{\rm com}} =\tilde{h}_*^{{\rm cap}} =0.$ Note, however, that trivially $\tilde{h}_*=-\infty$ since there are no unbounded sets on finite graphs, and so the inequality in \eqref{capcomboucap} can be strict. In fact, the situation $0= \tilde{h}_*^{{\rm com}} =\tilde{h}_*^{{\rm cap}} > \tilde{h}_* \geq -\infty $ is emblematic of graphs with sub-exponential volume growth and (say) a uniform killing measure, and one typically has both strict inequalities $0> \tilde{h}_* > -\infty $ when $\mathcal{G}$ is infinite, see Corollary \ref*{Pre1:Cor:h_*<0} and Remark \ref*{Pre1:linkwithfinitary},\ref*{Pre1:h_*u_*finite}) in
\cite{Pre1}.  

\item We refer to Proposition \ref*{Pre1:h*infinity} in 
\cite{Pre1} for an example of a graph for which \eqref{capcondition} is not satisfied, and $\tilde{h}_*^{{\rm cap}} \le 0$ (necessarily by Theorem \ref{mainresult}) yet $\tilde{h}_*^{{\rm com}} =\tilde{h}_*=\infty$ -- in particular, this is a further example where the critical parameters do not coincide.
\item \label{signwithoutcap} We now construct an example of a graph not fulfilling \eqref{capcondition}, but for which we still have $\tilde{h}_*=\tilde{h}_*^{{\rm cap}} =\tilde{h}_*^{{\rm com}} =0$ (and therefore, as will turn out, \eqref{eq:0bounded} holds, cf.\ Corollary \ref{percoath*} below, or the first equivalence in Theorem \ref{T:main},\ref{maint3}). Consider a graph $\G$ with $\kappa\equiv0$ except possibly at $x \in G$, where $\kappa_x \in [0,\infty)$. Let $A \subset I_x $ be an infinite sequence converging towards the open end of $I_x$, and, simultaneously interpreting $A$ as the set given by the values of $A$, consider  $\mathcal{G}^A$ the graph given by Lemma \ref{GA}. If $\mathcal{G}\stackrel{\text{def.}}{=}\Z^3$ with unit weights and $\kappa \equiv 0$, then noting that $(\tilde{\mathcal{G}}^A)\setminus\bigcup_{x\in{A}}I_x$ can be identified with $\tilde{\mathcal{G}}$ (see \eqref{eq:GAsubsetG} and below \eqref{eq:enh1.1}), it readily follows that $\tilde{h}_*=\tilde{h}_*^{{\rm cap}} =\tilde{h}_*^{{\rm com}} =0$ on $\widetilde{\mathcal{G}}^A$. This chain of equalities follows (with a moment's thought) from the corresponding one on $\tilde{\mathcal{G}}$, where it holds by Corollary \ref{maincor}, for instance using Lemma \ref{equivcapcondition},ii) to argue that \eqref{capcondition} holds on $\tilde{\mathcal{G}}$. But for $A_n$ finite with $A_n \nearrow A$, the capacity of $A_n$ is supported on at most two points, whence $\text{cap}(A)< \infty$, by \eqref{defcapinfinity}. In particular, $\G^A$ does not fulfill \eqref{capcondition}. 

The previous example remains instructive if one considers instead $\mathcal{G}$ a finite graph and $\kappa_x>0$, in order to appreciate the difference between $\tilde{h}_*$ and $\tilde{h}_*^{{\rm com}}$. With $A$ as above, one has $\tilde{h}_*^{{\rm com}}({\mathcal{G}}), \tilde{h}_*^{{\rm com}}({\mathcal{G}}^A)\geq 0 $ by Theorem \ref{mainresult}. On the other hand, $\tilde{h}_*({\mathcal{G}})=-\infty$ since $\mathcal{G}$ is finite, but $\tilde{h}_*({\mathcal{G}}^A)\geq 0$ by Corollary \ref{maincor},ii) since $\kappa^A\equiv 0$. This shows that $\tilde{h}_*$ really depends on the choice of base graph $\mathcal{G}$ and not only on $\widetilde{\mathcal{G}}.$ We refer to Proposition \ref*{Pre1:Z20counterexample} in 
\cite{Pre1} for a less trivial example of a graph verifying \eqref{eq:0bounded} but not \eqref{capcondition}.

\item 
An interesting direct consequence of Corollary \ref{maincor} concerns ${\mathcal{L}}_{\alpha}$, the discrete (Poissonian) loop soup at intensity parameter $\alpha > 0$ (we refer to \cite{MR3502602} for precise definitions). 

\begin{Cor}\label{Cor:loopsoups}
Let $\G$ be a transient weighted graph such that \eqref{capcondition} holds. Then $\mathcal{L}_{1/2}$ a.s.~consists of finite clusters only.
\end{Cor}
\begin{proof}
If $\G$ satisfies \eqref{capcondition}, then by Corollary \ref{maincor}, i) and the symmetry and continuity of $\varphi$, the set $\{x\in{\tilde{\G}}:|\phi_x|>0\}$ only contains compact connected components. Hence, by Theorem 1 in \cite{MR3502602}, the loop soup $\tilde{\mathcal{L}}_{1/2}$ on $\tilde{\G}$ only contains compact connected components on which its field of local times is positive. A fortiori, $\mathcal{L}_{1/2}$ only consists of finite clusters. 
\end{proof}
\item 
The condition \eqref{gbounded} is strictly stronger than the condition \eqref{capcondition}. Indeed, consider $\G$ a rooted $(d+1)$-regular tree, with weights $1/(n+1)$ for each edge between a vertex at generation $n$ and one of its children at generation $n+1,$ and zero killing measure. Then $g(x,x)\geq n+1$ for each $x$ in generation $n,$ and so \eqref{gbounded} does not hold. On the other hand, for each infinite connected subset $K$ of the tree having at most one vertex per generation, denoting by $K_n\subset K$ the subset of all points in $K$ having generation at most $n$, one sees that for $x\in{K}$ at generation $k$ and all $n \geq k$, the equilibrium measure of $K_n$ at $x$ is at least $c(k+1)^{-1}$ for some absolute constant $c=c(d),$ and so $\text{cap}(K)= \infty$ on account of \eqref{defcapinfinity}. Since any infinite connected set $A$ contains such $K$, \eqref{capcondition} follows using Lemma~\ref{equivcapcondition},\ref{equivcapcondition_a} and \eqref{capincreasing}. All in all, $\G$ verifies \eqref{capcondition} but not \eqref{gbounded}.

\end{enumerate}
\end{Rk}

Next, we investigate the random variable
$\text{cap}\big({E}^{\geq h}(x_0)\big), \text{ for }x_0 \in \tilde{\mathcal {G}}, \, h \in \mathbb R$ (see \eqref{defcapinfinity} for the definition of $\text{cap}(\cdot)$ in this context), which will play a central role throughout the remainder of this article.
\begin{The}
	\label{mainresultcap}
	Let $\G$ be a transient weighted graph. For all $x_0\in \tilde{\G}$ and $h\geq0,$ if $E^{\geq h}(x_0)$ is $\P^G$-a.s.\ bounded, then the random variable $\text{{\rm cap}}\big({E}^{\geq h}(x_0)\big)$ has moment generating function given by \eqref{eq:laplacecap} and density given by
	\begin{equation}
	\label{eq:hdensity}
	\rho_h(t)=\frac{1}{2\pi t\sqrt{g(x_0,x_0)(t-g(x_0,x_0)^{-1})}}\exp\Big(-\frac{h^2t}{2}\Big)1_{t\geq g(x_0,x_0)^{-1}}.
	\end{equation} 
	Furthermore, assuming only that $\G$ satisfies \eqref{capcondition}, one has for each $h\geq0$ and $x_0 \in \tilde{\mathcal{G}}$ that
		\begin{align}
		& \label{item:LawhTrue}
		 \eqref{eq:laplacecap} \text{ holds, and}\\ 
		&\label{lawforhnegative}
	\text{{\rm cap}}\big({E}^{\geq -h}(x_0)\big)1_{\text{{\rm cap}}({E}^{\geq -h}(x_0))\in{(0,\infty)}}\text{ has the same law as }\text{{\rm cap}}\big({E}^{\geq h}(x_0)\big) 1_{\phi_{x_0}{\geq h}}.
\end{align}
In particular,
	\begin{equation}
	\label{eq:capinfinity}
	\P^G\big(\text{{\rm cap}}\big({E}^{\geq -h}(x_0)\big)=\infty\big)=\P^G(\phi_{x_0}\in{(-h,h)}).
	\end{equation}
\end{The}

\begin{Rk}
\label{R:lawh}
\begin{enumerate}[label={\arabic*)}]
\item \label{R:removecables2}
In case $\kappa\equiv0$ one can replace $\tilde{\G}$  in the statements of Theorems \ref{mainresult} and \ref{mainresultcap} by $\tilde{\G}^{\Ed}$, which corresponds to removing the edges $I_x,$ $x\in{G},$ from $\tilde{\G},$ see above \eqref{dirge} for notation, but not when $\kappa \not\equiv 0,$ see Remark \ref{R:removecables1}.
\item When $\G$ is a finite graph, one can deduce \eqref{eq:capinfinity} directly from Corollary 1, (ii) in \cite{LW18} with constant boundary condition $h\geq0,$ since saying that the random pseudo-metric between $x_0$ and the boundary of $\G$ introduced therein is equal to $0,$ is equivalent to saying that $E^{\geq -h}(x_0)$ is non-compact, or equivalently has infinite capacity. The statement \eqref{eq:capinfinity} then follows by using the reflection principle and that the effective resistance between $x_0$ and the boundary of $\G$ is equal to $g(x_0,x_0).$ When $\G=\Z^d,$ $d\geq3,$  \eqref{eq:capinfinity} is equivalent to the statement in Theorem 3 of \cite{DiWi}.
\end{enumerate}
\end{Rk}

\medskip
The proof of Theorem \ref{mainresultcap} (along with that of Theorem \ref{mainresult}) is given in the next two sections. Our starting point for both proofs is the observation (see Proposition \ref{couplingimplytheorem} below) that, if true, the isomorphism \eqref{eqcouplingintergffszn} entails a great deal of information about the observables $\text{cap}\big({E}^{\geq h}(x_0)\big),$ $h \in \R$. We use this observation on suitable finite-volume approximations of the free field on $\mathcal{G}$, which our setup naturally allows for (essentially obtained by iteratively reducing $\kappa$ starting from $\kappa=\infty$ outside a finite set). This is possible because \eqref{eqcouplingintergffszn} can be shown to hold without further assumptions on finite graphs. The condition \eqref{capcondition} then provides a very efficient criterion in order to avoid losing too much information when passing to the limit (in particular, one retains \eqref{eq:laplacecap}), thus yielding \eqref{item:LawhTrue}--\eqref{eq:capinfinity}. In a sense, the first part of Theorem \ref{mainresult} describes the information that survives in the limit \textit{without} any further assumptions on $\mathcal{G}$.

As $\text{ \eqref{eq:laplacecap}}_{h \geq 0}$ is essentially derived from \eqref{eqcouplingintergffszn} on finite-volume approximations of $\tilde{\mathcal{G}}$, one naturally wonders how the validity of $\text{ \eqref{eq:laplacecap}}_{h \geq 0}$ compares to that of \eqref{eqcouplingintergffszn} on $\tilde{\mathcal{G}}$ itself. This is the object of our next main result, Theorem \ref{couplingintergff} below; see in particular \eqref{equivisom}. Addressing this question will require us proving that the full strength of \eqref{eqcouplingintergffszn} can be passed to the limit (which is rather more involved than what is required for the proof of Theorem \ref{mainresultcap}), and thereby obtain an isomorphism on $\tilde{\mathcal{G}}$, under suitable assumptions (namely \eqref{eq:0bounded} or \hyperref[eq:laplacecap]{(\emph{$\textnormal{Law}_0$})}). 

In order to state Theorem \ref{couplingintergff}, we introduce a variation \eqref{eqcouplingintergff} of the identity \eqref{eqcouplingintergffszn}, which will sometimes be more convenient to work with. The two are in fact equivalent, see \eqref{equivisom} and Corollary~\ref{isomequivisom'} below. The appeal of \eqref{eqcouplingintergff} is that it makes certain symmetries more apparent (see for instance Lemma~\ref{h-hsamelaw}). It will also naturally imply a certain discrete isomorphism on the base graph $\mathcal{G}$, see \eqref{eqcouplingintergffdis} below, interesting in its own right. 

The identity \eqref{eqcouplingintergff} involves additional randomness. We henceforth assume that, on a suitable extension $\tilde{\P}_{\tilde{\G}}$ of $\P^G_{\tilde{\G}}\otimes \P^{I}_{\tilde{\G}}$ (which we simply denote by $\tilde{\P}$ when there is no risk of ambiguity) there exists for each $u>0$ an additional process $(\sigma_x^u)_{x\in{\tilde{\G}}}\in{\{-1,1\}^{\tilde{\G}}},$ such that, conditionally on $(|\phi_x|)_{x\in{\tilde{\G}}}$ and $\omega_u,$
$\sigma^u$ is constant on each of the connected components of $\{x\in{\tilde{\G}}:\ 2\ell_{x,u}+\phi_x^2>0\},$ $\sigma^u_x=1$ for all $x\in{\I^u},$ and the values of $\sigma^u$ on each other cluster of $\{x\in{\tilde{\G}}:\ 2\ell_{x,u}+\phi_x^2>0\}$ are independent and uniformly distributed. For $x$ such that  $2\ell_{x,u}+\phi_x^2=0,$ the value of $\sigma_x^u$ will not play any role in what follows, and one can fix it arbitrarily (e.g.~to have the value $+1$). Recalling the definition of $\mathcal{C}_u$ from below \eqref{eqcouplingintergffszn}, it is clear that the clusters of $\{x\in{\tilde{\G}}:\ 2\ell_{x,u}+\phi_x^2>0\}$ are the union of the clusters of the interior of $\mathcal{C}_u$ and the clusters of $\{x\in{\tilde{\G}:|\phi_x|>0}\}\cap (\mathcal{C}_u)^c,$ and so one can equivalently define $\sigma^u$ as follows: $\sigma^u_x=1$ for all $x\in{\mathcal{C}_u},$ $\sigma^u$ is constant on each of the clusters of $\{x\in{\tilde{\G}}:\ |\phi_x|>0\}\cap(\mathcal{C}_u)^c,$ and its values on each cluster are independent and uniformly distributed. 
We will investigate the validity of the relation
\begin{equation}
\label{eqcouplingintergff} \tag{Isom'}
\begin{array}{l}
\text{for each }u>0, \text{ the field }
\displaystyle \big(\sigma_x^u\sqrt{2\ell_{x,u}+\phi_x^2}\big)_{x\in{\tilde{\G}}}\text{ has the same}\\
\text{law under }\tilde{\P}
\text{ as the field }\big(\phi_x+\sqrt{2u}\big)_{x\in{\tilde{\G}}}\text{ under }\P^G.
\end{array}
\end{equation}
It is then an easy matter to see that \eqref{eqcouplingintergffszn} and \eqref{eqcouplingintergff} are equivalent, see Lemma \ref{isomequivisom'} below. Let
$ p_e^{\G}:\R^{{G}}\times[0,\infty)^{{G}}\rightarrow[0,1]$ for $e = \{x,y\} \in{E}$, and similarly $ p_x^{u,\G}$, $x\in G$, be defined by
	\begin{align}
	&\label{defpe}
	p_e (f,g) \equiv p_e^{\G}(f,g)=\exp\Big(-\lambda_{x,y}\big(f(x)f(y)+\sqrt{(f(x)^2+2g(x))(f(y)^2+2g(y)})\big)\Big), \\
	&	\label{defpx}
	p_x(f,g) \equiv p_x^{u,\G}(f,g)=\exp\Big(-\kappa_x\sqrt{2u(f(x)^2+2g(x))}\Big).
	\end{align} 
Our last main result is the following theorem, which is proved in Section \ref{sec:iso}.

\begin{The}
	\label{couplingintergff}
	Let $\G$ be a transient weighted graph. Then
	\begin{equation}
	\label{equivisom}
	\hyperref[eq:laplacecap]{(\emph{\text{Law}}_0)}\Longleftrightarrow \eqref{eq:laplacecap}_{h>0} \Longleftrightarrow \eqref{eqcouplingintergffszn} \Longleftrightarrow \eqref{eqcouplingintergff}.
	\end{equation}
	Moreover, defining for any $u>0$ on a suitable extension $\widehat{\P}$ of $\P^G\otimes\P^I$ a random set $\hat{\mathcal{E}}_u\subset E\cup G$ such that, conditionally on $(\phi_x)_{x\in{G}}$ and $\hat{\omega}_u,$ the set $\hat{\mathcal{E}}_u$ contains each edge and vertex that is contained in $\I_E^u$ (see below \eqref{defIu} for notation), and it contains each additional edge and vertex $e\in{E\cup G}$ conditionally independently with probability $1-p_e(\phi,\ell_{.,u}),$ the following holds: If  any of the conditions in \eqref{equivisom} is fulfilled, with $\mathcal{E}_u\stackrel{\text{def.}}{=}\{e\in{E\cup{G}}:\,2\ell_{x,u}+\phi_x^2>0\text{ for all }x\in{I_e}\}$,
	\begin{equation}
	\label{eq:isomisomdis}
	\text{$\hat{\mathcal{E}}_u$ has the same law under $\hat{\P}$ as $\mathcal{E}_u$ under $\tilde{\P}.$}
	\end{equation}
	 In particular, if one defines (under $\widehat{\P}$) a process $(\hat{\sigma}_x^u)_{x\in{{G}}}\in\{-1,1\}^{G},$ such that, conditionally on $(\phi_x)_{x\in{G}},$ $\hat{\omega}_u$ and $\hat{\mathcal{E}}_u,$
	\begin{itemize}
	\item the process $\hat{\sigma}^u$ is constant on each of the clusters (of edges) induced by $\hat{\mathcal{E}}_u\cap E,$
	\item $\hat{\sigma}_x^u=1$ for all $x\in{(\I^u\cup\hat{\mathcal{E}}_u)\cap G},$ and
	\item
	 the values of $\hat{\sigma}^u$ on all other clusters are independent and uniformly distributed, 
	 \end{itemize}
	 then
	\begin{equation}
	\label{eqcouplingintergffdis}
	\big(\hat{\sigma}_x^u\sqrt{2\ell_{x,u}+\phi_x^2}\big)_{x\in{{G}}}\text{ has the same law under }\hat{\P} \text{ as }\big(\phi_x+\sqrt{2u}\big)_{x\in{{G}}}\text{ under }\P^G.
	\end{equation}
\end{The}

 \begin{Rk}\label{R:isomisom}
 \begin{enumerate}[label={\arabic*)}]
\item The conclusions of Theorem \ref{mainresultcap} in combination with \eqref{equivisom} yield the validity of \eqref{eqcouplingintergffszn} assuming  either \eqref{eq:0bounded} or \eqref{capcondition} only. 
\item The discrete isomorphism \eqref{eqcouplingintergffdis} bears similarities to the coupling derived in Theorem 1.bis of \cite{MR3502602} (see also \eqref{eqcouplingloopsgffdis} below) in the context of loop soups, as well as with the coupling derived in Theorem 8 of \cite{LuSaTa} in the context of Markov jump processes. Notice that by construction, see the definition of $\widehat{\mathcal{E}}_u$ and \eqref{defpe}, \eqref{defpx}, the coupling $\widehat{\P}$ yielding $(\hat{\sigma}_{x})_{x\in G}$ only requires information on $\G$, i.e., the reference to $\widetilde{\G}$ can be completely bypassed.
\item If $\h$ is a harmonic function on $\tilde{\G},$ one can define the notion of $\h$-transform of random interlacements, and an isomorphism between the $\h$-transform of random interlacements and the Gaussian free field on $\tilde{\G}$ similar to \eqref{eqcouplingintergffszn} holds, under the same conditions, see Theorem \ref*{Pre1:couplingintergffh} in \cite{Pre1} for details.
\item \label{eqcouplingintergffG-}One can also deduce from Theorem \ref{couplingintergff} another isomorphism on $\tilde{\G}^{\Ed},$ see Section \ref{S:I_x}. Let $\mathcal{E}_u^{\Ed}\subset\tilde{\G}^{\Ed}$ be a random set such that, conditionally on $(\phi_x)_{x\in{\tilde{\G}^{\Ed}}}$ and $\omega_u^{\tilde{\G}^{\Ed}},$ the trace of the random interlacement process $\omega_u$ on $\tilde{\G}^{\Ed},$ the set $\mathcal{E}_u^{\Ed}$ contains $\I^u\cap\tilde{\G}^{\Ed}$ and each additional vertex $x\in{G}$ conditionally independently with probability $1-p_x^{u,\G}(\phi,\ell_{\cdot,u})$ (or equivalently $1-p_x^{u,\G}(\phi,0)$). Let also $\mathcal{C}_u^{\Ed}$ be the closure of the union of the connected components of the sign clusters $\{x\in{\tilde{\G}^{\Ed}}:\,|\phi_x|>0\}$ intersecting $\mathcal{E}_u^{\Ed}.$ Then the isomorphism obtained by replacing $\tilde{\G}$ by $\tilde{\G}^{\Ed}$ and $\mathcal{C}_u$ by $\mathcal{C}_u^{\Ed}$ in \eqref{eqcouplingintergffszn} is also equivalent to any of the conditions in \eqref{equivisom}. In particular, if $\kappa\equiv0,$ then $\mathcal{C}_u^{\Ed}=\mathcal{C}_u\cap\tilde{\G}^-,$ and so the isomorphism \eqref{eqcouplingintergffszn} (or also \eqref{eq:laplacecap} in view of Lemma \ref{capged}) can be equivalently stated on $\tilde{\G}$ or $\tilde{\G}^{\Ed}.$

\item\label{R:symmetry} The conclusion \eqref{lawforhnegative} can a-posteriori be strengthened. Indeed, knowing that \eqref{eqcouplingintergff} holds (which follows from \eqref{item:LawhTrue} and \eqref{equivisom}), one easily shows that compact clusters in $E^{\geq h}$ and $E^{\geq -h}$ have the same law, for all $h >0$, see Lemma \ref{h-hsamelaw} below. In particular under~\eqref{eq:0bounded}, the clusters of $E^{\geq h}$ have the same law as the compact clusters of $E^{\geq -h},$ and so for all $x_0\in{\tilde{\G}}$
\begin{align}
\label{lawforhnegativecompact}
	\text{{\rm cap}}\big({E}^{\geq -h}(x_0)\big)1_{{E}^{\geq -h}(x_0)\text{ is compact},\phi_{x_0}\geq-h}\text{ has the same law as }\text{{\rm cap}}\big({E}^{\geq h}(x_0)\big) 1_{\phi_{x_0}{\geq h}},
\end{align}
whose law is described by \eqref{eq:laplacecap} in view of Theorem \ref{mainresultcap}. Contrary to \eqref{lawforhnegative}, the conclusion \eqref{lawforhnegativecompact} is however not sufficient to entirely describe the law of our variable of interest $\text{cap}({E}^{\geq -h}(x_0)).$ But if condition \eqref{capcondition} holds, then on account of Lemma \ref{1stpartofmaincor} ${E}^{\geq -h}(x_0)$ is compact if and only if $\text{cap}({E}^{\geq -h}(x_0))<\infty,$ and so \eqref{lawforhnegativecompact} is then equivalent to \eqref{lawforhnegative}. 

Similarly, with regards to~\eqref{eq:capinfinity}, using Lemma \ref{h-hsamelaw} (which applies under \eqref{eq:0bounded} by means of Theorems~\ref{mainresultcap} and~\ref{couplingintergff}), one finds that, under \eqref{eq:0bounded}, for all $h \geq 0$,
\begin{equation}
\label{eq:theta_short}
\begin{split}
\P^G({E}^{\geq -h}(x_0) \text{ is compact})&= \P^G(\varphi_0 \leq -h)+ \P^G( \emptyset \neq {E}^{\geq -h}(x_0) \text{ is compact})\\
&= \P^G(\varphi_0 \leq -h)+ \P^G( \emptyset \neq {E}^{\geq h}(x_0) \text{ is compact})\\
&= \P^G(\varphi_0 \leq -h)+\P^G(\varphi_0 \geq h),
\end{split}
\end{equation}
using \eqref{eq:0bounded} and Lemma~\ref{1stpartofmaincor} in the last step. In particular, one recovers ~\eqref{eq:capinfinity} from \eqref{eq:theta_short} in case \eqref{capcondition} holds. We further refer to Remark~\ref{approimplylawcap},\ref{stronglawnegative} regarding the symmetry of clusters in $E^{\geq h}$ and $E^{\geq -h}$ contained in a given compact set $K \subset \tilde{\G}$, which does not require \eqref{eqcouplingintergff} to hold.

\item Let us explain how to explicitly construct the process $\sigma$ on $\tilde{\G}$ in \eqref{eqcouplingintergff}. Let $(x_n)_{n\in\N}$ be a dense sequence in $\tilde{\G}$ and $(\sigma'_n)_{n\in\N}\in{\{-1,1\}^\N}$ be a sequence of independent and uniformly distributed random variables under  $\tilde{\P}.$  Let $m(x)$ be the smallest $n\in\N$ such that $x_n$ and $x$ are in the same cluster of $\{y\in{\tilde{\G}}:\ 2\ell_{y,u}+\phi_y^2>0\};$ since $(x_n)_{n\in\N}$ is dense and $y\mapsto 2\ell_{y,u}+\phi_y^2$ is continuous, we have that $m(x)<\infty$ once $2\ell_{x,u}+\phi_x^2>0.$ We then define $\sigma_x=\sigma'_{{m(x)}}$ if $\phi_x^2>0$ and $x\notin{\mathcal{C}_u},$ and $\sigma_x=1$ otherwise, which has the desired properties. As an aside, note that in the isomorphism \eqref{eqcouplingloopsgff} between loop soups and the Gaussian free field, one could also construct explicitly the law of the signs $\sigma$ by a similar procedure.
\end{enumerate}
 \end{Rk}
 
Let us now give several interesting consequences of Theorem \ref{couplingintergff}, as well as the usual isomorphism \eqref{usualiso}. By continuity of the Gaussian free field, as already noted in (5.3) and below in \cite{DrePreRod2}, one can easily deduce from \eqref{usualiso} that 
\begin{equation}
\label{couplingusualiso}
\begin{gathered}
\text{there exists a coupling between $\I^u$ and $\phi$ such that a.s.\ each connected component}\\\text{of }\I^u\text{ is either included in }\{x\in{\tilde{\G}}:\,\phi_x>-\sqrt{2u}\}\text{ or in }\{x\in{\tilde{\G}}:\,\phi_x<-\sqrt{2u}\}.
\end{gathered}
\end{equation} 
Moreover, if $\h_{\text{kill}}<1,$ see \eqref{defh0}, then each forwards trajectory of the random interlacement process has a positive probability to not be killed, and so $\I^u$ is unbounded with positive probability for all $u>0.$ Hence, we obtain that for all $u>0$ either $\{x\in{\tilde{\G}}:\,\phi_x>-\sqrt{2u}\}$ or $\{x\in{\tilde{\G}}:\,\phi_x<-\sqrt{2u}\}$ is unbounded with positive probability, and by symmetry of the Gaussian free field, it follows that \eqref{ifhkill<1thenh_*>0} holds. 

Note that this improves the result from Corollary \ref{maincor}, ii). However, the proof of \eqref{ifhkill<1thenh_*>0} relies on the isomorphism \eqref{usualiso} between random interlacements and the Gaussian free field on infinite graphs, whereas the proof of Corollary \ref{maincor}, ii) only relies on this isomorphism on finite graphs, or equivalently the second Ray-Knight theorem (see Theorem 2 in \cite{LuSaTa}), or alternatively on an argument based on the Markov property for the Gaussian free field from \cite{MR914444}, as explained below Theorem \ref{mainresult}.

The advantage of the isomorphism \eqref{eqcouplingintergffszn} is that when it holds, or equivalently \hyperref[eq:laplacecap]{{($\text{Law}_0$)}} by Theorem \ref{couplingintergff}, one can directly improve \eqref{couplingusualiso} to prove that
\begin{equation}
\label{couplingnewiso}
\text{there exists a coupling between $\I^u$ and $\phi$ such that a.s.\ }\I^u\subset\{x\in{\tilde{\G}}:\,\phi_x>-\sqrt{2u}\}.
\end{equation} 
In particular, by symmetry of the Gaussian free field, we obtain that there exists a coupling between $\V^u$ and $\phi$ such that $E^{\geq\sqrt{2u}}\subset \V^u,$ where $\V^u=(\I^u)^c$ is the vacant set of random interlacements, thus generalizing Theorem 3 in \cite{MR3502602} from $\Z^d$ to any graph satisfying \hyperref[eq:laplacecap]{{($\text{Law}_0$)}}, or simply \eqref{capcondition} by \eqref{item:LawhTrue}. We refer to \cite{MR3492939}, \cite{MR3765885} and \cite{DrePreRod2} for other applications of couplings similar to \eqref{couplingnewiso}. Another interesting consequence of Theorem \ref{couplingintergff} is the following for the value of $\tilde{h}_*.$

\begin{Cor}
	\label{dichotomy}
	Let $\G$ be a transient weighted graph satisfying \hyperref[eq:laplacecap]{\emph{($\text{Law}_0$)}}. Then either $\P^G$-a.s.\ the sign clusters of the Gaussian free field on $\tilde{\G}$ only contain compact connected components, or $E^{\geq h}$ contains for each $h\in\R$ at least one unbounded connected component with $\P^G$-positive probability. In particular, if \hyperref[eq:laplacecap]{\emph{($\text{Law}_0$)}} holds and $\h_{\text{kill}}<1,$ then by \eqref{ifhkill<1thenh_*>0}, $\tilde{h}_*=\tilde{h}_*^{{\rm com}} \in{\{0,\infty\}}.$
\end{Cor}

The proof of Corollary \ref{dichotomy} appears at the end of Section \ref{sec:iso}. We refer to 
\cite{Pre1} for an example of a graph satisfying $\h_{\text{kill}}<1,$ but for which $\tilde{h}_*=\tilde{h}_*^{{\rm com}} =\infty.$ Note however that we still have $\tilde{h}_*^{{\rm cap}} \leq0$ by Theorem \ref{mainresult}. In view of Corollary \ref{dichotomy}, an interesting open question is then whether a transient graph with $\tilde{h}_*\in{(0,\infty)},$ or $\tilde{h}_*^{{\rm com}} \in{(0,\infty)},$ exists or not.  Another interesting consequence of Corollary \ref{dichotomy} is that if $\tilde{h}_*=0,$ then the level sets of the Gaussian free field do no percolate at the critical point $h=0,$ as implied by the following: 
\begin{Cor}
	\label{percoath*}
	If $\mathcal{G}$ is a transient graph such that $\tilde{h}_*\leq0,$ then $E^{\geq0}$ contains only bounded connected components. 
\end{Cor}

 We refer to the end of Section \ref{sec:iso} for the proof of Corollary~\ref{percoath*}. We conclude this section with the short
 
 \begin{proof}[Proof of Theorem \ref{T:main}]
Theorem \ref{T:main},\ref{maint1} follows from the first conclusion of Theorem \ref{mainresult} and Corollary \ref{maincor},$i)$, 
The first equivalence in Theorem \ref{T:main},\ref{maint3} is a consequence of Corollary \ref{percoath*} (the reverse implication being immediate, see \eqref{defh*bou}). Finally, the implication $\hyperref[eq:0bounded]{(\emph{\textnormal{Sign}})}\Longrightarrow \hyperref[eq:laplacecap]{(\emph{\textnormal{Law}}_0)}$ is a consequence of the first conclusion of Theorem \ref{mainresultcap} and the remaining equivalences follow from Corollary \ref{percoath*} and \eqref{equivisom} in Theorem \ref{couplingintergff}. Finally, Theorem \ref{T:main},\ref{maint4} is implied by Corollary \ref{dichotomy}.
 \end{proof}

\section{Some preparation}
\label{sec:prep}
In this section, we prepare the ground for the proofs of Theorems \ref{mainresult} and \ref{mainresultcap}. Their proofs, given in the next section, combine three main ingredients, corresponding to Proposition \ref{couplingimplytheorem}, Lemma \ref{Theforfinite} and Lemma \ref{LemmaapproGFF} below. They also rely on a symmetry property implied by
\eqref{eqcouplingintergff}, stated in Lemma~\ref{h-hsamelaw}, which is of independent interest. These results will also be useful in Section~\ref{sec:iso} in the course of proving Theorem \ref{couplingintergff}, albeit in a different manner.

Our starting point, Proposition \ref{couplingimplytheorem} below, contains the key observation that \eqref{eq:laplacecap}$_{h\geq0}$ follows from the identity \eqref{eqcouplingintergff}, if assumed to hold. Lemma~\ref{Theforfinite} implies a version of the isomorphism \eqref{eqcouplingintergff}, valid on finite graphs (this result is in fact a consequence of the isomorphism theorems between loop soups and the Gaussian free field from \cite{MR3502602}, see also \eqref{eqcouplingloopsgff} below; the proof of  Lemma~\ref{Theforfinite} is given in Appendix \ref{App:isom}). Importantly, Lemma~\ref{Theforfinite} allows for Proposition \ref{couplingimplytheorem} to automatically apply in a finite setup. Finally, Lemma \ref{LemmaapproGFF} supplies a useful approximation scheme for $\varphi$ based on \eqref{eq:graphinclusion}, see \eqref{eq:def_approx} below, which entails the important limits \eqref{capGngeqcapG}, \eqref{capGngeqcapG'} from Corollary~\ref{Cor:limcap}.
With these results at hand, the proofs of Theorems \ref{mainresult} and \ref{mainresultcap} quickly follow. They appear in the next section.

Unless specified otherwise, we tacitly assume that $\mathcal{G}$ is a transient weighted graph (see above \eqref{deflambdarho} for our setup). We begin with the following technical lemma.

\begin{Lemme}
\label{E>hbar=E>=h}
For each $x_0\in{\tilde{\G}}$ and $h\in\R,$ defining $E^{>h}=\{y\in{\tilde{\G}}:\phi_y>h\}$ and $E^{>h}(x_0)=\{y\in{\tilde{\G}}:\,y\leftrightarrow x_0\text{ in }E^{>h}\},$ and denoting by $\overline{E^{>h}(x_0)}$ the closure of $E^{>h}(x_0),$ one has $$\overline{E^{>h}(x_0)}=E^{\geq h}(x_0)\ \P^G\text{-a.s.}$$
\end{Lemme}
\begin{proof}
Since $E^{\geq h}(x_0)$ is closed, it is clear that $\overline{E^{>h}(x_0)}\subset E^{\geq h}(x_0).$ Let us now fix some compact $K\subset\tilde{\G},$ let $E^{>h}_K(x_0)=\{y\in{\tilde{\G}}:\,y\leftrightarrow x_0\text{ in }E^{>h}\cap K\},$ and $\mathcal{K}$ be the set containing $\overline{E^{>h}_K(x_0)}$ as well as each $x\in{G}$ such that $\overline{I_{\{x,y\}}}\cap \overline{E^{>h}_K(x_0)}\neq\varnothing$ for some $y\sim x.$ In order to apply the Markov property \eqref{Markov} to the random compact $\mathcal{K},$ we first need to show that it is compatible. Let us thus fix some open set $O,$ and let us define $O'$ the set obtained from $O$ by removing $\overline{I_{\{x,y\}}}$ from $O$ for all $x\sim y$ such that $I_{\{x,y\}}\cap O\neq\varnothing$ and $x\notin O$ or $y\notin O.$ One then sees that $\mathcal{K}\subset O$ if and only if $\overline{E^{>h}_K(x_0)}\subset O'.$ Moreover, $\overline{E^{>h}_K(x_0)}\subset O'$ if and only if for every connected path $\pi$ from $x_0$ to $y\in{\partial O'},$ with $\pi$ closed in $x$ and open in $y,$  there exists $z\in{\pi}$ with $\phi_z\leq h.$ Therefore, the event $\overline{E^{>h}_K(x_0)}\subset O'$ is $\mathcal{A}_{O'}\subset \mathcal{A}_O$ measurable, and so $\mathcal{K}$ is compatible.

Let us now assume that $E^{\geq h}_K(x_0)\not\subset \overline{E^{>h}_K(x_0)}$. Hence, there exists a closed path $\pi\subset E^{\geq h}_K(x_0)$ starting in $x_0$ such that $\pi\not\subset \overline{E^{>h}_K(x_0)}.$ With probability one, we can moreover assume that $\phi\neq h$ on $G.$ Then by definition of $\K$ there exists an edge or vertex $e\in{E\cup G},$ $x\in{{I_e}\cap\partial E^{>h}_{K}(x_0)},$ with $x$ in the interior of $\pi,$ and, if $e\in{E},$ $y\in{\overline{I_e}\cap \partial\K}$ with $y\neq x.$ Since $\phi_x=h$ by continuity of $\varphi$, using the Markov property \eqref{Markov} and a similar reasoning as above (2.9) in \cite{DrePreRod}, one can show that when $e\in{E},$ conditionally on $\mathcal{A}_{\mathcal{\K}}^+,$ the law of $\phi$ on the edge between $x$ and $y$ is the same as the law of a Brownian bridge with variance $2$ at time $1,$ on the edge between $x$ and $y$ with value $h$ at $x$ and $\phi_y$ at $y.$ This Brownian bridge is a.s.\ strictly smaller than $h$ infinitely many times in any neighborhood of $x,$ and so a.s.\ $\phi<h$ infinitely many times in any neighborhood of $x,$ that is $x\in{\partial E^{\geq h}(x_0)}.$ If $e\in{G},$ one can prove similarly that $x\in{\partial E^{\geq h}(x_0)}$ since the law of $\phi$ on the edge between $x$ and the open end of $I_e$ is the same as the law of a Brownian bridge with variance $2$ at time $1$ between $\phi_x$ and $0.$  This is a contradiction since $x$ is in the interior of $\pi\subset E^{\geq h}_K(x_0),$ and so $E^{\geq h}_K(x_0)\subset \overline{E^{>h}_K(x_0)}\subset \overline{E^{>h}(x_0)}$ a.s. Taking a sequence of compacts $K=K_n$ increasing to $\tilde{\G},$ we conclude.
\end{proof}

\begin{Prop}
\label{couplingimplytheorem}
	Suppose \eqref{eqcouplingintergff} is verified on $\G$. Then \eqref{eq:laplacecap}$_{h\geq0}$ holds true. 
\end{Prop}
\begin{proof}
Let
\begin{align} \label{Eabs}
\begin{split}
\Sigma_h\stackrel{\text{def.}}{=}\{y\in{\tilde{\G}};\,|\phi_y- h|>0\}, \, \Sigma \equiv \Sigma_0 \text{ and }\Sigma(x)\stackrel{\text{def.}}{=}\{y\in{\tilde{\G}}:y\leftrightarrow x\text{ in }\Sigma \} \text{ for $x\in{\tilde{\G}}$}
\end{split}
\end{align}
(see below \eqref{Ehx0} for notation). 
We first consider the case $h=0,$ and 
the sets $\overline{\Sigma(x)},$ $x\in{\tilde{\G}},$ which are the closures of the sign clusters ${\Sigma(x)}$.  Note that if $\Sigma(x)\cap\I^u=\varnothing,$ then the cluster of $x$ in $\{y\in{\tilde{\G}}:2\ell_{y,u}+\phi_y^2>0\}$ is equal to $\Sigma(x)$ (both $\Sigma(x)$ and $\I^u$ are open) and so $\sigma^u_x=\pm1$ with conditional probability $\frac12$ given $(|\phi_x|)_{x\in{\tilde{\G}}}$ and $\omega_u$ under $\tilde{\P}$ (recall $\sigma^u$ as defined above Theorem \ref{couplingintergff}). On the other hand, if $\Sigma(x)\cap\I^u\neq\varnothing,$ then $x\leftrightarrow\I^u$ in  $\{y\in{\tilde{\G}}:2\ell_{y,u}+\phi_y^2>0\},$ and so $\sigma^u_x=1.$ As $E[\text{sign}(X+a)]=P(|X|< a)$ for any centered Gaussian variable $X$ and $a >0$, by \eqref{eqcouplingintergff}, \eqref{defIu} and the symmetry of the Gaussian free field, we thus obtain, for all $u>0$ and $x\in{\tilde{\G}}$,
\begin{align}\label{eq:lapTraf}
\begin{split}
2\P^G(\phi_x\geq\sqrt{2u})&=1-\E^G\big[\text{sign}(\phi_x+\sqrt{2u})\big]
=1-\tilde{\E}[\sigma_x^u]
\\&=1-\tilde{\P}\big(\Sigma(x)\cap\I^u\neq\varnothing\big)
=\E^G\left[\exp\left(-u\mathrm{{\rm cap}}\big(\overline{\Sigma(x)}\big)\right)\right].
\end{split}
\end{align}
Next, we note that by Lemma \ref{E>hbar=E>=h} for $h=0,$ 
$\P^G$-a.s., $\overline{\Sigma(x)}=E^{\geq0}(x)$ on $\{ \phi_x>0\}$.
Therefore, by symmetry of the Gaussian free field in combination with \eqref{eq:lapTraf} we thus have
\begin{align}
\label{eq:lawointer}
    \E^G\Big[\exp\big(-u\mathrm{{\rm cap}}\big(E^{\geq 0}(x)\big)\big)1_{\phi_x\geq0}\Big]&=\frac12\E^G\left[\exp\left(-u\mathrm{{\rm cap}}\big(\overline{\Sigma(x)}\big)\right)\right]
=\P^G(\phi_x\geq\sqrt{2u}),
\end{align}
which is ($\text{Law}_0$).

Let us now consider some $h>0,$ and let $u_0=h^2/2.$ We will reduce this to the case $h=0$. By the symmetry of the Gaussian free field, \eqref{eqcouplingintergff} and Lemma \ref{E>hbar=E>=h}, we have that $E^{\geq h}(x)$ has the same law under $\P^G$ as the closure of the connected component of $x$ in $\{y\in{\tilde{\G}}:\,\sigma^{u_0}_y=-1\}$ under $\tilde{\P},$  which is the law of the set that equals $\overline{\Sigma(x)}$ if $\I^{u_0}\cap {\Sigma(x)}=\varnothing$ and $\sigma_x=-1,$ and equals $\varnothing$ otherwise. Therefore, by \eqref{defIu} we have for all $u>0$
\begin{equation}
\label{eq:lawointer1}
\begin{split}
    &\E^G\Big[\exp\big(-u\mathrm{{\rm cap}}(E^{\geq h}(x))\big)1_{\phi_{x}\geq h}\Big]=\tilde{\E}\left[1_{\I^{u_0}\cap {\Sigma(x)}=\varnothing,\sigma^{u_0}_x=-1}\exp\left(-u\mathrm{{\rm cap}}(\overline{\Sigma(x)})\right)\right]
    \\&\qquad \qquad =\frac12\E^G\left[\exp\left(-(u+u_0)\mathrm{{\rm cap}}(\overline{\Sigma(x)})\right)\right] =\P^G\big(\phi_x\geq\sqrt{2u+h^2}\big),
\end{split}
\end{equation}
using \eqref{eq:lapTraf} in the last step.
\end{proof}

Next, we observe a symmetry property of compact clusters implied by \eqref{eqcouplingintergff}.

\begin{Lemme}
	\label{h-hsamelaw}
	Let $\G$ be a graph such that \eqref{eqcouplingintergff} holds. Then for all $h \geq0$, the compact clusters of  $E^{\geq -h}$ have the same law as the compact clusters of $E^{\geq h}.$ 
\end{Lemme}

\begin{proof}
If \eqref{eqcouplingintergff} holds, then by Lemma \ref{E>hbar=E>=h} the compact clusters of $E^{\geq -\sqrt{2u}}$ have the same law as the closure of the clusters 
of $\{x\in{\tilde{\G}}:\sigma_x^u=1\}$ whose closure is compact. Each cluster of $\I^u$ is non-compact, and so by definition of $\sigma^u,$ the compact clusters of $E^{\geq -\sqrt{2u}}$ have the same law as the closure of the clusters of $\Sigma$ (cf.\ \eqref{Eabs}) whose closure is compact, that do not intersect $\I^u$ and for which $\sigma^u=1$. By definition of $\sigma^u$, the law of these clusters of $\Sigma$ is unchanged if one retains all the previous properties but the last one and requires $\sigma^u=-1$ instead. But by \eqref{eqcouplingintergff}, 
the resulting clusters have the same law as those of $\{x\in{\tilde{\G}}: \varphi_x < - \sqrt{2u} \}$ whose closure is compact, i.e.\ by Lemma \ref{E>hbar=E>=h} the clusters whose closures are the compact clusters of $\{x\in{\tilde{\G}}:\phi_x\leq-\sqrt{2u}\}.$  Finally by the symmetry of the Gaussian free field, these closures have the same law as the compact clusters of $E^{\geq\sqrt{2u}}$.
\end{proof}

The proofs of our next two ingredients, Lemmas \ref{Theforfinite} and \ref{LemmaapproGFF} below, rely on certain aspects of Poissonian loop soups. This requires a small amount of notation, which we now introduce. We also review certain features of loop soups, which will be used in the sequel. Following e.g.~\cite{MR3238780}, \cite{MR2815763}, one defines a measure $\mu^L$ on loops in $\tilde{\G}$ with compact closure in $\tilde{\G}$ associated with $P^{\tilde{\G}}_x,$ $x\in{\tilde{\G}},$ and, under a suitable probability measure $\P^L= \P^L_{\tilde{\G}}$, for all $\alpha>0$ the loop soup $\tilde{\mathcal{L}}_{\alpha}$ with parameter $\alpha$ as the Poisson point process on the space of (compact) loops on $\tilde{\G}$ with intensity $\alpha\mu^L.$ We denote by $(L_x^{(\alpha)})_{x\in{\tilde{\G}}}$ its field of local times relative to $m$ on $\tilde{\G}$ (cf. above \eqref{eq:quadraticform}), which can be taken to be continuous, see Lemma 2.2 in \cite{MR3502602}. Moreover, we denote by $\mathcal{L}_{\alpha}$ the Poisson point process consisting of the trace on ${G}$ of each loop in $\tilde{\mathcal{L}}_{\alpha},$ which has the same law as the loop soup associated with $P_x^\G,$ see Section 2 of \cite{MR3502602} or Section 7.3 of \cite{MR3238780} for details. 
An important property of the loop soup $\tilde{\mathcal{L}}_{\alpha}$ is the restriction property, see Section 6 of \cite{MR3238780}: for all connected and open subsets $A$ of $\tilde{\G},$ if $\tilde{\mathcal{L}}_{\alpha}^A$ stands for the set of loops in $\tilde{\mathcal{L}}_{\alpha}$ which are entirely included in $A,$ then
\begin{equation}
\label{looprestriction}
\tilde{\mathcal{L}}_{\alpha}^A\text{ has the same law under }\P^L_{\tilde{\G}} \text{ as }\tilde{\mathcal{L}}_{\alpha}\text{ under }\P^L_{\tilde{\G}^A_{\infty}};
\end{equation}
here, $\G^A_{\infty}$ is the graph with the same vertices, edges and weights as $\G^{\partial A}$ (see Lemma \ref{GA}), but with killing measure equal to $\kappa$ on ${(G\cap A)\setminus\partial A},$ and equal to infinity on $\partial A\cup (G\cap A^c).$ I.e., for all $x\in{A},$ the diffusion $X$ under $P^{\tilde{\G}^A_{\infty}}_x$ has the same law as $X$ killed on exiting $A$ under $P^{\tilde{{\G}}}_x.$

When $\alpha=\frac12,$ the loop soup $\tilde{\mathcal{L}}_{1/2}$ is linked to the Gaussian free field on $\tilde{\G}$ via the following isomorphism, due to Lupu \cite{MR3502602}; see also Le Jan, Theorem 2 of \cite{MR2815763} for a similar identity regarding the square of the Gaussian free field on the discrete base graph $G$ (not including the sign of $\varphi$).
Introducing the shorthand $L_{\cdot}=L_{\cdot}^{(1/2)}$ for the local time field of $\tilde{\mathcal{L}}_{1/2}$ to simplify notation, let $\tilde{\P}^{L}_{\tilde{\G}}$ be a suitable extension of
$\P^{L}_{\tilde{\G}}$ carrying a process $(\sigma_x)_{x\in{\tilde{\G}}}\in{\{-1,1\}^{\tilde{\G}}}$ such that, conditionally on $\tilde{\mathcal{L}}_{1/2},$ $\sigma$ is constant on each cluster of $\{x\in{\tilde{\G}}:L_x>0\}$, and its values on each cluster are independent and uniformly distributed. Then
	\begin{equation}
	\label{eqcouplingloopsgff}
	\text{ under }\tilde{\P}^L_{\tilde{\G}} \text{ the law of } \big(\sigma_x\sqrt{2L_x}\big)_{x\in{\tilde{\G}}}\text{ is }\P^G_{\tilde{\G}};
	\end{equation}
the measure $\tilde{\P}^L_{\tilde{\G}}$ is essentially the coupling constructed in Proposition~2.1 of \cite{MR3502602}, where the (explicit) law of $\sigma$ on $\tilde{\G}$ follows from a version of Lemma 3.2 in \cite{MR3502602} on $\tilde{\G}$ rather than $\tilde{\G}^{\Ed},$ cf. above \eqref{dirge}.

The identity \eqref{eqcouplingloopsgff} also comes with the following discrete version. Define (still under $\tilde{\P}^{L}_{\tilde{\G}}$) a random subset $\hat{\mathcal{E}}$ of $E$ such that, conditionally on $\mathcal{L}_{\frac12},$  $\hat{\mathcal{E}}$ contains each edge crossed by some loop in $\mathcal{L}_{\frac12},$ and each additional edge $e\in{E}$ conditionally independently with probability $1-p_e^{\G}(\sqrt{L}),$ with $p_e^{\G}$ as given by \eqref{defpef}. Then 
\begin{equation}
\label{hatEsameasE}
\text{$\hat{\mathcal{E}}$ has the same law under $\tilde{\P}^{L}_{\tilde{\G}}(\cdot\,|\,\mathcal{L}_{\frac12})$ as $\mathcal{E}\stackrel{\text{def.}}{=}\{e\in{E}:\,L_x>0\text{ for all }x\in{I_e}\}$ under ${\P}^{L}_{\tilde{\G}}(\cdot\,|\,\mathcal{L}_{\frac12}).$}
\end{equation}
 In particular, if we define a process $(\hat{\sigma}_x)_{x\in{{G}}}\in\{-1,1\}^\G,$ such that, conditionally on $\mathcal{L}_{\frac12},$ and $\hat{\mathcal{E}},$ $\hat{\sigma}$ is constant on each of the (discrete) clusters induced by $\hat{\mathcal{E}}$ and its values on each cluster are independent and uniformly distributed, then
	\begin{equation}
	\label{eqcouplingloopsgffdis}
	\big(\hat{\sigma}_x\sqrt{2L_x}\big)_{x\in{{{G}}}}\text{ has the same law under }\tilde{\P}^L_{\tilde{\G}} \text{ as }(\phi_x)_{x\in{G}}\text{ under }\P^G_{\tilde{\G}}
	\end{equation}
(Corollary 3.6 in \cite{MR3502602} provides \eqref{hatEsameasE}, and one can then directly derive \eqref{eqcouplingloopsgffdis}, see Theorem 1.bis in \cite{MR3502602}). The identity \eqref{eqcouplingloopsgff} is an analogue in the context of loop soups of the relation \eqref{eqcouplingintergff} for interlacements (a similar analogy can be drawn between \eqref{eqcouplingloopsgffdis} and \eqref{eqcouplingintergffdis}). In particular, the following holds on finite graphs, i.e.~on graphs $\G=(\overline{G},\bar\lambda,\bar\kappa)$ such that $\{x \in \overline{G} : \bar\kappa_x<\infty \}$ is finite (note that this implies that the induced graph $(G,\lambda, \kappa)$ has finite vertex set $G$, cf.~\eqref{eq:defGfinite}).
\begin{Lemme}
	\label{Theforfinite}
	If $\G$ is a finite transient weighted graph, then \eqref{eqcouplingintergff} holds. Moreover, conditionally on $\hat{\omega}_u$ and $(\phi_x)_{x\in{{G}}},$ the family $\{e\in{\mathcal{E}_u}\},$ $e\in{E\cup{G}}$ (defined above \eqref{eq:isomisomdis}) is independent, and for all $e\in{E\cup{G}}$
	\begin{equation}
	\label{lawofEu}
	\tilde{\P}(e\in{\mathcal{E}_u}\,|\,\hat{\omega}_u,(\phi_x)_{x\in G})=1_{e\in{\I_E^u}}\vee (1-p_e(\phi,\ell_{.,u})).
	\end{equation}
\end{Lemme}
For completeness, we have included the proof of Lemma \ref{Theforfinite} in Appendix \ref{App:isom}. We briefly sketch the proof here. To deduce \eqref{eqcouplingintergff}, one essentially considers the decomposition   $\tilde{\mathcal{L}}_{1/2}=\tilde{\mathcal{L}}_{1/2}^{\, \textnormal{in}}+\widetilde{\mathcal{L}}_{1/2}^{\, *}$ of the loop soup on the cable system $\tilde{\G}^*$ of a suitable one-point compactification ${\G}^*={\G} \cup \{ x_*\}$ of ${\G}$ (with killing at $x_*$, so $\G^{*}$ is transient), into the `interior' loops constituting $\tilde{\mathcal{L}}_{1/2}^{\, \textnormal{in}}$ which never hit $x_*,$ and the loops $\widetilde{\mathcal{L}}_{1/2}^{\, *}$ which contain $x_*$. The two processes are independent. Inserting the corresponding decomposition of the local times $L_{\cdot}$ of $\tilde{\mathcal{L}}_{1/2}$ into \eqref{eqcouplingloopsgff} (applied on $\tilde{\G}_*$), one can then generate in law the field $\sigma_{\cdot}^u\sqrt{2\ell_{\cdot,u}+\phi_{\cdot}^2}$ appearing in \eqref{eqcouplingintergff} by suitable conditioning, and witnesses that this conditioning causes a global shift by $\sqrt{2u}$ in \eqref{eqcouplingloopsgff}. Roughly speaking, the local times of $\tilde{\mathcal{L}}_{1/2}^{\, \textnormal{in}}$ generate $\phi_{\cdot}^2/2$ in this procedure by \eqref{looprestriction} and \eqref{eqcouplingloopsgff}, whereas the local times of $\widetilde{\mathcal{L}}_{1/2}^{\, *}$ give rise to $\ell_{\cdot,u}$; see also \cite{LuSaTa}, or Section 2 of \cite{MR3417508}, for similar ideas to deduce the second Ray-Knight theorem from \eqref{eqcouplingloopsgff}, which is related to the interlacement by concatenating the trajectories contributing to $\ell_{\cdot,u}$ to represent the successive excursions of a single diffusion $X_{\cdot \wedge \tau_u}$ under $P_{x_*}^{\tilde{\G}^*}$ stopped at $\tau_u =\inf \{t \geq 0: \ell_{x_*}(t) \geq u\}$. The conditional law in \eqref{lawofEu} is then obtained by following ideas of  \cite{LuSaTa}, Section 2.5. 

\begin{Rk}
The proof of Lemma \ref{Theforfinite} delineated above uses the isomorphism \eqref{eqcouplingloopsgff} relating loop soups and the Gaussian free field. Similarly to the proof of Theorem 2.4 of \cite{MR3492939}, one could alternatively use the Markov property \eqref{Markov} to prove that \eqref{eqcouplingintergffszn} (which is easily seen to be equivalent to \eqref{eqcouplingintergff}, see Lemma \ref{isomequivisom'} below) holds on any finite transient graph (or more generally on any transient graph with bounded Green function such that \eqref{eq:0bounded} holds). However, this approach does not directly provide the discrete isomorphism described by \eqref{lawofEu}. 
\end{Rk}

\medskip
We proceed to state the third ingredient, Lemma \ref{LemmaapproGFF} below, which supplies a way to approximate the Gaussian free field on any transient graph $\G$ by Gaussian free fields on finite graphs. The following definition is key. For a given graph $\G=(\overline{G}, \overline{\lambda}, \overline{\kappa})$, we say that
\begin{equation}
\label{eq:def_approx}
\begin{split}
&\text{a sequence of graphs $\G_n$ increases to $\G$ if $\G_n=(\overline{G},\overline{\lambda}, \overline{\kappa}^{(n)})$ for a sequence}\\&\text{$\overline{\kappa}^{(n)} \subset[0,\infty]^{\overline{G}}$ of killing measures such that $\overline{\kappa}^{(n)}_x \searrow \overline{\kappa}_x$ as $n \to \infty$ for all $x\in{\overline{G}}.$}
\end{split}
\end{equation}
In particular, we will be interested in finite-volume approximations of $\G$, for which $\overline{\kappa}^{(n)}=\infty$ outside of a finite set $U_n$ for every $n$, with $U_n$ exhausting $\overline{G}$ as $n \to \infty$. The graphs $\G_n$ thus considered are finite (in the sense defined above Lemma~\ref{Theforfinite}).

Due to the observations made around \eqref{eq:graphinclusion}, for $\G_n$ as in \eqref{eq:def_approx}, we can view $\tilde{\G}_n$ as a subset of $\tilde{\G}$ such that the sequence $\tilde{\G}_n$ increases to $\tilde{\G}$ and such that for each compact $K \subset \tilde{\G}$ we have $K\subset\tilde{\G}_n$ for large enough $n$.
\begin{Lemme}
	\label{LemmaapproGFF}
	Let $\G$ be a transient weighted graph, and let $\G_n,$ $n\in\N,$ be a sequence of transient weighted graphs increasing to $\G_{\infty}=\G.$ There exists a probability space $(\Omega,\mathcal{F},\P)$ on which the processes $(\phi^{(n)}_x)_{x\in{\tilde{\G}_n}},$ $n\in\N,$ and $(\phi^{(\infty)}_x)_{x\in{\tilde{\G}}}$ are defined, with the following properties: 
	\begin{align}
&\label{eq:GFFappro1}\text{for all $n\in\N\cup\{\infty\}$, $(\phi_{x}^{(n)})_{x\in{\tilde{\G}_n}}$ has law $\P_{\tilde{\G}_n}^G$;}\\
&\P\text{-a.s.\ for all compact $K\subset\tilde{\G},$ one has $\phi_x^{(n)}=\phi_x^{(\infty)}$ for $x\in{K}$ and $n$ large enough.}
\label{eq:equalityKapprox}
	\end{align}
\end{Lemme}
\begin{proof}
	Let $(\Omega,\mathcal{F},\P)$ be a probability space carrying a process $\tilde{\mathcal{L}}^{(\infty)}$ with the same law as $\tilde{\mathcal{L}}_{1/2}$ under $\P_{\tilde{\G}}^L$ (for instance one can choose $\P = \P_{\tilde{\G}}^L$). For each $n\in\N$ we denote by $({L}_x^{(n)})_{x\in{\tilde{\G}_n}}$ the accumulated local times of those loops in $\tilde{\mathcal{L}}^{(\infty)}$ which are entirely contained in the open set $\tilde{\G}_n\subset\tilde{\G}.$ One can clearly identify $\tilde{\G}_n$ with $\tilde{\G}_\infty^{\tilde{\G}_n},$ and by \eqref{looprestriction}, the law of  $({L}_x^{(n)})_{x\in{ \tilde{\G}_n}}$ is the same as the law of $(L_x)_{x\in{\tilde{\G}}}$ under $\P_{\tilde{\G}_n}^L.$ Moreover, for each $x\in{\tilde{\G}},$ the sequence $L_x^{(n)},$ $n\in\N,$ is increasing, and we denote by $L_x^{(\infty)}$ its limit. Since each loop of $\tilde{\mathcal{L}}^{(\infty)}$ is relatively compact, it is contained in $\tilde{\G}_n$ for $n$ large enough, and so $(L_x^{(\infty)})_{x\in{\tilde{\G}}}$ equals the total local times of the loops in $\tilde{\mathcal{L}}^{(\infty)},$ whence
\begin{equation}
\label{eq:addlaw1}
L_{\cdot}^{(\infty)} =\lim_n \uparrow L_{\cdot}^{(n)} \stackrel{\text{law}}{=} L_{\cdot}
\end{equation}	
where $L_{\cdot}$ is the occupation time field of $\tilde{\mathcal{L}}_{1/2}$ (on $\tilde{\G}$).

	For each $n\in\N,$ let $(\mathcal{A}_p^{(n)})_{p\in\N}$ be some enumeration of the countably many clusters of $\{ L^{(n)}>0 \} (=\{x\in{\tilde{\G}}:L_x^{(n)}>0\} \subset \tilde{\G}_n),$ and let $(\sigma_p)_{p\in\N}\in{\{-1,1\}^\N}$ be an independent sequence of uniformly distributed random variables. For each $n\in\N$ and $x\in{\tilde{\G}_n}$ we define $E_n^{\mathcal{L}}(x)=\{y\in{\tilde{\G}_n}:x\leftrightarrow y\text{ in } \{ L^{(n)}>0 \} \},$ and if $L_x^{(n)}\neq0,$ we denote by $k_n(x)\in\{1,\dots,n\}$ the smallest index $k$ such that $\tilde{\G}_k$ intersects the cluster of $x$ in $\{L^{(n)}>0\}$, i.e.~$E_n^{\mathcal{L}}(x)\cap \tilde{\G}_{k_n(x)}\neq\varnothing$ and $E_n^{\mathcal{L}}(x)\cap \tilde{\G}_{k_n(x)-1}=\varnothing,$ with the convention $\tilde{\G}_0=\varnothing.$ 
	
	We also define $p_n(x)=\inf\{p\in\N:\,\mathcal{A}_p^{(k_n(x))}\subset E^{\mathcal{L}}_n(x)\},$ with the convention $\inf\varnothing=+\infty.$ Note that since $L^{(n)}_x,$ $n\in\N,$ is increasing for all $x\in{\tilde{\G}}$ and $k_n(x)\leq n,$ we have that $p_n(x)<\infty$ if $L_x^{(n)}\neq0.$ For each $n\in\N$ and $x\in{\tilde{\G}_n},$ we then let $\sigma_{x}^{(n)}=\sigma_{p_n(x)}$ if $L_x^{(n)}>0$ and $\sigma_x^{(n)}=1$ otherwise, and set
\begin{equation}
\label{eq:approxphi_n}
 \phi_x^{(n)}\stackrel{\text{def.}}{=}\sigma_x^{(n)}\sqrt{2L_x^{(n)}}.
 \end{equation}
  Due to \eqref{eqcouplingloopsgff}, $(\phi_x^{(n)})_{x\in{\tilde{\G}}}$ has law $\P_{\tilde{\G}_n}^G.$ Moreover, for each $x\in{\tilde{\G}}$ with $L_x^{(\infty)}>0,$ for all $n$ large enough we have $x\in{\tilde{\G}_n}$ as well as $L_x^{(n)}>0,$ hence $k_n(x)$ is constant for $n$ large enough since $E_n^{\mathcal{L}}(x)$ increases to $E_\infty^{\mathcal{L}}(x).$ As a consequence, the sequence $p_n(x),$ $n\in\N,$ is decreasing for $n$ large enough, and we denote by $p_\infty(x)$ its limit. Note that we then have $p_n(x)=p_\infty(x)$ for $n$ large enough. We define $\sigma_x^{(\infty)}=\sigma_{p_\infty(x)}$ if $L_x^{(\infty)}>0$ and $\sigma_x^{(\infty)}=1$ otherwise, and $\phi_x^{(\infty)}=\sigma_x^{(\infty)}\sqrt{2L_x^{(\infty)}}.$ We then have  $\phi^{(n)}_x\tend{n}{\infty}\phi_x^{(\infty)}$ for all $x\in{\tilde{\G}}$ due to \eqref{eq:addlaw1}, \eqref{eq:approxphi_n} and since $\text{sign}(\phi^{(n)}_x)=\text{sign}(\phi^{(\infty)}_x) $ for all large enough $n$. Finally,  $g_{\tilde{\G}_n}(x,y)\tend{n}{\infty}g_{\tilde{\G}}(x,y)=g(x,y)$ for all $x,y\in{\tilde{\G}},$ whence
\begin{equation}
\label{eq:approxcf}
\lim_n \E\big[\exp( i \langle \mu_{\alpha},\phi^{(n)}\rangle )\big] =\exp \big(-\langle  \mu_{\alpha},G \mu_{\alpha} \rangle /2\big)= \E^G\big[\exp( i \langle \mu_{\alpha},\phi \rangle )\big]
\end{equation}
 for any finite point measure $\mu_{\alpha}=\sum_{x \in A}\alpha_x \delta_x$, $\alpha \in \R^{A}$ with $A \subset \tilde{\G}$ finite and $(G\mu)(x)=\int_{\tilde{\G}} g(x,y)d\mu(y)$. The statement that $(\phi_x^{(\infty)})_{x\in{\tilde{\G}}}$ has law $\P^G$ follows from \eqref{eq:approxcf} and convergence of  $\phi^{(n)}$ (in law). This shows \eqref{eq:GFFappro1}.
	
	With probability $1$, for each $K \subset \tilde{\G}$ connected compact, there exists a random $N\in\N,$ such that for all $n\geq N,$ one has $K\subset\tilde{\G}_n,$ and no trajectory in $\tilde{\mathcal{L}}^{(\infty)}$ hitting $K$ hits $\tilde{\G}\setminus\tilde{\G}_n$. One then has the equality $L_x^{(n)}=L_x^{(\infty)}$  for all $n\geq N$ and $x\in{K},$ and the clusters of $\{L_{\cdot}^{(n)}>0\}$ in $\tilde{\G}$ whose closure is contained in $K$ are equal to the clusters of $\{L_{\cdot}^{(\infty)}>0\}$ whose closure is contained in $K.$ As a consequence, once  $n\geq N,$ on has that $\sigma_x^{(n)}=\sigma_x^{(\infty)}$ on all these clusters. Since $\partial K$ is finite, we also have $\sigma_x^{(n)}=\sigma_x^{(\infty)} (=1)$ for all $x\in{\partial K}$ and $n$ large enough.  The claim \eqref{eq:equalityKapprox} follows.
\end{proof}

Lemma~\ref{LemmaapproGFF} yields the following important result.
\begin{Cor}[Limits of cluster capacities]\label{Cor:limcap}
Let $E_{n}^{\geq h}(x_0)= E_{n,\tilde{\G}}^{\geq h}(x_0)$, where $E_{n,K}^{\geq h}(x_0)=\{x\in{\tilde{\G}_n\cap K}:\,x_0\leftrightarrow x\text{ in } \{ \varphi^{(n)} \geq h\}\cap K\}$, for $K \subset \tilde{\G}$. Then  $\P$-a.s., for all $h \in \R$, $x_0\in{\tilde{\G}}$,
\begin{align}
&\label{capGngeqcapG}
	\lim_{n\rightarrow\infty}\mathrm{{\rm cap}}_{\tilde{\G}_n}\big(E_{n,K}^{\geq h}(x_0)\big)=\mathrm{{\rm cap}}_{\tilde{\G}}\big(E^{\geq h}_{\infty,K}(x_0)\big),  \text{ for compact $K \subset \tilde{\G}$, and} \\
&\label{capGngeqcapG'}
	\lim_{n\rightarrow\infty}\mathrm{{\rm cap}}_{\tilde{\G}_n}\big(E_{n}^{\geq h}(x_0)\big)=\mathrm{{\rm cap}}_{\tilde{\G}}\big(E^{\geq h}_\infty(x_0)\big),  \text{ if $E^{\geq h}_\infty(x_0)$ is compact.}
\end{align}
\end{Cor}

\begin{proof}
As a consequence of \eqref{eq:equalityKapprox} one knows that for compact $K \subset \tilde{\G}$, one has $\phi^{(n)}=\phi^{(\infty)}$ on $K$ for large enough $n$, whence $\mathrm{{\rm cap}}_{\tilde{\G}_n}(E_{n,K}^{\geq h}(x_0))=\mathrm{{\rm cap}}_{\tilde{\G}_n}(E^{\geq h}_{\infty,K}(x_0))$ for such $n$. From this, \eqref{capGngeqcapG} follows using that $\mathrm{{\rm cap}}_{\tilde{\G}_n}(A) \to \mathrm{{\rm cap}}_{\tilde{\G}}(A)$ for compact $A$ as $n\to \infty$, applied with the choice $A= E^{\geq h}_{\infty,K}(x_0)$ (indeed, using \eqref{defeAcap}, \eqref{defequilibriumcable} and \eqref{defcap}, it is not hard to show that the equilibrium measure of any compact set $A$ on $\tilde{\G}_n$ converges --in fact decreases-- to the equilibrium measure of $A$ on $\tilde{\G}$). Now, if $E_{\infty}^{\geq h}(x_0)$ is compact, then $E^{\geq h}_{\infty}(x_0)=E^{\geq h}_{\infty,K}(x_0)=E^{\geq h}_{n,K}(x_0)$ for large enough $n$ and $K$ depending on $\varphi$. Together with \eqref{capGngeqcapG}, this immediately gives \eqref{capGngeqcapG'}.
\end{proof}

\section{Proofs of Theorems \ref{mainresult} and \ref{mainresultcap}} 
\label{sec:mainviainter}

With the results of the last section at hand, we are ready to give the proofs of Theorems~\ref{mainresult} and~\ref{mainresultcap}. This is the subject of the present section. Both proofs rely on Proposition \ref{couplingimplytheorem} in combination with Lemmas \ref{Theforfinite} and \ref{LemmaapproGFF} and Corollary \ref{Cor:limcap}. 

First, as a consequence of Proposition~\ref{couplingimplytheorem} and Lemma~\ref{h-hsamelaw}, we collect the following
\begin{Cor}\label{C:thm2}
If \eqref{eqcouplingintergff} and \eqref{capcondition} are satisfied on $\G$, then \eqref{lawforhnegative} and \eqref{eq:capinfinity} hold.
\end{Cor}

\begin{proof} 
If \eqref{eqcouplingintergff} and \eqref{capcondition} are satisfied, \eqref{eq:0bounded} follows from $\hyperref[eq:laplacecap]{(\text{Law}_0)}$ (which holds on account of Proposition~\ref{couplingimplytheorem}) by letting $u\downarrow 0$ and using \eqref{capcondition}. Therefore \eqref{lawforhnegativecompact} holds, which, together with \eqref{capcondition} and Lemma \ref{1stpartofmaincor}, yields \eqref{lawforhnegative}. Then, using  \eqref{lawforhnegative} we have	that $\P^G\big(\mathrm{{\rm cap}}({E}^{\geq-h}(x_0))\in{(\mathrm{{\rm cap}}(\{x_0\}),\infty)}\big)=\P^G(\phi_{x_0}\geq h)$. Since $\P^G\big(\mathrm{{\rm cap}}({E}^{\geq-h}(x_0))\leq\mathrm{{\rm cap}}(\{x_0\})\big)=\P^G(\phi_{x_0}\leq -h),$ we infer \eqref{eq:capinfinity}.
\end{proof}

We now give the 

\begin{proof}[Proof of Theorem \ref{mainresult}]
For a given graph $\G =(\overline{G},\overline{\lambda}, \overline{\kappa})$, consider an increasing sequence $U_n,$ $n\in\N,$ of finite connected subsets of $\overline{G}$ exhausting $\overline{G}$, i.e.~satisfying $U_n \subset U_{n+1}$ for all $n$ and $\bigcup_n U_n=\overline{G}$. Now, define $\G_n=(\overline{G},\overline{\lambda}, \overline{\kappa}^{(n)})$ with killing measure $\overline \kappa_x^{(n)}=\overline \kappa_x$ if $x\in{U_n},$ and $\overline \kappa_x^{(n)}=\infty$ otherwise. The sequence of graphs $\G_n,$ $n\in\N,$ increases to $\G$ in the sense of \eqref{eq:def_approx}, and $G_n$ is finite for each $n\in\N$ in the sense as above Lemma~\ref{Theforfinite}. Fixing a point $x_0 \in \tilde{\G}$, we may furthermore assume that $x_0 \in \tilde \G_n$ for all $n \in \N$ (for instance by choosing $U_n= B_d(z_0, n+1) $, where $z_0 \in G$ is the vertex closest to $x_0$ relative to $d$).

Considering the sequence $(\phi_x^{(n)})_{x\in{\tilde{\G}_n}},$ $n \in \N,$ from Lemma \ref{LemmaapproGFF}, which is in force, we obtain, applying Lemma~\ref{Theforfinite} and Proposition~\ref{couplingimplytheorem}, which implies \hyperref[eq:laplacecap]{($\text{Law}_h$)}, that for all $n\in\N,$
	\begin{equation}
	\label{lawcapat0n}
	\E\left[\exp\left(-u\mathrm{{\rm cap}}_{\tilde{\G}_n}({E}_n^{\geq h}(x_0))\right) 1_{\phi_{x_0}^{(n)}\geq h}\right]=\P(\phi_{x_0}^{(n)}\geq \sqrt{2u+h^2})\text{ for all }u>0, \, h \geq 0.
	\end{equation}
	Fixing $h=0$, \eqref{lawcapat0n} and the monotonicity property \eqref{capincreasing} thus yield, for any compact $K \subset \tilde{\G}$,
	\begin{equation}
	\label{lawcapat0nK}
	\E\left[\exp\left(-u\mathrm{{\rm cap}}_{\tilde{\G}_n}({E}_{n,K}^{\geq 0}(x_0))\right) 1_{\phi_{x_0}^{(n)}\geq 0}\right]\geq \P(\phi_{x_0}^{(n)}\geq \sqrt{2u})\text{ for all }u>0.
	\end{equation}
with ${E}_{n,K}^{\geq h}(x_0)$ as defined above \eqref{capGngeqcapG}. Now, applying \eqref{capGngeqcapG} and dominated convergence to take the limit $n\rightarrow\infty$ on both sides of \eqref{lawcapat0n}, and subsequently considering an increasing sequence of compacts $K$ exhausting $\tilde{\G}$, one obtains, in view of \eqref{defcapinfinity}, 
	\begin{equation}
		\label{lawcapat0nlim}
	\E^G\left[\exp\left(-u\mathrm{{\rm cap}}_{\tilde{\G}}({E}^{\geq 0}(x_0))\right) 1_{\phi_{x_0}\geq 0}\right]\geq\P^G(\phi_{x_0}\geq \sqrt{2u})\text{ for all }u>0.
	\end{equation}
	Hence, taking $u\rightarrow0$ we obtain by dominated convergence that
	\begin{equation*}
	\P^G\big(\mathrm{{\rm cap}}(E^{\geq0}(x_0))<\infty,\phi_{x_0}\geq0 \big)\geq\frac12.
	\end{equation*}
	Since $E^{\geq0}(x_0)=\varnothing$ when $\phi_{x_0}<0$ and $\P^G(\phi_{x_0}<0)=\frac12,$ we obtain that $\mathrm{{\rm cap}}(E^{\geq0}(x_0))$ is $\P^G$-a.s.\ finite, which proves the first part of the statement.
	
	Let us now fix some $h<0.$ If $E^{\geq h}_n(x_0)$ is a non-compact subset of $\tilde{\G}_n$ for infinitely many $n,$ then for all compacts $K$ of $\tilde{\G}$ we have $E_n^{\geq h}(x_0)\not\subset K$ for infinitely many $n\in\N.$ Since $\phi^{(n)}=\phi^{(\infty)}$ on a neighborhood of $K$ for $n$ large enough, we then have that $E_{\infty}^{\geq h}(x_0)\not\subset K$ for all compacts $K,$ that is $E_{\infty}^{\geq h}(x_0)$ is a non-compact subset of $\tilde{\G}.$ Since \eqref{eq:capinfinity} holds on $\G_n$ by Lemma \ref{Theforfinite} and Corollary~\ref{C:thm2}, we moreover have that
	\begin{align*}
	\P(E^{\geq h}_n(x_0)\text{ is non-compact in }\tilde{\G}_n\text{ i.o.})&\geq \liminf_{n\rightarrow\infty}\P(E^{\geq h}_n(x_0)\text{ is non-compact in }\tilde{\G}_n)\\&=\liminf_{n\rightarrow\infty}\P(\phi_{x_0}^{(n)}\in{(-h,h)}) =\P^G(\phi_{x_0}\in{(-h,h)})>0,
	\end{align*}
and so $E_{\infty}^{\geq h}(x_0)$ is non-compact with positive probability. 
\end{proof}


Prior to giving the proof of Theorem \ref{mainresultcap}, we first briefly study some properties of the law of the capacity of the level sets of the Gaussian free field, when their the Laplace transform is given by \eqref{eq:laplacecap} (see above Theorem \ref{T:main}). The next lemma computes the corresponding density (on the event $ \{ {E}^{\geq h}(x_0) \neq \emptyset \}$).

\begin{Lemme}
\label{equivalenceforthelaw}
For all $u\geq0$ and $h\in\R,$
\begin{equation}
\label{eq:laplace}
    \int_{g(x_0,x_0)^{-1}}^\infty\rho_h(t)\exp(-ut)\diff t=\P^G\big(\phi_{x_0}\geq \sqrt{2u+h^2}\big),
\end{equation}
with $\rho_h$ as defined in \eqref{eq:hdensity}.
\end{Lemme}
\begin{proof}
Taking $v=u+h^2/2$ and $a=g(x_0,x_0)^{-1},$ it is enough to show that
\begin{equation}
\label{equivlaplace}
        \int_{a}^\infty\frac{1}{t\sqrt{2\pi(t-a)}}\exp(-vt)\diff t=\int_{\sqrt{2v}}^\infty\exp\Big(-\frac{at^2}{2}\Big)\diff t\text{ for all }v,a\geq0.
\end{equation}
For $v=0$ we have, taking $s=\sqrt{t-a},$
\begin{equation*}
     \int_{a}^\infty\frac{1}{t\sqrt{2\pi(t-a)}}\diff t=\sqrt{\frac{2}{\pi}}\int_{0}^\infty\frac{1}{s^2+a}\diff s=\sqrt{\frac{2}{a\pi}}\Big[\arctan\Big(\frac{s}{\sqrt{a}}\Big)\Big]_0^\infty=\sqrt{\frac{\pi}{2a}},
\end{equation*}
and so \eqref{equivlaplace} holds for $v=0.$ Moreover, by dominated convergence, the left-hand side of \eqref{equivlaplace} viewed as a function of $v > 0$ is continuously differentiable with derivative
\begin{equation*}
    -\int_{a}^\infty\frac{1}{\sqrt{2\pi(t-a)}}\exp(-vt)\diff t=-\sqrt{\frac{2}{\pi}}\int_{0}^\infty\exp\big(-v(a+s^2)\big)\diff s=-\frac{1}{\sqrt{2v}}\exp(-va),
\end{equation*}
and so is equal to the derivative with respect to $v$ of the term on the right-hand side of \eqref{equivlaplace}. This yields \eqref{equivlaplace} and hence \eqref{eq:laplace}.
\end{proof}

We now proceed to the

\begin{proof}[Proof of Theorem \ref{mainresultcap}]
Consider the approximating sequence $\G_n$ introduced at the beginning of the proof of Theorem \ref{mainresult}. In particular, \eqref{lawcapat0n} still holds (as a consequence of Lemma~\ref{Theforfinite} and Proposition~\ref{couplingimplytheorem}). Now, let $h\geq0$ and suppose $E^{\geq h}(x_0)$ is $\P^G$-a.s.\ bounded, hence compact in view of Lemma~\ref{1stpartofmaincor}. Then \eqref{capGngeqcapG'} holds and one can safely pass to the limit in \eqref{lawcapat0n} using dominated convergence, thus obtaining that \eqref{eq:laplacecap} holds on $\tilde{\G}$. Then, \eqref{eq:hdensity} holds on $\tilde{\G}$ by means of Lemma \ref{equivalenceforthelaw}. In particular, the previous argument shows that, if $h\geq0$ and $E^{\geq h}(x_0)$ is $\P^G$-a.s.\ bounded, then
\begin{equation}
\label{eq:convlaw}
\text{$\mathrm{{\rm cap}}_{\tilde{\G}_n}(E_n^{\geq h}(x_0))$ converges in law to $\mathrm{{\rm cap}}_{\tilde{\G}}(E^{\geq h}(x_0))$, which is given by \eqref{eq:laplacecap}.}
\end{equation}
Assume now that \eqref{capcondition} is fulfilled on $\G.$ Then \eqref{eq:0bounded} holds by Corollary \ref{maincor}, and so we obtain \eqref{item:LawhTrue} from \eqref{eq:convlaw}. In order to deduce \eqref{lawforhnegative}, first observe that \eqref{lawforhnegative} holds on $\tilde{\G}_n$ by means of Lemma~\ref{Theforfinite} and Corollary~\ref{C:thm2}, as \eqref{capcondition} is trivially satisfied on $\tilde{\G}_n$. For all $h \geq 0$, due to \eqref{eq:convlaw} the random variable $\text{{\rm cap}}({E}_n^{\geq h}(x_0)) 1_{\phi_{x_0}^{(n)}{\geq h}}$ converges in law to $\text{{\rm cap}}({E}^{\geq h}(x_0)) 1_{\phi_{x_0}{\geq h}}$, hence so does $\text{{\rm cap}}({E}_n^{\geq -h}(x_0))1_{\text{{\rm cap}}({E}_n^{\geq -h}(x_0))\in{(0,\infty)}}$. To identify this with the law of $\text{{\rm cap}}({E}^{\geq -h}(x_0))1_{\text{{\rm cap}}({E}^{\geq -h}(x_0))\in{(0,\infty)}}$, one applies dominated convergence, noting that, due to \eqref{capcondition} and Lemma~\ref{1stpartofmaincor}, $\text{{\rm cap}}({E}^{\geq -h}(x_0))< \infty$ is tantamount to ${E}^{\geq -h}(x_0)$ being compact, and using \eqref{capGngeqcapG'}. All in all, this gives \eqref{lawforhnegative}. Finally \eqref{eq:capinfinity} is an immediate consequence of \eqref{lawforhnegative}, as in the proof of Corollary~\ref{C:thm2}. This completes the proof of Theorem \ref{mainresultcap}.
\end{proof}

\begin{Rk}
	\label{approimplylawcap}
	\begin{enumerate}[label={\arabic*)}]
		\item \label{ifGntendGthenlaw0}
		In view of the above proof of Theorem~\ref{mainresultcap}, we see that the validity of~\hyperref[eq:laplacecap]{($\text{Law}_0$)} (and thus equivalently of \eqref{eqcouplingintergffszn} by \eqref{equivisom} after Theorem~\ref{couplingintergff} is proved) can be viewed as a question about removing the compactness assumption in \eqref{capGngeqcapG'}. Indeed \hyperref[eq:laplacecap]{($\text{Law}_0$)} holds if and only if
there exists a sequence $\G_n$ of graphs verifying \hyperref[eq:laplacecap]{($\text{Law}_0$)} increasing to $\G$ in the sense of \eqref{eq:def_approx} such that, $\P$-a.s.,
\begin{equation}\label{eq:capconvcrit}
\mathrm{{\rm cap}}_{\tilde{\G}_n}\big(E_n^{\geq0}(x_0)\big)\stackrel{n \to \infty}{\longrightarrow} \mathrm{{\rm cap}}_{\tilde{\G}}\big(E_{\infty}^{\geq0}(x_0)\big)\text{ for all }x_0\in{\tilde{\G}}.
\end{equation}		
		\item\label{stronglawnegative} Let $K \subset \tilde{\G}$ be connected and compact. By Lemmas  \ref{Theforfinite} and \ref{h-hsamelaw}, it follows that if $\G$ is a finite graph, then the compact clusters of $E^{\geq -h}$ and $E^{\geq h}$ have the same law. In particular, 
		\begin{equation}
		\label{eq:symclustersinK}
		\text{the clusters of $E^{\geq -h}$ and $E^{\geq h}$ \text{included in $K$} have the same law.}
		\end{equation}
	The conclusion	\eqref{eq:symclustersinK} remains true for arbitrary transient graph $\G$. Indeed, by following the arguments of Proposition 1.11~in \cite{MR2932978}, starting from ${\G}^{\partial K}$, one can construct a transient weighted graph ${\G}_*^{\partial K}$ with (finite) vertex set $G^{\partial K}\cap K$ (recall Lemma~\ref{GA} for notation)
whose weights coincide with $\lambda_{x,y}^{\partial K}$ whenever $x,y \in G^{\partial K}\cap K$ are neighbors in ${\G}^{\partial K}$, in such a way that $(\phi_x)_{x\in{K}}$  has the same law under $\P_{\tilde{\G}}^G$ as under $\P_{\tilde{\G}^{\partial K}_*}^G$. The conclusion \eqref{eq:symclustersinK} for arbitrary $\G$ then simply follows by regarding the
clusters of $E^{\geq -h}$ and $E^{\geq h}$ included in $K$ as parts of $\tilde{\G}^{\partial K}_*$. One can also prove that the conclusion \eqref{lawforhnegativecompact} holds under condition \eqref{eq:0bounded} using \eqref{eq:symclustersinK}, by considering a sequence of compacts increasing to $\tilde{\G}.$

	\end{enumerate}
\end{Rk}

\section{Proof of Theorem~\ref{couplingintergff}}
\label{sec:iso}

In this section, we prove Theorem \ref{couplingintergff}, along with its corollaries. In particular, this comprises the isomorphism between random interlacements and the Gaussian free field and the equivalences \eqref{equivisom}, as well as its discrete counterpart \eqref{eqcouplingintergffdis}. We first compare random interlacements on $\G=\G_{\bar{\kappa}}$ (recall the notation from above \eqref{eq:graphinclusion}) with random interlacements on $\G_{\bar \kappa'}$ for some $\bar\kappa'\geq \bar\kappa$ in Lemma \ref{limitKn}, and then take advantage of this comparison to approximate random interlacements on any transient graphs by random interlacements on finite graphs as in \eqref{eq:def_approx}, see Lemma \ref{approximationprop}. Together with the corresponding `finite-volume'  approximation of the Gaussian free field from Lemma \ref{LemmaapproGFF} and in combination with the fact that Theorem \ref{couplingintergff} holds on finite graphs (see Lemma \ref{Theforfinite}), we can then prove the isomorphism \eqref{eqcouplingintergffszn}, see Lemma \ref{couplingisalwaystrue}, under suitable assumptions. This is the key step of the proof of  Theorem \ref{couplingintergff}, presented thereafter. Finally, at the end of the section, we deduce from Theorem \ref{couplingintergff} that Corollaries \ref{dichotomy} and \ref{percoath*} also hold.

We first dispense with the equivalence between~\eqref{eqcouplingintergffszn} (see p.\pageref{eqcouplingintergffszn}) and \eqref{eqcouplingintergff} (see p.\pageref{eqcouplingintergff}).

\begin{Lemme}
	\label{isomequivisom'}
	The identity \eqref{eqcouplingintergffszn} holds true if and only if \eqref{eqcouplingintergff} does.
\end{Lemme}
\begin{proof}
It suffices to argue that $(\phi_x 1_{x\notin{\mathcal{C}_u}}+\sqrt{\phi_x^2+2\ell_{x,u}}\,  1_{x\in{\mathcal{C}_u}})_{x\in{\tilde{\G}}}$ has the same law under ${\P}^{I}\otimes\P^G $ as $(\sigma_x^u\sqrt{2\ell_{x,u}+\phi_x^2})_{x\in{\tilde{\G}}}$ under $\tilde{\P}$. By definition of ${\mathcal{C}_u}$ and since $|\sigma_x^u|=1$, the absolute value of either field equals $\sqrt{2\ell_{\cdot,u}+\phi_{\cdot}^2}$ in law. To deal with the signs, rewriting $\phi_x=\text{sign}(\phi_x)\sqrt{\phi_x^2+2\ell_{x,u}}$ for all $x\notin{\mathcal{C}_u}$, one observes that  the law of $(\text{sign}(\phi_x)1_{x\notin{\mathcal{C}_u}}+1_{x\in{\mathcal{C}_u}})_{x\in{\tilde{\G}}}$ under $(\P^I\otimes\P^G)(\cdot\,|\,|\phi|,\omega^u)$ is the same as the law of $\sigma^u$ under $\tilde{\P}_{\tilde{\G}}(\cdot\,|\,|\phi|,\omega^u),$ which follows immediately from the definitions of $\mathcal{C}_u$ and $\sigma^u$, respectively, together with Lemma 3.2 in~\cite{MR3502602} (the latter asserts that given $|\varphi|$, the field $\text{sign}(\phi)$ is constant on each cluster of $\{|\phi|>0\},$ and the values on each cluster are independent and uniformly distributed, a consequence of the strong Markov property).
\end{proof}

We are now going to approximate random interlacements on any transient graph $\G$ by random interlacements on a sequence of finite graphs $\G_{n}$ increasing to $\G$ in the sense of  \eqref{eq:def_approx}. To this end, we first compare random interlacements on two graphs $\G = (\overline{G}, \bar{\lambda}, \bar{\kappa})$ and $\G' =(\overline{G}, \bar{\lambda}, \bar{\kappa}')$ with killing measures $ \bar{\kappa}' \geq  \bar{\kappa}$, and corresponding cable systems $\tilde{\G}$ and $\tilde{\G}'$. Thus, $\tilde{\G}= \tilde{\G}_{\bar \kappa}$, $\tilde{\G}'= \tilde{\G}_{\bar \kappa'}$ in the notation from the beginning of Section \ref{S:I_x} and in particular, cf.~\eqref{eq:graphinclusion}, one can regard $\tilde{\G}'$ as a subset of $\tilde{\G}.$ Accordingly, for all trajectories $w\in{W_{\tilde{\G}}}$ with $\zeta^- < 0 < \zeta^+$ (see Section \ref{subsec:RI} for notation; recall in particular that $\zeta^\pm$ are such that $w(t)=\Delta$ if and only if $t\notin{(\zeta^-,\zeta^+)}$), we define the killing times $\zeta_{\bar\kappa'}^{\pm}$ by
\begin{equation*}
\zeta_{\bar\kappa'}^\pm(w)\overset{\text{def.}}{=}\pm \inf\big\{t\in{[0, \pm\zeta^\pm(w))}:\, w(\pm t)\notin{\tilde{\G}'}\big\}
\end{equation*}
with the convention 
\begin{equation} \label{eq:emptyConvention}
\inf\varnothing=\pm\zeta^\pm(w)
\end{equation}
so that $ \zeta^-(w) \leq \zeta_{\bar\kappa'}^{-}(w)< 0 < \zeta_{\bar\kappa'}^+(w) \leq \zeta^+(w)$ for any $w\in{W_{\tilde{\G}}}$. For any compact $K\subset\tilde{\G},$ we then introduce $\pi_{K} :W_{K,\tilde{\G}}^{0}\rightarrow W_{K,\tilde{\G}'}^{0}$ by
\begin{equation}
\label{eq:piK}
\pi_{K}(w) \equiv \pi_{K,\tilde{\G},\tilde{\G}'} (w)=\begin{cases}
w(t),&\text{if }t\in{(\zeta_{\bar\kappa'}^-(w),\zeta_{\bar\kappa'}^+(w))},
\\\Delta,&\text{otherwise,}
\end{cases}
\end{equation}
and denote by $\pi_{K}^{*}:W_{K,\tilde{\G}}^{*}\rightarrow W_{K,\tilde{\G}'}^{*}$ the unique function such that $p_{\tilde{\G}'}^*\circ\pi_{K}(w)=\pi_{K}^{*}\circ p^*_{\tilde{\G}}(w)$ for all $w\in{W_{K,\tilde{\G}}^{0}}.$ In words $\pi_{K}^{*}(w^*)$ is the doubly infinite trajectory modulo time shift on $\tilde{\G}',$ whose forward and backward parts seen from the first time of hitting $K$ are the forward and backward parts of $w^*$ seen from the first time of hitting $K,$ both stopped on exiting~$\tilde{\G}'.$

\begin{Lemme}
	\label{limitKn}
$(\tilde{\G}=  (\overline{G}, \bar{\lambda}, \bar{\kappa}),\, \tilde{\G}'=  (\overline{G}, \bar{\lambda}, \bar{\kappa}'), \,  \bar{\kappa}' \geq  \bar{\kappa} )$. Let $V\subset K$ be compact subsets of $\tilde{\G}'$. There exists a non-negative measure ${\mu}^{K,V}={\mu}^{K,V}_{\tilde{\G},\tilde{\G}'}$ on $W^*_{K,\tilde{\G}'}$ such that 
	\begin{equation}
	\label{defmutildeGkappa}
	\big(\nu_{\tilde{\G}}1_{W^*_{K,\tilde{\G}}\setminus W^*_{V,\tilde{\G}}}\big)\circ({\pi}_{K}^*)^{-1}+{\mu}^{K,V}=\nu_{\tilde{\G}'}1_{W^*_{K,\tilde{\G}'}\setminus W^*_{V,\tilde{\G}'}}
	\end{equation}
(with a slight abuse of notation, the right-hand side is viewed as a measure on $W^*_{K,\tilde{\G}'}$). Moreover, 
	\begin{equation}
	\label{massmu}
	{\mu}^{K,V}(W^*_{K,\tilde{\G}'})=\mathrm{{\rm cap}}_{\tilde{\G}'}(K)-\mathrm{{\rm cap}}_{\tilde{\G}'}(V)-\mathrm{{\rm cap}}_{\tilde{\G}}(K)+\mathrm{{\rm cap}}_{\tilde{\G}}(V).
	\end{equation}
\end{Lemme}

\begin{proof}
Throughout the proof, let $ \hat{\partial}K$ be as in \eqref{defpartialext} but relative to $P_x^{\tilde{\G}'}$ (rather than $P_x= P_x^{\tilde{\G}}$). 
Let $(G,\lambda,\kappa)$ and $(G',\lambda',\kappa')$ refer to the induced graphs corresponding to $\G$ and $\G'$, respectively (cf. \eqref{eq:defGfinite}). By considering the graphs $\G^{A}$ and $(\G')^{A}$ for any $A \supset \hat{\partial} K$, see Lemma \ref{GA} instead of $\G$ and $\G'$, we can assume without loss of generality that $ \hat{\partial} K\subset({G} \cap G')$. 
By choosing $A= A' \cup \hat{\partial} K$ where $A'\subset\tilde{\G}'$ is a set containing exactly one (arbitrary) vertex between each $x\in{\hat{\partial} K}$ and $y\in{\partial\tilde{\G}'}$ which are connected by a cable, we can further `move away' $\hat{\partial}K$ from $\partial\tilde{\G}',$ so that $d(\hat{\partial}K,\partial{\tilde{\G}'})>1,$ where $d$ is the canonical distance on $\tilde{\G}$ defined above \eqref{Ehx0}. All in all, we thus assume henceforth that
\begin{equation} \label{eq:kappaCoincide}
\hat{\partial} K\subset{(G \cap G')}
\quad  \text{and}\quad d(\hat{\partial}K,\partial{\tilde{\G}'})>1,
\end{equation}
which is no loss of generality.
Recall $X' \equiv X^{\bar \kappa'}$ and $\zeta' \equiv \zeta_{\bar\kappa'}$ from \eqref{eq:zeta} and note that for all $w\in{W_{K,\tilde{\G}}^0},$ the forward part $\{ (\pi_{K}(w))_t : 0 \leq t \leq \zeta_{\bar\kappa'}^+\}$ of $\pi_{K}(w)$ from the time of first hitting $K$ onward, is precisely $\{X_t'(w^+) : 0\leq t \leq \zeta'\}$, where $w^+$ is the forward part of $w.$ Recalling \eqref{eq:emptyConvention} as well as the notation from \eqref{defequilibriumcable} and \eqref{eq:PKG}, we then define the countably additive set function $\tilde{\mu}^{K,V}$ on  $\mathcal{W}_{K,\tilde{\G}'}^0$ by
	\begin{equation}
	\label{eq:mudefapprox}
	\begin{split}
	\tilde{\mu}^{K,V}(A)\stackrel{\text{def.}}{=}\sum_{x\in{ \hat{\partial} K}}&\Big(e_{K,\tilde{\G}'}(x)P_x^{\tilde{\G}'}\big(X \in A^+,{H}_{V}=\zeta\big)P^{K,\tilde{\G}'}_x(X\in A^-)
	\\&-e_{K,\tilde{\G}}(x)P_x^{\tilde{\G}}\big(X'\in{A^+},{H}_{V}=\zeta\big)P^{K,\tilde{\G}}_x\big(X'\in{A^-}\big)\Big)
	\end{split}
	\end{equation}
(note that following our convention below \eqref{swappinglemma}, $\{{H}_{V}=\zeta \}$ under $P_x^{\tilde{\G}}$ refers to the event that $V$ is not visited by $X$) with $A^{\pm}$ denoting $\{(w(\pm t))_{t\geq0}:\,w\in{A}\}$ for all $A \in \mathcal{W}_{K,\tilde{\G}'}^0$ and $X'$ as introduced below \eqref{eq:kappaCoincide}.  In \eqref{eq:mudefapprox}, we also used implicitly the convention that $e_{K,\tilde{\G}}(x)P^{K,\tilde{\G}}_x=0$ for all $x\in{ \hat{\partial}K}$ with $e_{K,\tilde{\G}}(x)=0.$ Moreover, $e_{K,\tilde{\G}}(x)\leq e_{K,\tilde{\G}'}(x)$ for all $x\in{\tilde{\G}}$ by \eqref{defeAcap} and \eqref{consistencyequilibrium}, and so it follows from \eqref{exitequi} that $\text{supp}(e_{K,\tilde{\G}}) \subset  \hat{\partial} K.$ If $\tilde{\mu}^{K,V}$ is non-negative on $\W_{K,\tilde{\G}'}^0$ we can extend it to a measure on $\W_{K,\tilde{\G}'}$ by taking $\tilde{\mu}^{K,V}(A)=0$ for all $A\in{\mathcal{W}_{K,\tilde{\G}'}}$ with $A\cap {W}_{K,\tilde{\G}'}^0=\varnothing.$ Defining ${\mu}^{K,V}=\tilde{\mu}^{K,V}\circ(p_{\tilde{\G}'}^*)^{-1},$ in view of \eqref{eq:mudefapprox}, \eqref{defQK} and \eqref{definter}, it then follows that \eqref{defmutildeGkappa} is fulfilled. 

We now show that \(\tilde{\mu}^{K,V}\) is non-negative. Recall $\widehat{Z}$, the discrete skeleton of $Z$, from below \eqref{traceonG}. We denote by $\hat{L}_K=\sup\{n\in{\N}:\,\hat{Z}_n\in{K}\}$ the last exit time of $K$ for $\hat{Z}$ and by $L_K=\sup\{t\geq0:\,X_t\in{K}\}$ the last exit time of $K$ for $X,$ with the convention $\sup\varnothing=\infty,$ so in particular $\{X_{L_K}=x\}=\{\hat{Z}_{\hat{L}_K}=x\}$ for all $x\in{\hat{\partial}K}$ (on the event $\{L_K < \infty \}=\{\hat{L}_K < \infty\} $, which has full $P_x^{\tilde{\G}}$-measure by transience). 
We also define $(Y_t)_{t\geq0}$ the same process as $(X_{t+L_K})_{t\geq 0},$ but killed the first time $(X_{t+L_K})_{t\geq 0}$ hits $\partial\tilde{\G}'.$ By definition of $P^{K,\tilde{\G}}_x$, see \eqref{eq:PKG}, and \eqref{exitequi}, we have for all $x\in{ \hat{\partial} K}$ with $e_{K,\tilde{\G}}(x)>0$  that
	\begin{align*}
	&e_{K,\tilde{\G}}(x)P^{K,\tilde{\G}}_x(X'\in{\cdot}\, )
		\\&\qquad =e_{K,\tilde{\G}}(x)P_x^{\tilde{\G}}\big((Y_t)_{t>0}\in{\cdot}\,|\,X_{L_K}=x\big)
 =\frac{1}{g_{\tilde{\G}}(x,x)}P_x^{\tilde{\G}}\big((Y_t)_{t>0}\in{\cdot},\hat{Z}_{\hat{L}_K}=x\big)
	\\&\qquad =\frac{1}{g_{\tilde{\G}}(x,x)}\sum_{n\geq0}P_x^{\tilde{\G}}\big((Y_t)_{t>0}\in{\cdot},\hat{Z}_{n}=x,\hat{L}_{K}=n\big) =\lambda_xP_x^{\tilde{\G}}\big((Y_t)_{t>0}\in{\cdot},{\hat{L}_K}=0\big);
	\end{align*}
here, we used in the last equality the strong Markov property at the time of $n$-th jump and the fact that	$g_{\tilde{\G}}(x,x)=\frac{1}{\lambda_x}\sum_{n\geq0}P_x^{\tilde{\G}}(\hat{Z}_n=x)$. By a similar calculation, and in view of \eqref{lawkappa'vskappa}, we obtain for $x \in \hat{\partial}K$,
	\begin{align*}
	e_{K,\tilde{\G}'}(x)P^{K,\tilde{\G}'}_x(X\in \cdot\,)=\lambda'_xP_x^{\tilde{\G}'}\big((X_{t+L_K})_{t>0}\in{\cdot},{\hat{L}_K}=0\big)=\lambda'_xP_x^{\tilde{\G}}\big((X_{t+L_K'}')_{t>0}\in{\cdot},{\hat{L}'_K}=0\big),
	\end{align*}
where $L_K'$, $\hat{L}'_K$ are defined as above but with $X'$ in place of $X$.
	On the event $\hat{L}_K=0,$ since $d(\hat{\partial}K,\partial{\tilde{\G}'})>1$ due to \eqref{eq:kappaCoincide}, we have $L_K'=L_K,$ $(Y_t)_{t>0}=(X'_{t+L_K'})_{t>0}$ and $\lambda_x=\lambda'_x$ for all $x\in{\hat{\partial} K}.$ Hence, for all $x\in \hat{\partial} K$ with $e_{K,\tilde{\G}}(x)>0$,
	\begin{equation}\label{eq:mupos_calc1}
	\begin{split}
	&e_{K,\tilde{\G}'}(x)P^{K,\tilde{\G}'}_x(X\in \cdot\,)-e_{K,\tilde{\G}}(x)P^{K,\tilde{\G}}_x\big(X'\in{\cdot}\big)=\lambda_xP_x^{\tilde{\G}}\big((X_{t+L_K'}')_{t>0}\in{\cdot},{\hat{L}_K'}=0<{\hat{L}_K}\big)
	\\&\quad =e_{K,\tilde{\G}'}(x)P^{\tilde{\G}}_x\big((X_{t+L_K'}')_{t>0} \in \cdot ,L_K'<L_K\,|\,X_{L_K'}=x\big).
	\end{split}
	\end{equation}
	Note that if $e_{K,\tilde{\G}}(x)=0$ and $x\in{\hat{\partial}K},$ then $L_K'<L_K$ $P^{\tilde{\G}}_x$-a.s., and so the previous equality still holds. Moreover, using \eqref{lawkappa'vskappa}, we have for all $x\in{ \hat{\partial} K}$ that
	\begin{equation}\label{eq:mupos_calc2}
	P_x^{\tilde{\G}'}\big(X \in \cdot,{H}_{V}=\zeta\big)-P_x^{\tilde{\G}}\big(X'\in{\cdot},{H}_{V}=\zeta\big)=P_x^{\tilde{\G}}\big(X'\in\cdot,\zeta>{H}_{V}>\zeta'\big).
	\end{equation}
	Combining \eqref{eq:mudefapprox}, \eqref{eq:mupos_calc1} and \eqref{eq:mupos_calc2}, we thus obtain that, for $ A \in \mathcal{W}^0_{K,\tilde{\G}'},$
	\begin{align}
	\label{threepartmu}
	\begin{split}\tilde{\mu}^{K,V}(A)=\sum_{x\in{ \hat{\partial} K}}&\Big(e_{K,\tilde{\G}'}(x)P_x^{\tilde{\G}'}\big(X\in A^+,{H}_{V}=\zeta\big)P^{\tilde{\G}}_x\big((X_{t+L_K'}')_{t>0}\in A^-,L_K'<L_K\,|\,X_{L_K'}=x\big)
	\\&+e_{K,\tilde{\G}}(x)P_x^{\tilde{\G}}\big(X'\in A^+,\zeta >{H}_{V}>\zeta'\big)P^{K,\tilde{\G}}_x\big(X'\in A^-\big) \Big),
	\end{split}
	\end{align}
	and so $\tilde{\mu}^{K,V}$ is positive on $\mathcal{W}^0_{K,\tilde{\G}'}.$ Finally, we have by \eqref{defcap} and \eqref{swappinglemma} that
	\begin{align*}
	{\mu}^{K,V}(W^*_{K,\tilde{\G}'})&=\tilde{\mu}^{K,V}(W_{K,\tilde{\G}'}^0) \stackrel{\eqref{eq:mudefapprox}}{=}\sum_{x\in{ \hat{\partial} K}}\Big(e_{K,\tilde{\G}'}(x)P^{\tilde{\G}'}_x({H}_{V}=\zeta)-e_{K,\tilde{\G}}(x)P^{\tilde{\G}}_x({H}_{V}=\zeta)\Big)
	\\&=\mathrm{{\rm cap}}_{\tilde{\G}'}(K)-\mathrm{{\rm cap}}_{\tilde{\G}'}({V})-\mathrm{{\rm cap}}_{\tilde{\G}}(K)+\mathrm{{\rm cap}}_{\tilde{\G}}({V}),
	\end{align*}
which gives \eqref{massmu} and completes the proof.
\end{proof}

In words, the difference between the trajectories under $\nu_{\tilde{\G}}$ and $\nu_{\tilde{\G}'}$ that hit $K$ but not $V,$ when $V\subset K$ are compact subsets of $\tilde{\G}',$ comes in two parts: first it is more likely for the forward trajectories to not hit $V$ before time $\zeta'$ than before time $\zeta,$ and secondly it is more likely for the backward trajectories to not come back to $K$ before time $\zeta'$ than before time $\zeta$. These two differences are contained in the measure $\mu_{\tilde{\G},\kappa'}^{K,V}$ from \eqref{defmutildeGkappa}, see \eqref{threepartmu}. 

Taking a sequence $(K_p)_{p\in\N}$ of compacts increasing to $\tilde{\G}',$ one can then use Lemma \ref{limitKn} to construct a random interlacement process on $\tilde{\G}'$ from the random interlacement process $\omega$ on $\tilde{\G}$: take the image through $\pi_{K_p}^*$ of each trajectory in the support of $\omega$ hitting $K_p$ but not $K_{p-1}$ for all $p\in\N,$ with $K_0=\varnothing,$ and add Poisson point processes with intensity $\mu_{\tilde{\G},\tilde{\G}'}^{K_{p},K_{p-1}}\otimes\lambda$ for all $p\in\N.$ Using this construction and the estimate \eqref{massmu}, we will now suitably approximate random interlacements on $\G$ by random interlacements on a sequence of finite graphs, thus mirrorring Lemma~\ref{LemmaapproGFF}.

\begin{Lemme}
	\label{approximationprop}
	Let $\G$ be a transient weighted graph and $\G_n,$ $n\in\N,$ be a sequence of transient weighted graphs increasing to $\G_{\infty}=\G$ in the sense of \eqref{eq:def_approx}. There exists a probability space $(\Omega',\mathcal{F}',\P')$ on which one can define a sequence of processes $\omega^{(n)},$ $n\in\N,$ and $\omega^{(\infty)}$ with the following properties: 
	\begin{align}
&\label{eq:RIappro1}\ \text{ for all $n\in\N\cup\{\infty\}$, the process $\omega^{(n)}$ has the same law as $\omega$ under $\P_{\tilde{\G}_{n}}^{I}$;}\\
&\begin{array}{l}
\text{there exists an increasing sequence $(a_n)_{n\in\N}$ such that for each $u>0$, $\P'$-a.s.~for}\\
\text{all compact $K\subset \tilde{\G}$, the restriction to $K$ of the set of trajectories hitting $K$ is the}\\
\text{same for $\omega^{(a_n)}_u$ and $\omega^{(\infty)}_u$ for all $n$ large enough.}
\end{array}
\label{eq:equalityKapproxRI}
	\end{align}
\end{Lemme}
\begin{proof}
	Let $(K_n)_{n\in\N}$ be a sequence such that $K_n$ is a compact subset of $\tilde{\G}_n$ for each $n\in\N,$ and such that $K_n,$ $n\in\N,$ increases to $\tilde{\G}.$ Let $\omega^{(\infty)}$ be a Poisson point process under $(\Omega',\mathcal{F}',\P')$ with the same law as the random interlacement process $\omega$  under $\P^I_{\tilde{\G}}.$ For each $n\in\N$ and $k\in\{1,\dots,n\},$ we define, recalling the notation from \eqref{eq:def_approx}, the process $\omega_1^{(k,n)}$ as the Poisson point process which is given by the image through $ \pi^*_{k,n} \equiv \pi_{ K_k, \tilde{\G}, \tilde{\G}_n}^*$, cf.~\eqref{eq:piK}, of all the trajectories in $\omega^{(\infty)}_u$ which hit $K_{k}$ but not $K_{k-1},$ with the convention $K_0=\varnothing;$ this constitutes a Poisson point process with intensity $(\nu_{\tilde{\G}}1_{W^*_{K_k,\tilde{\G}}\setminus W^*_{K_{k-1},\tilde{\G}}})\circ\big( \pi^*_{k,n} )^{-1}.$ By suitably extending $\P'$ we further introduce $\omega_2^{(k,n)}$ as an independent Poisson point process with intensity $\mu^{K_{k},K_{k-1}}_{\tilde{\G}, \tilde{\G}_n}\otimes\lambda$ (see Lemma~\ref{limitKn}) and $\omega_3^{(n)}$ as an independent Poisson point process with intensity $(\nu_{\tilde{\G}_{n}}1_{(W_{K_n,\tilde{\G}_{n}}^*)^c})\otimes\lambda.$ Thus, defining for each $n\in\N$
	\begin{equation*}
	\omega^{(n)}\stackrel{\text{def.}}{=}\omega_3^{(n)}+\sum_{k=1}^{n}\big(\omega_1^{(k,n)}+\omega_2^{(k,n)}\big),
	\end{equation*}
	we have by \eqref{defmutildeGkappa} that $\omega^{(n)}$ has the same law as $\omega$ under $\P_{\tilde{\G}_{n}}^{I},$ whence~\eqref{eq:RIappro1}.
	
	We now argue that~\eqref{eq:equalityKapproxRI} holds. Let $u>0$ and $p\in\N.$ By definition, no trajectories of $\omega_1^{(k,n)},$ $\omega_2^{(k,n)}$ and $\omega_3^{(n)}$ hit $K_p$ if $p<k\leq n.$ Moreover, there is a only a finite number of trajectories in $\omega_u^{(\infty)}$ hitting $K_p,$ each returning finitely many times to $K_p,$ and so for each $k\in\{1,\dots,p\},$ we have that the restriction to $K_p$ of all the trajectories of $\omega_1^{(k,n)}$ at level $u$ hitting $K_p$ is constant for all $n$ large enough.  By \eqref{massmu}, for each $n\geq p,$ the number of trajectories in $\sum_{k=1}^p\omega_2^{(k,n)}$ at level $u$ is a Poisson random variable with parameter $u(\mathrm{{\rm cap}}_{\tilde{\G}_n}(K_p)-\mathrm{{\rm cap}}_{\tilde{\G}}(K_p)),$ and one can easily prove by \eqref{defeAcap}, \eqref{defequilibriumcable} and \eqref{defcap} since $K_p$ is compact that $\mathrm{{\rm cap}}_{\tilde{\G}_n}(K_p)-\mathrm{{\rm cap}}_{\tilde{\G}}(K_p)\to0$ as $n \to \infty$. As a consequence of Borel-Cantelli, one can thus find a sequence $(a_n)_{n\in\N}$ such that $\P'$-a.s., $\sum_{k=1}^p\omega_2^{(k,a_n)}$ contains no trajectory at level $u$ for all $u>0$ and $n$ large enough, and by a diagonal argument, one can take $(a_n)_{n\in\N}$ independent of the choice of $p.$ Since for all compacts $K \subset \tilde{\G},$ there exist $p\in\N$ such that $K\subset K_p,$ and $\P'$-a.s., the restriction to $K_p$ of all the trajectories of $\omega^{(a_n)}_u$ hitting $K_p$ is constant for all $n$ large enough,  we conclude~\eqref{eq:equalityKapproxRI}.
\end{proof}

Together, Lemmas \ref{LemmaapproGFF} and \ref{approximationprop} supply suitable `finite-volume' approximations for the Gaussian free field and random interlacements on a general transient weighted graph $\tilde{\G}$. With the help of  Lemma \ref{Theforfinite}, this yields the following result, from which Theorem \ref{couplingintergff} will readily follow. 

\begin{Lemme}
	\label{couplingisalwaystrue}
	If either \eqref{eq:0bounded} or \hyperref[eq:laplacecap]{\emph{($\text{Law}_0$)}} is fulfilled, then \eqref{eqcouplingintergffszn} and \eqref{lawofEu} hold true on $\G.$
\end{Lemme}
\begin{proof}
	
	Let $\G_n,$ $n\in\N$ be a sequence of finite graphs increasing to $\G$  in the sense of \eqref{eq:def_approx} (for instance, the one introduced at the beginning of the proof of Theorem \ref{mainresult}) and consider the space $(\Omega\times\Omega',\F\otimes\F',\P\otimes\P'),$ which is the product of the probability spaces from Lemmas \ref{LemmaapproGFF} and \ref{approximationprop}. By passing to a subsequence of $\G_n,$ $n\in\N,$ we may assume that $a_n=n$ in \eqref{eq:equalityKapproxRI}. Note that Lemma \ref{Theforfinite} applies to $\G_n.$ For each $n\in\N\cup\{\infty\},$ let $(\ell_{x,u}^{(n)})_{x\in{\tilde{\G}_{n}}}$ denote the total local times of the trajectories of $\omega_u^{(n)},$ $\I^u_{n}=\{x\in{\tilde{\G}_{n}}:\,\ell_{x,u}^{(n)}>0\},$ ${\Sigma_n(x)}=\{y\in{\tilde{\G}_{n}}:x\leftrightarrow y\text{ in }\{z\in{\tilde{\G}_{n}}:|\phi_z^{(n)}|>0\}\}$ and $\overline{\Sigma_n(x)}$ its closure for all $x\in{\tilde{\G}_{n}},$ as well as  $\mathcal{C}_{u,n}$ the closure of $\{x\in{\tilde{\G}}:{\Sigma_n(x)}\cap\I^u_n\neq\varnothing\}.$ Let us first prove that there exists a sequence $(b_n)_{n\in\N}$ such that, $\P\otimes\P'$-a.s.\ for all $x\in{\tilde{\G}}$ with $|\phi_x^{(\infty)}|>0,$
	\begin{equation}
	\label{IinftyimpliesIn}
	\big\{x\in{\mathcal{C}_{u,\infty}}\big\}=\liminf_{n\rightarrow\infty}\big\{x\in{\mathcal{C}_{u,b_n}}\big\}=\limsup_{n\rightarrow\infty}\big\{x\in{\mathcal{C}_{u,b_n}}\big\}.
	\end{equation}
	For this purpose, consider $x\in{\tilde{\G}}$ with $|\phi_x^{(\infty)}|>0.$ If $x\in{\mathcal{C}_{u,\infty}},$ then there exists $y\in{\I_{\infty}^u\cap \Sigma_{\infty}(x)}.$ By \eqref{eq:equalityKapproxRI},
	$y\in{\I^{u}_{n}}$ for $n$ large enough and there is a path $\pi\subset\tilde{\G}$ between $x$ and $y$ in $\{z\in{\tilde{\G}}:|\phi_z^{(\infty)}|>0\}.$ Since $\pi$ can be chosen to be compact, by \eqref{eq:equalityKapprox} we have $\phi^{(n)}=\phi^{(\infty)}$ on $\pi$ for all $n$ large enough.  Therefore, $\pi$ is also a path between $x$ and $y$ in $\{z\in{\tilde{\G}}:|\phi_z^{(n)}|>0\},$ and so $y\in{\I^u_{n}\cap \Sigma_{n}(x)}$ for $n$ large enough, that is $x\in{\mathcal{C}_{u,n}}.$ As a consequence,
	\begin{equation}
	\label{firstinclusion}
	\big\{x\in{\mathcal{C}_{u,\infty}}\big\}\subset\liminf_{n\rightarrow\infty}\big\{x\in{\mathcal{C}_{u,n}}\big\}(\subset\limsup_{n\rightarrow\infty}\big\{x\in{\mathcal{C}_{u,n}}\big\}).
	\end{equation}
	To prove the reverse inclusions in \eqref{IinftyimpliesIn}, first assume that \eqref{eq:0bounded} is fulfilled and that $x\in{\mathcal{C}_{u,n}}$ for infinitely many $n.$ By \eqref{eq:equalityKapprox} and \eqref{eq:equalityKapproxRI}, since $\overline{\Sigma_{\infty}(x)}$ is compact, we have that $\phi^{(n)}$ and $\I^u_n$ are constant for $n$ large enough on $\overline{\Sigma_{\infty}(x)},$ and then $\Sigma_n(x)\cap \I_n^u=\Sigma_{\infty}(x)\cap \I_\infty^u$ for $n$ large enough. Therefore, infinitely often, $\I_{\infty}^u\cap \Sigma_{\infty}(x)=\I_{n}^u\cap \Sigma_n(x)\neq\varnothing$ (note that $x$ cannot lie in the boundary since $\I_{\infty}^u$, $\I_{n}^u$ are open and $|\phi_x^{(n)}|>0$ for large enough $n$), that is $x\in{\mathcal{C}_{u,\infty}}.$ Combining with \eqref{firstinclusion}, we obtain \eqref{IinftyimpliesIn} with $b_n=n.$
	
	Now suppose that \hyperref[eq:laplacecap]{($\text{Law}_0$)} holds on $\G.$ For all $n\in\N\cup\{\infty\},$ by \eqref{defIu}, since $\I^u_n$ is open,
	\begin{equation}
	\label{eq:Cnuformula}
	(\P\otimes\P')\big(x\in{\mathcal{C}_{u,n}}\big)=(\P\otimes\P')\big(\I^u_n\cap\overline{\Sigma_n(x)} \ne \emptyset \big)=1-\E\Big[\exp\Big(-u\mathrm{{\rm cap}}_{\tilde{\G}_n}\big(\overline{\Sigma_n(x)}\big)\Big)\Big].
	\end{equation}
	As $\G_n$ is finite for each $n\in\N,$ Lemma \ref{Theforfinite} and Proposition \ref{couplingimplytheorem} imply that \hyperref[eq:laplacecap]{($\text{Law}_0$)} holds on $\tilde{\G}_{n}.$ Therefore, denoting by $\Phi$ the distribution function of a standard Gaussian random variable, by symmetry of $\phi^{(n)}$ we obtain that
	\begin{equation}
	\label{eq:lastBC}
\begin{split}
	(\P\otimes\P')\big(x\in{\mathcal{C}_{u,n}}\big)
	&\stackrel{\eqref{eq:Cnuformula},\text{\hyperref[eq:laplacecap]{($\text{Law}_0$)}}}{=}1-2\P^G_{\tilde{\G}_{n}}(\phi_{x}\geq\sqrt{2u})
	=2\Phi\big(\sqrt{2u}(g_{\tilde{\G}_{n}}(x,x))^{-1/2}\big)-1\\
	&\quad \ \tend{n}{\infty}2\Phi\big(\sqrt{2u}(g_{\tilde{\G}}(x,x))^{-1/2}\big)-1
	=(\P\otimes\P')\big(x\in{\mathcal{C}_{u,\infty}}\big),
	\end{split}
	\end{equation}
	taking advantage of the validity of \hyperref[eq:laplacecap]{($\text{Law}_0$)} for the graph $\G$ and \eqref{eq:Cnuformula} in the last equality. Hence, using \eqref{firstinclusion} and \eqref{eq:lastBC}, there exists a sequence $(b_n)_{n\in\N}$ such that for all $n\in\N,$
	\begin{equation*}
	\sum_{n\in\N}\P\otimes\P'\big(\big\{x\in{\mathcal{C}_{u,b_n}}\big\}\setminus\{x\in{\mathcal{C}_{u,\infty}}\big\}\big)<\infty,
	\end{equation*}
	and Borel-Cantelli entails that $(\P\otimes\P')$-a.s., $\limsup_{n\rightarrow\infty}\big\{x\in{\mathcal{C}_{u,b_n}}\big\}=\{x\in{\mathcal{C}_{u,\infty}}\big\}$. Using a diagonal argument and the separability of $\tilde{\G},$ we can actually choose the sequence $(b_n)_{n\in\N}$ uniformly in $x\in{\tilde{\G}}.$ Combining with \eqref{firstinclusion}, we obtain \eqref{IinftyimpliesIn}.
	
	By passing to a subsequence of $\G_n,$ $n\in\N,$ we assume without loss of generality from now on that $b_n=n$ in \eqref{IinftyimpliesIn}, which, together with \eqref{eq:equalityKapprox} and \eqref{eq:equalityKapproxRI} directly implies that 
	\begin{equation}\label{eq:lastlimit}
	\lim\limits_{n\rightarrow\infty}\Big( \phi_x^{(n)} 1_{x\notin{\mathcal{C}_{u,n}}}+\sqrt{(\phi_x^{(n)})^2+2\ell_{x,u}^{(n)}} \,  1_{x\in{\mathcal{C}_{u,n}}}\Big)
	=\phi_x^{(\infty)} 1_{x\notin{\mathcal{C}_{u,\infty}}}+\sqrt{(\phi_x^{(\infty)})^2+2\ell_{x,u}^{(\infty)}}\,  1_{x\in{\mathcal{C}_{u,\infty}}}.
	\end{equation}
	for all $x\in{\tilde{\G}}$ with $\phi_x^{(\infty)}\neq0.$
	Moreover if $\phi_x^{(\infty)}=0,$ then by \eqref{eq:equalityKapprox} and \eqref{eq:equalityKapproxRI} we have $\phi_x^{(n)}=0$ and $\ell_{x,u}^{(n)}=\ell_{x,u}^{(\infty)}$ for all $n$ large enough, and so \eqref{eq:lastlimit} remains true. Since $\G_n$ is finite for all $n\in\N,$ Lemma \ref{isomequivisom'} and Lemma~\ref{Theforfinite} yield that \eqref{eqcouplingintergffszn} holds on $\G_n$ for all $n\in\N,$ and, noting that $\phi_x^{(n)}+\sqrt{2u}\to\phi_x^{(\infty)}+\sqrt{2u}$ as $n \to \infty$ and applying \eqref{eq:lastlimit}, we infer that \eqref{eqcouplingintergffszn} holds for $\G_{\infty}=\G.$

	It remains to show that \eqref{lawofEu} holds (on ${\G}$). Fix $e \in E\cup G$. For sufficently large $n$, which we will tacitly assume henceforth, $e\in E_n\cup G_n$, where $(G_n,E_n)$ refers to the graph induced by $\G_n$. Define for all $n\in\N\cup\{\infty\}$ the random set of edges and vertices $\mathcal{E}_u^{(n)}=\{e\in{E_n\cup{G}_n}:\,2\ell_{x,u}^{(n)}+(\phi_x^{(n)})^2>0\text{ for all }x\in{I_e}\}.$ By Lemma~\ref{Theforfinite} applied to $\G_n$, we have for all $n\in\N$ that
	\begin{equation*}
	(\P\otimes\P')\big(e\in{\mathcal{E}_u^{(n)}}\,|\,\hat{\omega}_u^{(n)},\phi^{(n)}_{|G_n}\big)= 1_{e\in{\I_{E,n}^u}}\vee p_e^{\G_n}(\phi^{(n)},\ell^{(n)}_{.,u}),
	\end{equation*}
	where $\I_{E,n}^u$ is the union of the set of edges crossed by the trace $\hat{\omega}_u^{(n)}$ of $\omega_u^{(n)}$ on $G_n,$ and of the set of vertices on which a trajectory of $\hat{\omega}_u^{(n)}$ is killed. Moreover, using \eqref{phiedgeind} and \eqref{omegaedgeind}, we have that for any finite set $S \subset (E\cup G)$, conditionally on $(\phi_x^{(n)})_{x\in{{G_n}}}$ and $\hat{\omega}_u^{(n)},$ the family $\{e\in{\mathcal{E}_u^{(n)}}\},$ $e \in S,$ is independent for all large enough $n$ (including $\infty$), and for all $e\in S,$
	\begin{equation*}
	(\P\otimes\P')\big(e\in{\mathcal{E}_u^{(n)}}\,|\,\hat{\omega}_u^{(n)},\phi^{(n)}_{|G_n}\big)=(\P\otimes\P')\big(e\in{\mathcal{E}_u^{(n)}}\,|\,\hat{\omega}_{u,e}^{(n)},(\phi^{(n)})_{|e}\big).
	\end{equation*}
	Note that $(\P\otimes\P')$-a.s., for all large enough $n$, we have $(\phi^{(n)})_{|e}=(\phi^{(\infty)})_{|e}$ and $\hat{\omega}_{u,e}^{(n)}=\hat{\omega}_{u,e}^{(\infty)}$ as well as $1\{ e\in \I_{E,n}^u \}= 1\{ e\in \I_{E,\infty}^u\}$ for each $e\in S$ by \eqref{eq:equalityKapprox} and \eqref{eq:equalityKapproxRI}.  Now due to \eqref{defpe} and \eqref{defpx}, we also have $p_e^{\G_n}(\phi^{(n)},\ell^{(n)}_{.,u})=p_e^{\G}(\phi^{(\infty)},\ell^{(\infty)}_{.,u})$ for each $e\in S$ and all $n$ large enough, and so
	\begin{equation*}
	(\P\otimes\P')\big(e\in{\mathcal{E}_u^{(\infty)}}\,|\,\hat{\omega}_u^{(\infty)},\phi^{(\infty)}_{|G}\big)= 1_{e\in{\I_{E,\infty}^u}}\vee p_e^{\G}(\phi^{(\infty)},\ell^{(\infty)}_{.,u}),
	\end{equation*}
	which yields \eqref{lawofEu} for the graph $\G$ on account of \eqref{eq:GFFappro1}, \eqref{eq:RIappro1} and since $S \subset (E\cup G)$ was arbitrary.
\end{proof}

Let us now quickly explain how to deduce Theorem \ref{couplingintergff} and Corollaries \ref{dichotomy} and \ref{percoath*} from Lemma \ref{couplingisalwaystrue}.

\begin{proof}[Proof of Theorem \ref{couplingintergff}]
	We start with the proof of \eqref{equivisom}. If \eqref{eqcouplingintergff} holds, then \eqref{eq:laplacecap}$_{h>0}$ also holds by Proposition \ref{couplingimplytheorem}. If \eqref{eq:laplacecap}$_{h>0}$ holds, then \hyperref[eq:laplacecap]{($\text{Law}_0$)} also holds by taking the limit as $h\searrow0$ in \eqref{eq:laplacecap} and using \eqref{capincreasing}. If \hyperref[eq:laplacecap]{($\text{Law}_0$)} holds, then  \eqref{eqcouplingintergffszn} also holds by Lemma \ref{couplingisalwaystrue}. Since \eqref{eqcouplingintergff} and \eqref{eqcouplingintergffszn} are equivalent by Lemma \ref{isomequivisom'}, we obtain \eqref{equivisom}.
	
	Let us now assume that one of the conditions in \eqref{equivisom} holds. Then by Lemma \ref{couplingisalwaystrue}, we have that \eqref{eqcouplingintergff} and \eqref{lawofEu} hold.  Moreover, the family $\{e\in{\mathcal{E}_u}\},$ $e\in{E\cup{G}},$ is independent conditionally on $\hat{\omega}_u$ and $(\phi_x)_{x\in{G}}$ by \eqref{phiedgeind} and \eqref{omegaedgeind}, and, by \eqref{lawofEu}  we thus have that $(\mathcal{E}_u,(\phi_x)_{x\in{{G}}},\hat{\omega}_u)$ has the same law under $\tilde{\P}$ as $(\mathcal{\hat{E}}_u,(\phi_x)_{x\in G},\hat{\omega}_u)$ under $\hat{\P}.$ Finally, since by \eqref{defIu} and \eqref{capIx} $\P^I(\I^u\cap I_x\neq\varnothing)=1$ for all $x\in{G}$ with $\kappa_x>0,$ for each $x\in{G},$ we have $x\in{\mathcal{C}_u}\cap G$ if and only if there is a path $\pi\subset\mathcal{E}_u\cap E$ between $x$ and some $y\in{(\I^u\cup \mathcal{E}_u)\cap G},$ and so $(\sigma_x)_{x\in{G}}$ and $\hat{\sigma}$ also have the same law.  The equality \eqref{eqcouplingintergffdis} then follows directly from \eqref{eqcouplingintergff}.
\end{proof}

\begin{proof}[Proof of Corollary \ref{dichotomy}]
	Let $\G$ be a graph such that \hyperref[eq:laplacecap]{($\text{Law}_0$)} is fulfilled. Then \eqref{eqcouplingintergffszn} holds by \eqref{equivisom}. Let us assume that ${E}^{\geq0}$ contains at least one non-compact component with positive probability. In particular, there exists $x_0\in{\tilde{\G}}$ such that $E^{\geq 0}(x_0)$ is non-compact with positive probability. By Theorem \ref{mainresult}, we know that $\mathrm{{\rm cap}}(E^{\geq 0}(x_0))<\infty$ $\P^G$-a.s, and so by Lemma \ref{1stpartofmaincor}, $E^{\geq 0}(x_0)$ is also unbounded with positive probability. Now, by \eqref{defIu}, it follows that for all $u>0,$ with $({\P}^I\otimes\P^G)$-positive probability, $E^{\geq0}(x_0)$ is unbounded and $x_0\notin{\mathcal{C}_{u}}.$ By \eqref{eqcouplingintergffszn} and symmetry of the Gaussian free field, we obtain that for all $u>0$ $E^{\geq\sqrt{2u}}(x_0)$ is unbounded with positive probability. In particular, if $\tilde{h}_*^{{\rm com}} >0,$ then $E^{\geq0}$ contains a non-compact component with positive probability, and so $E^{\geq h}$ contains an unbounded component for all $h>0$ by the above reasoning, that is $\tilde{h}_*=\infty.$ If moreover $\h_{\text{kill}}<1,$ then $\tilde{h}_*\geq0$ by \eqref{ifhkill<1thenh_*>0}.  Therefore by \eqref{capcombougen}, we have $\tilde{h}_*^{{\rm com}} \geq \tilde{h}_*\geq0.$ Since $\tilde{h}_*=\infty$ if $\tilde{h}_*^{{\rm com}} >0,$ we thus obtain $\tilde{h}_*=\tilde{h}_*^{{\rm com}} \in{\{0,\infty\}}.$
\end{proof}

\begin{proof}[Proof of Corollary \ref{percoath*}]
	Let us assume that $\tilde{h}_*\leq0,$ then $E^{\geq h}$ is $\P^G$-a.s.\ bounded for all $h>0.$ By Theorem \ref{mainresultcap}, we thus have that \eqref{eq:laplacecap} holds for all $h>0,$ and so \hyperref[eq:laplacecap]{($\text{Law}_0$)} also holds by \eqref{equivisom}. Since  $E^{\geq h}$ is $\P^G$-a.s.\ bounded for all $h>0,$ we thus obtain by Corollary \ref{dichotomy} that $E^{\geq0}$ is $\P^G$-a.s.\ bounded.
\end{proof}

\begin{Rk}
	\label{remarkend}
	\begin{enumerate}[label={\arabic*)}]
		\item From Proposition \ref{couplingimplytheorem}, Corollary \ref{C:thm2} and Lemma \ref{couplingisalwaystrue}, one could immediately prove again Theorem \nolinebreak \ref{mainresultcap} (which however does not require accessing to the information  \eqref{eqcouplingintergff} on $\tilde{\G}$).
		\item Similarly to Theorem 8 of \cite{LuSaTa}, one could also use \eqref{eqcouplingloopsgff} to deduce an isomorphism theorem between random interlacements and the Gaussian free field even if $G$ is infinite. More precisely, if $\G$ is a graph such that $|\{x\in{G}:\kappa_x>0\}|<\infty,$ one can merge all the open ends of the cables $I_x,$ $x\in{G}$ with $\kappa_x>0,$ into a new vertex $x_*,$ and apply \eqref{eqcouplingloopsgff} to the new (locally finite) graph $\G\cup\{x_*\}.$ Decomposing the loop soup into loops hitting $x_*$ and loops avoiding $x_*$ similarly as in Appendix \ref{App:isom}, one can then prove an isomorphism similar to Theorem \ref{couplingintergff}, but replacing random interlacements on $\tilde{\G}$ by killed random interlacements on $\tilde{\G},$ that is all the trajectories in the random interlacement process whose forward and backward parts are both killed before escaping all bounded sets, and replacing $\phi+\sqrt{2u}$ by $\phi+\sqrt{2u}\h_{\text{kill}},$ see \eqref{defh0}. In Corollary \ref*{Pre1:h0transformiso} of \cite{Pre1}, this isomorphism between killed random interlacements and the Gaussian free field is extended to any graphs satisfying \hyperref[eq:laplacecap]{($\text{Law}_0$)}.
		\item An interesting open question is whether a transient graph $\G$ exists such that \hyperref[eq:laplacecap]{($\text{Law}_0$)} does not hold, or any of the other equivalent conditions appearing in \eqref{equivisom}. In view of Corollary \ref{dichotomy}, one could also ask if a transient graph $\G$ exists, such that $\h_{\text{kill}}<1$ is fulfilled, but $\tilde{h}_*\in{(0,\infty)}$ or $\tilde{h}_*^{{\rm com}} \in{(0,\infty)},$ and then \hyperref[eq:laplacecap]{($\text{Law}_0$)} would not hold. On such a graph, we would still have by Theorem \ref{mainresultcap} that \eqref{eq:laplacecap} holds for all $h>\tilde{h}_*^{{\rm com}} .$
	\end{enumerate}
\end{Rk}

\appendix
\renewcommand*{\theThe}{A.\arabic{The}}

\section{Appendix: the condition \eqref{capcondition}}
\label{subsec:capandkappa=0}

We gather in this section various pertinent observations around the condition  \eqref{capcondition} appearing on p.\pageref{capcondition}, including a proof of Lemma \ref{equivcapcondition}. 
The following result is simple but useful in absence of any quantitative information on the asymptotic behavior of $g(\cdot,\cdot)$.
\begin{Lemme}[Decay of Green's function] \label{L:Acap1}
If $A\subset G$ is an infinite set, then for all sequences $x_n \in A$, $n \geq 0$, such that $ \lim_n d_{\G}(x_0,x_n)= \infty$ and $g(x_n,x_n) \leq g_0 \in (0,\infty)$ for all (but finitely many) $n$, one has
\begin{equation}
\label{eq:Acap3}
g(x_0,x_n) \to 0, \text{ as } n\to \infty.
\end{equation}

\end{Lemme}

\begin{proof}
We argue by contradiction. Suppose that for some $\varepsilon >0$ and some $x_n \in A$ with $g(x_n,x_n)\leq g_0,$ $n \geq 0$ and $\lim_nd_{\G}(x_0,x_n)=\infty$
\begin{equation}
\label{eq:Acap4}
g(x_0,x_n)\geq \varepsilon, \text{ for all } n \geq 0. 
\end{equation}
By passing to a subsequence we may also assume that $d_{\G}(x_0,x_{n})> n,$ for all $n \geq 0$. Let $\hat{H}_y=\inf\{n\in\N:\hat{Z}_n=y\}$ the hitting time of $y$ for the discrete skeleton $\hat{Z},$ with $\inf\varnothing=\infty.$ Since for all $x,y\in{{G}}$ $g(x,y)=P_x(\hat{H}_y<\infty)g(y,y)$ and $g$ is symmetric, one then has by \eqref{eq:Acap4}
\begin{equation}
\label{eq:Acap5}
    P_{x_0}(\hat{H}_{x_n}<\infty){\geq} \, g_0^{-1}\varepsilon\text{ and } P_{x_n}(\hat{H}_{x_0}<\infty){\geq} \, g_0^{-1}\varepsilon ,\text{ for all }n \geq 0.
\end{equation}
Since $d_{\G}(x_0,x_{n})> n$, \eqref{eq:Acap5} and the strong Markov property then imply, for all $n \geq 0$
\begin{equation*}
    P_{x_0}(\exists\,p\geq \hat{T}_{B(x_0,n)}:\,\hat{Z}_p=x_0)\geq P_{x_0}(H_{x_{n}}<\infty,\exists\,p\geq \hat{H}_{x_{n}}:\,\hat{Z}_p=x_0)\geq g_0^{-2}\eps^2,
\end{equation*}
where $\hat{T}_{B(x,n)}=\inf\{p\in\N:\,\hat{Z}_p\notin B(x,n)\}$ is the first exit time of the discrete ball $B(x,n)$ for the graph distance $d_{\G}$ on $\G,$ with $\inf\varnothing=\infty.$ Since $\hat{T}_{n} = \hat{T}_{B(x_0,n)}$ increases to $\infty,$ there exists a sequence $(n_k)_{k\geq 0}$ such that
\begin{equation*}
    P_{x_0}\big(\exists\,p\in\{\hat{T}_{n_k},\dots,\hat{T}_{n_{k+1}}-1\}:\,\hat{Z}_p=x_0\big)\geq\frac{\eps^2}{2g_0^2}, \text{ for all $k \geq 0$,}
\end{equation*}
whence
\begin{equation*}
    g(x_0,x_0)=\frac{1}{\lambda_{x_0}}E_{x_0}\Big[\sum_{p=0}^{\infty} 1_{\hat{Z}_p=x_0}\Big]\geq\frac{1}{\lambda_{x_0}}\sum_{k \geq 0}E_{x_0}\Big[\sum_{\hat{T}_{n_k}\leq p <\hat{T}_{n_{k+1}}} 1_{\hat{Z}_p=x_0} \Big]\geq\sum_{k\geq 0}\frac{\eps^2}{2g_0^2\lambda_{x_0}}=\infty,
\end{equation*}
a contradiction to the transience of $\mathcal G.$
\end{proof}

The utility of a control like \eqref{eq:Acap3} is illustrated by the following criterion. 

\begin{Lemme}[Criterion for infinite capacity] \label{L:Acap2}
If $A\subset G$ satisfies
\begin{equation}
\label{eq:Acap9}
\begin{cases}
|A|=\infty, \text{ and }\\[0.3em] 
g(x,x) \leq g_0, \text{ for all }x \in A,
\end{cases}
\end{equation}
then $\textnormal{cap}(A)=\infty$. 
\end{Lemme}
\begin{proof}
A proof of this can be found in~\cite{Hu19}, Lemma~2.13. We give a different proof.
Let $\varepsilon > 0$ and $n \geq 1$. Consider the `refined' set $A_{\varepsilon,n}=\{x_0,\dots,x_n\} \subset A$ defined as follows. Fix $x_0 \in A$ arbitrary. Given $\{x_0,\dots,x_{k-1}\}$ for some $1\leq k < n $, applying Lemma~\ref{L:Acap2}, which is in force due to \eqref{eq:Acap9}, we find by means of \eqref{eq:Acap3} a point $x_k \in A$ such that $g(x_k,x_{k'})< \varepsilon$ for all $k' < k$. Overall it follows that
\begin{equation}
\label{eq:Acap10}
g(x,y) \leq \varepsilon, \text{ for all } x\neq y \in A_{\varepsilon,n}.
\end{equation}
Now, by the variational principle \eqref{variational} and by monotonicity, see \eqref{capincreasing}, one obtains that
\begin{equation*}
      \mathrm{{\rm cap}}(A) \geq \mathrm{{\rm cap}}( A_{\varepsilon,n})\geq\Big(\frac{1}{n^2}\sum_{x,y\in  A_{\varepsilon,n}}g(x,y)\Big)^{-1}\stackrel{\eqref{eq:Acap9},\eqref{eq:Acap10}}{\geq} \left(\frac{g_0}{n}+\eps\right)^{-1}.
\end{equation*} 
from which $\mathrm{{\rm cap}}(A)=\infty$ follows by letting first $n\to \infty$ and then $\varepsilon \to 0$.
\end{proof}

We conclude this section with the

\begin{proof}[Proof of Lemma ~\ref{equivcapcondition}]
\ref{equivcapcondition_a}
Let us first assume that \eqref{capcondition} holds true for the graph $\G.$ In this case, for all infinite and connected $A\subset G,$ writing $\tilde{A}$ for the union of the $I_e$ for all edges $e\in{E}$ between two vertices of $A,$ we have by \eqref{defequilibriumcable} and \eqref{defcapinfinity}
\begin{equation*}
    \mathrm{{\rm cap}}(A)=\mathrm{{\rm cap}}(\tilde{A})=\infty,
\end{equation*}
since $\tilde{A}$ is an unbounded and connected set of $\tilde{\G},$ and so \eqref{capconditiondis} is satisfied. Assume now that $\G$ is a graph such that \eqref{capconditiondis} is verified, and let $\tilde{A}$ be a connected and unbounded subset of $\tilde{\G}.$ Then $\tilde{A}$ contains an infinite and connected set $A\subset G,$ and so by \eqref{capincreasing} and \eqref{capconditiondis} $\mathrm{{\rm cap}}(\tilde{A})\geq\mathrm{{\rm cap}}(A)=\infty,$ that is \eqref{capcondition} holds.
    
\ref{equivcapcondition_b}  
By \eqref{gbounded}, the set $A'\stackrel{\text{def.}}{=}\{ x \in A : g(x,x)\leq g_0\}$ is infinite for any infinite and connected sets $A\subset G.$ 
Thus, $A'$ satisfies \eqref{eq:Acap9}, and Lemma~\ref{L:Acap2} yields that $(\textnormal{cap}({A})\geq) \, \textnormal{cap}({A'})=\infty$. Hence by Lemma \ref{equivcapcondition},\ref{equivcapcondition_a}, \eqref{capcondition} holds. If $\G$ is vertex-transitive, then $g(x,x)=g_0$ is constant, and so \eqref{gbounded} holds.

\ref{capconditionontrees} By Lemma \ref{equivcapcondition},\ref{equivcapcondition_a},\eqref{capincreasing} and \eqref{defcapinfinity}, it is enough to prove that $\mathrm{{\rm cap}}(B)=\infty$ for all infinite and connected sets $B$ containing exactly one vertex per generation. Let us fix some $x_0\in{B},$ and for all $i\geq0$ define recursively $x_{i+1}$ as the first descendant $x\in{B}$ of $x_i$ in $B$ such that $\mathbb{T}_{x}\setminus B$ is infinite. Note that such a vertex $x_{i+1}$ must exists, otherwise $R^{\infty}_x=\infty$ for all descendants $x$ of $x_i$ in $B.$ For each $i\in\N,$ $\{x\in{\mathbb{T}_{x_i}\setminus B}:\,R_x^{\infty}>A\}$ is finite, and so there exists a cut-set $C_i$ between $x_i$ and infinity in $\mathbb{T}_{x_i}\setminus B,$ such that $R_y^{\infty}\leq A$ for all $y\in{C_i}.$ Taking $B_n=\{x_0,\dots,x_n\},$ we have for all $n\in\N$ and $i\in{\{1,\dots,n-1\}}$ that
\begin{align*}
    e_{B_n}(x_i)&=\lambda_{x_i}P^{\mathbb{T}}_{x_i}(Z_n\in{\mathbb{T}_{x_i}\setminus B}\text{ for all }n\in\N)
    \\&\geq \lambda_{x_i}\sum_{y\in{C_i}}P^{\mathbb{T}}_{x_i}(Z_{H_{C_i}}=y,H_{C_i}<\tilde{H}_{x_i})P_{y}(H_{y^-}=\infty)
    \\&\geq \lambda_{x_i}\sum_{y\in{C_i}}P^{\mathbb{T}}_{x_i}(Z_{H_{C_i}}=y,H_{C_i}<\tilde{H}_{x_i})\frac{1}{1+R_y^{\infty}}
    \geq\frac{\lambda_{x_i}}{1+A}P^{\mathbb{T}}_{x_i}(H_{C_i}<\tilde{H}_{x_i}),
\end{align*}
where $y_i^-$ is the first ancestor of $y_i$ and we used (1.11) in \cite{MR3765885} in the second inequality. Since $\mathbb{T}$ is transient and the random walk on $\Z$ is recurrent, it is easy to see that $B$ is visited infinitely often with probability $0.$  Therefore, for each $i\in\N,$ under $P^{\mathbb{T}}_{x_i},$ if $Z_n\in{\mathbb{T}_{x_i}}$ for all $n\in\N,$ then there exists $p\geq i$ such that $H_{C_p}<\tilde{H}_{x_i},$ and so
\begin{align}
\label{boundRxi}
    \frac{1}{R_{x_i}^{\infty}}&\leq\lambda_{x_i}P^{\mathbb{T}}_{x_i}(\exists\,p\geq i,H_{C_p}<\tilde{H}_{x_i})
    \leq \sum_{p\geq i}\lambda_{x_p}P^{\mathbb{T}}_{x_p}(H_{C_p}<\tilde{H}_{x_i}),
\end{align}
where in the last inequality we used $\lambda_{x_i}P^{\mathbb{T}}_{x_i}(H_{x_p}<\tilde{H}_{x_i})=\lambda_{x_p}P^{\mathbb{T}}_{x_p}(H_{x_i}<\tilde{H}_{x_p})\leq\lambda_{x_p}.$ Moreover, for all $y\in{B}$ between $x_i$ and $x_{i+1},$ the effective resistance between $y$ and $\infty$ in $\mathbb{T}_{x_i}$ is $R_y^{\infty},$ and so using a series transformation we have $R_y^{\infty}\geq R_{x_{i+1}}^{\infty}.$ Therefore, since $B$ is an unbounded and connected set, we have $R_{x_i}^{\infty}\leq A$ infinitely often, and so the sum on the right-hand side of \eqref{boundRxi} must be infinite. Using \eqref{defcap} and \eqref{defcapinfinity} we conclude that
\begin{equation*}
    \mathrm{{\rm cap}}(B)=\lim\limits_{n\rightarrow\infty}\sum_{i\in{\{0,\dots,n\}}}e_{B_n}(x_i)\geq \frac{1}{1+A}\sum_{i\in\N}\lambda_{x_i}P^{\mathbb{T}}_{x_i}(H_{C_i}<\tilde{H}_{x_i})=\infty,
\end{equation*}
which completes the proof.
\end{proof}

\section{Appendix: Proof of Lemma \ref{Theforfinite}} \label{App:isom}
\renewcommand{\theequation}{B.\arabic{equation}}
\renewcommand*{\theThe}{B.\arabic{The}}

In this Appendix we are going to prove that the coupling between loop soups and the Gaussian free field, \eqref{eqcouplingloopsgff}, implies the coupling between random interlacements and the Gaussian free field on finite graphs, Lemma \ref{Theforfinite}, following similar ideas to the proof of Theorem 8 in \cite{LuSaTa}. Let us define
\begin{equation*}
    U_\kappa\stackrel{\text{def.}}{=}\{x\in{{G}}:\kappa_x>0\}
\end{equation*}
and let $\G^*$ be the graph with vertex set $G,$ plus an additional vertex $x_*.$ The symmetric weights on $\G^*$ are 
\begin{equation*}
    \lambda_{x,y}^{*}=\begin{cases}
         \lambda_{x,y}&\text{ when }x,y\in{G}\\
         \kappa_x&\text{ when }x\in{U_{\kappa}}\text{ and }y=x_*\\
         0&\text{ when }x\notin{U_{\kappa}}\text{ and }y=x_*,
    \end{cases}
\end{equation*}
and the killing measure $\kappa^{*}= 1_{x_*}.$ We write $G^*=G\cup\{x_*\}$ and $E^*=\{\{x,y\}\in{G^* \times G^*}:\lambda_{x,y}^*>0\}$ for the vertex and edge set of $\G^*.$ Note that each edge $I_e$ of $\tilde{\G}^*,$ $e\in{E^*},$ can be identified with some edge $I_e$ of $\tilde{\G},$ $e\in{E\cup U_{\kappa}},$ and one can then identify the cable system $\tilde{\G}^*\setminus\{I_{x_*}\cup\bigcup_{x\in{U_{\kappa}}}I_x\}$ with $\tilde{\G}.$ By \eqref{lawkappa'vskappa}, for all $x\in{\tilde{\G}}$ the law of the trace of $X$ on $\tilde{\G}$ killed on hitting $x_*$ under $P^{\tilde{\G}^*}_x$ is thus $P^{\tilde{\G}}_x.$ Recall the decomposition of the loop soup $\tilde{\mathcal{L}}_{\frac12}=\tilde{\mathcal{L}}_{\frac12}^{\, \textnormal{in}}+\widetilde{\mathcal{L}}_{\frac12}^{\, *}$ on $\G^*$ defined below Lemma \ref{Theforfinite}, and let $({L}_x^{*})_{x\in{\tilde{\G}^*}}$ be the local times of $\widetilde{\mathcal{L}}_{\frac12}^{\, *}$ under $\P^L_{\tilde{\G}^*},$ and ${\mathcal{L}}_{\frac12}^{\, *}$ be the trace of $\widetilde{\mathcal{L}}_{\frac12}^{\, *}$ on $G^*.$ Each loop in $\mathcal{L}_{\frac12}^{\, *}$ can be decomposed into its excursions outside $x_*,$ that is a trajectory entirely contained in ${G},$ starting and ending in $U_{\kappa},$ and the process $\mathcal{L}_{\frac12}^{e,*}$ of excursions is then defined as the point process consisting of all the excursions outside $x_*$ for all the loops in $\mathcal{L}_{\frac12}^{\, *}.$ We can now compare the Gaussian free field on $\tilde{\G}^*$ with the Gaussian free field on $\tilde{\G},$  and the loops $\widetilde{\mathcal{L}}_{\frac12}^{\, *}$ hitting $x_*$ on $\tilde{\G}^*$ with random interlacements on $\tilde{\G}.$
    
\begin{Prop}
Let $\G$ be a transient graph such that $G$ is finite. For any $u>0,$
\begin{equation}
\label{phinonG*}
    (\phi_x)_{x\in\tilde{\G}}\text{ has the same law under } \P_{\tilde{\G}^*}^G(\cdot\,|\,\phi_{x_*}=\sqrt{2u})\text{ as }(\phi_x+\sqrt{2u})_{x\in\tilde{\G}}\text{ under }\P_{\tilde{\G}}^G,
\end{equation}
and
\begin{equation}
    \label{RidisonG*}
    \mathcal{L}_{\frac12}^{e,*}\text{ has the same law under }\P_{\tilde{\G}^*}^L(\cdot\,|\,L_{x_*}=u)\text{ as }\hat{\omega}_u\text{ under }\P_{\G}^I.
\end{equation}
In particular,
\begin{equation}
\label{RIonG*}
    ({L}_x^{*})_{x\in{\tilde{\G}}}\text{ has the same law under }\P_{\tilde{\G}^*}^{L}(\cdot\,|\,L_{x_*}=u)\text{ as }(\ell_{x,u})_{x\in{\tilde{\G}}}\text{ under }\P_{\tilde{\G}}^{I}. 
\end{equation}
\end{Prop}

\begin{proof}    
    We begin with \eqref{phinonG*}. By the Markov property applied to the graph $\G^*,$ see \eqref{Markov}, conditionally on $\A_{\{x_*\}}^+,$ $(\phi_x)_{x\in{\tilde{\G}}}$ is a Gaussian field with mean $\eta_{\{x_*\}}^\phi=\phi_{x_*}$ and variance $g_{\{x_*\}^c}=g_{\tilde{\G}},$ and thus $(\phi_x-\phi_{x_*})_{x\in{\tilde{\G}}}$ has the same law under $\P^G_{\tilde{\G}^*}(\cdot\,|\,\A_{\{x_*\}}^+)$ as $\phi$ under $\P^G_{\tilde{\G}},$ and \eqref{phinonG*} follows. 
    
    Let us now prove \eqref{RidisonG*}. By Proposition 3.7 in \cite{LuSaTa}, conditionally on $L_{x_*}(=L_{x_*}^*)=u,$ the excursions outside $x_*$ in ${\mathcal{L}}_{\frac12}^{*}$ have the same law as the excursions of the Markov jump process $Z$ outside $x_*$ stopped when reaching local time $u$ at $x_*$ under $P_{x_*}^{\tilde{\G}^*}(Z\in{\cdot}\,|\,\ell_{x_*}(\zeta)>u),$ which can be described as follows: first stay an exponential time with parameter $\lambda_{x_*}^*-\kappa_{x_*}^*$ in $x_*,$ then jump to an $x\in{U_{\kappa}}$ with probability $\frac{\kappa_{x}}{\lambda_{x_*}^*-\kappa_{x_*}^*}$ and follow on ${G}$ a process with the same law as $Z$ under $P^{\tilde{\G}}_x.$ Once this process is killed, jump back to $x_*$ and iterate this process until reaching local time $u$ at $x_*.$ By a property of exponential variables, the number of time this process is iterated is a Poisson variable with parameter $u(\lambda_{x_*}^*-\kappa_{x_*}^*),$ and thus, conditionally on $L_{x_*}=u,$ ${\mathcal{L}}_{\frac12}^{e,*}$ is a Poisson point process with intensity
    \begin{equation*}
        u\sum_{x\in{U_\kappa}}{\kappa_{x}}P^{\tilde{\G}}_{x}(Z\in{\cdot}).
    \end{equation*} 
    Note that, under $P^{\tilde{\G}}_x,$ we have $\tilde{H}_{{G}}=\infty$ if and only if $x\in{U_{\kappa}}$ and the discrete skeleton $\hat{Z}$ of $Z$ is killed at time $1,$ and thus $e_{{G}}(x)=\kappa_x$ for all $x\in{U_{\kappa}}$ and $e_{{G}}(x)=0$ otherwise. Therefore by \eqref{defQK} and \eqref{definter} with $K=G,$ conditionally on $L_{x_*}=u,$ $\mathcal{L}_{\frac12}^{e,*}$ is a Poisson point process with intensity $u\nu_{\G},$ where $\nu_{\G}$ is the print on $G$ of the intensity measure $\nu_{\tilde{\G}}$ of random interlacements, and we obtain \eqref{RidisonG*}. This implies in particular that $({L}_x^{*})_{x\in{{{G}}}}$ has the same law under $\P_{\tilde{\G}^*}^{L}(\cdot\,|\,L_{x_*}=u)$ as $(\ell_{x,u})_{x\in{{{G}}}}$ under $\P_{\tilde{\G}}^{I},$ and thus \eqref{RIonG*} follows by considering the graph $\G^A$ for any finite subset $A$ of $\tilde{\G},$ see Lemma \ref{GA}.
\end{proof}

Using \eqref{eqcouplingloopsgff} for the graph $\G^*,$ and decomposing $\tilde{\mathcal{L}}_{\frac12}$ on $\tilde{\G}^*$ into $\tilde{\mathcal{L}}_{\frac12}^{\, \textnormal{in}}$ and $\widetilde{\mathcal{L}}_{\frac12}^{\, *},$ we are now ready to prove Lemma \ref{Theforfinite}.

\begin{proof}[Proof of Lemma \ref{Theforfinite}]
Let us define $(L_x^{\textnormal{in}})_{x\in{\tilde{\G}^*}}$ the total local times of the loops in $\tilde{\mathcal{L}}_{\frac12}^{\, \textnormal{in}}$ under $\P^L_{\tilde{\G}^*}.$ By \eqref{looprestriction}, $(L_x^{\textnormal{in}})_{x\in{\tilde{\G}}}$ has the same law as the restriction to $\tilde{\G}$ of the local time of a loop soup on $\tilde{\G}^{\, \{ x_*\}^c}_{\infty},$ and thus the same law as the local time of a loop soup on $\tilde{\G}.$ By \eqref{eqcouplingloopsgff}, $(L_x^{\textnormal{in}})_{x\in{\tilde{\G}}}$ has thus the same law under $\P^L_{\tilde{\G}^*},$ or also $\tilde{\P}^L_{\tilde{\G}^*}(\cdot\,|\,\sigma_{x_*}=1,{L}_{x_*}^{*}=u),$ as $\frac12\phi^2$ under $\P^G_{\tilde{\G}}.$ Moreover, under $\tilde{\P}^L_{\tilde{\G}^*}(\cdot\,|\,\sigma_{x_*}=1,{L}_{x_*}^{*}=u),$ using the equality $L_x=L_x^{\textnormal{in}}+{L}_x^{*}$ for all $x\in{\tilde{\G}},$ the law of $(\sigma_x)_{x\in{\tilde{\G}}}$ can be described as follows: conditionally on $({L}_x^{\textnormal{in}})_{x\in{\tilde{\G}}}$ and $({L}_x^{*})_{x\in{\tilde{\G}}},$ $\sigma$ is constant on each cluster of $\{x\in{\tilde{\G}}:{L}_x^{*}+{L}_x^{\textnormal{in}}>0\},$ with $\sigma_x=1$ for all $x\in{\tilde{\G}}$ such that ${L}_x^{*}>0,$ and the values of $\sigma$ on each other cluster are independent and uniformly distributed.  Using \eqref{RIonG*} we thus have that, under $\tilde{\P}^L_{\tilde{\G}^*}(\cdot\,|\,\sigma_{x_*}=1,{L}_{x_*}^{*}=u),$
\begin{equation*}
    ({\sigma}_x\sqrt{2L_x})_{x\in{\tilde{\G}}}\text{ has the same law as }\big(\sigma_x^u\sqrt{2\ell_{x,u}+\phi_x^2}\big)_{x\in{\tilde{\G}}}\text{ under }\tilde{\P}_{\tilde{\G}}.
\end{equation*}
According to \eqref{eqcouplingloopsgff}, the law of $({\sigma}_x\sqrt{2L_x})_{x\in{\tilde{\G}}}$ under $\tilde{\P}^L_{\tilde{\G}^*}(\cdot\,|\,{\sigma}_{x_*}=1,L_{x_*}=u)$ is the same as the law of $(\phi_x)_{x\in{\tilde{\G}}}$ under ${\P}^G_{\tilde{\G}^*}(\cdot\,|\,\phi_{x_*}=\sqrt{2u}),$ and thus by \eqref{phinonG*} the same as the law of $(\phi_x+\sqrt{2u})_{x\in{\tilde{\G}}}$ under $\P_{\tilde{\G}}^G,$ and we obtain \eqref{eqcouplingintergff}.

By \eqref{phiedgeind} and \eqref{omegaedgeind}, it is clear that, conditionally on $\hat{\omega}_u$ and $(\phi_x)_{x\in{{G}}},$ the family $\{e\in{\mathcal{E}_u}\},$ $e\in{E\cup{G}},$ is independent, and we now turn to the proof of \eqref{lawofEu}.  Let $\mathcal{E}^{\textnormal{in}}=\{e\in{E^*}:L_x^{\textnormal{in}}>0\text{ for all }x\in{I_e}\},$ and, conditionally on $\mathcal{E}^{\textnormal{in}},$  let $(\sigma^{\textnormal{in}}_x)_{x\in{G}}$ be an independent additional process, such that $\sigma^{\textnormal{in}}$ is constant on each cluster induced by $\mathcal{E}^{\textnormal{in}},$ and its values on each cluster are independent and uniformly distributed. Note that the clusters of $G$ induced by $\mathcal{E}^{\textnormal{in}}$ are the same as the intersection with $G$ of the clusters of $\{x\in{\tilde{\G}^*}:\,L_x^{\textnormal{in}}>0\},$ and so by \eqref{looprestriction} and \eqref{eqcouplingloopsgff}, $(\sigma^{\textnormal{in}}_x\sqrt{2L^{\textnormal{in}}_x})_{x\in{{G}}}$ has the same law under ${\P}^L_{\tilde{\G}^*}(\cdot\,|\,L_{x_*}=u)$ as $\phi$ under ${\P}^G_{\tilde{\G}},$ and by \eqref{RIonG*}, $\mathcal{E}$ has the same law as $\mathcal{E}'_u:=\mathcal{E}_u\setminus\{I_x,x\in{G\setminus U_{\kappa}}\}$ under $\tilde{\P}_{\tilde{\G}},$ where $\mathcal{E}$ is defined in \eqref{hatEsameasE}. Therefore using \eqref{RidisonG*} we obtain that
\begin{equation}
\label{disloopsvsinter}
\big(\mathcal{E},\big(\sigma_x^{\text{in}}\sqrt{2L_{x}^{\text{in}}}\big)_{x\in{G}},\mathcal{L}_{\frac12}^{e,*}\big)\text{ has the same law under }\tilde{\P}^L_{\tilde{\G}^*}(\cdot\,|\,L_{x_*}=u)\text{ as }\big(\mathcal{E}_u',\phi_{|{G}},\hat{\omega}^u\big)\text{ under }\tilde{\P}.
\end{equation}
For each $e\in{E^*},$ the event $\{e\notin{\mathcal{E}^{\textnormal{in}}}\}$ is independent of ${\mathcal{L}}^{*}_{\frac12},$ and, conditionally on $\{e\notin{\mathcal{E}^{\textnormal{in}}}\},$ ${\mathcal{L}}^{*}_{\frac12}$ and $L_{|G}^{\textnormal{in}}=(L_x^{\textnormal{in}})_{x\in{G}},$ the event $\{e\notin\mathcal{E}\}$ is independent of $\sigma^{\textnormal{in}}_{|G}=(\sigma^{\textnormal{in}}_x)_{x\in{G}}.$ Therefore, since $\{e\notin{\mathcal{E}}\}\subset\{e\notin{\mathcal{E}^{\textnormal{in}}}\},$ we obtain
\begin{equation}
\label{decomposeEinEandE*}
    \tilde{\P}^L_{\tilde{\G}^*}\big(e\notin{{\mathcal{E}}}\,|\,{\mathcal{L}}^{*}_{\frac12},L_{|G}^{\textnormal{in}},\sigma^{\textnormal{in}}_{|G}\big)=\tilde{\P}^L_{\tilde{\G}^*}\big({e\notin{{\mathcal{E}^{\textnormal{in}}}}}\,|\,L_{|G}^{\textnormal{in}},\sigma^{\textnormal{in}}_{|G}\big)\tilde{\P}^L_{\tilde{\G}^*}\big(e\notin{\mathcal{E}}\,|\,{\mathcal{L}}^{*}_{\frac12},L_{|G}^{\textnormal{in}},e\notin{{\mathcal{E}^{\textnormal{in}}}}\big).
\end{equation}
Now, since $(\sigma_x^{\textnormal{in}}\sqrt{2L_x^{\textnormal{in}}},\{e\notin{\mathcal{E}^{\textnormal{in}}}\})$ has the same law as $((\phi_x)_{x\in{G}},\{\forall y\in{I_e}:|\phi_y|>0\}^c)$ under $\P^G_{\tilde{\G}},$ it follows from  \eqref{phionIx} that for all $e\in{E^*},$
\begin{equation}
\label{enotinex*}
    \tilde{\P}^L_{\tilde{\G}^*}\Big({e\notin{{\mathcal{E}^{\textnormal{in}}}}}\,|\,\big(\sigma_x^{\textnormal{in}}\sqrt{2L_x^{\textnormal{in}}}\big)_{x\in{G}}\Big)=p_e^{\G}(\sigma^{\textnormal{in}}\sqrt{2L^{\textnormal{in}}}) 1_{e\in{E}}+ 1_{e\notin{E}},
\end{equation}
where we identified $e$ with the corresponding edge or vertex of $E\cup G.$ Let us write $\I_E^{\mathcal{L}}\subset E\cup G$ for the set of edges of $\G$ crossed by at least one single trajectory in $\mathcal{L}^{e,*}_{\frac12},$ union with the set of vertices of $G$ at which a trajectory in $\mathcal{L}^{e,*}_{\frac12}$ is killed, which corresponds to the set of edges of $\G^*$ crossed by at least one single trajectory in $\mathcal{L}^*_{\frac12}.$  Now since $\{e\notin{\mathcal{E}^{\textnormal{in}}}\}$ is independent of ${\mathcal{L}}_{\frac12}^{\, *},$ we have by \eqref{hatEsameasE} that for all edges $e\in{E}$
\begin{align*}
    \tilde{\P}^L_{\tilde{\G}^*}\big(e\notin{\mathcal{E}}\,|\,{\mathcal{L}}^{*}_{\frac12},L_{|G}^{\textnormal{in}},e\notin{{\mathcal{E}^{\textnormal{in}}}}\big)&=\frac{\tilde{\E}^L_{\tilde{\G}^*}\big[\tilde{\P}^L_{\tilde{\G}^*}\big(e\notin{\mathcal{E}}\,|\,{\mathcal{L}}_{\frac12}\big)\,|\,{\mathcal{L}}^{*}_{\frac12},L_{|G}^{\textnormal{in}}\big]}{\tilde{\E}^L_{\tilde{\G}^*}\big[\tilde{\P}^L_{\tilde{\G}^*}\big(e\notin{\mathcal{E}^{\textnormal{in}}}\,|\,{\mathcal{L}}^{\, \textnormal{in}}_{\frac12}\big)\,|\,L_{|G}^{\textnormal{in}}\big]} =\frac{p_e^{\G^*}\big(\sqrt{L^{\textnormal{in}}+{L}^{*}}\big)}{p_e^{\G}(\sqrt{L^{\textnormal{in}}})} 1_{e\notin{\I_E^{\mathcal{L}}}}.
\end{align*}
Combining with \eqref{decomposeEinEandE*} and \eqref{enotinex*}, we thus obtain that for all edges $e\in{E},$ 
\begin{equation}
\label{equalitypeG}
\begin{split}
    \tilde{\P}^L_{\tilde{\G}^*}\big(e\notin{{\mathcal{E}}}\,|\,{\mathcal{L}}^{*}_{\frac12},L_{|G}^{\textnormal{in}},\sigma^{\textnormal{in}}_{|G}\big)&=\frac{p_e^{\G^*}\big(\sqrt{L^{\textnormal{in}}+{L}^{*}}\big)p_e^{\G}\big(\sigma^{\textnormal{in}}\sqrt{2L^{\textnormal{in}}}\big)}{p_e^{\G}\big(\sqrt{L^{\textnormal{in}}}\big)} 1_{e\notin{\I_E^{\mathcal{L}}}}
    \\&=p_e^{\G}\big(\sigma^{\textnormal{in}}\sqrt{2L^{\textnormal{in}}},{L}^{*}\big) 1_{e\notin{\I_E^{\mathcal{L}}}},
\end{split}
\end{equation}
where we used \eqref{defpe} and \eqref{defpef} in the last equality. Now if $e\in{E^*\setminus E},$ then one can identify $e$ with some $x_e\in{U_{\kappa}},$ and by \eqref{enotinex*}, we have $e\notin{\mathcal{E}^{\textnormal{in}}}$ $\tilde{\P}^L_{\tilde{\G}^*}$-a.s, and so $e$ is crossed by a loop in $\mathcal{L}_{\frac12},$ if and only if $e$ is crossed by a loop in ${\mathcal{L}}_{\frac12}^{\, *},$ that is $x_e\in{\I_E^{\mathcal{L}}}.$ Therefore by \eqref{hatEsameasE},
\begin{equation}
\label{equalitypxG}
\begin{split}
    \tilde{\P}^L_{\tilde{\G}^*}\big(e\notin{{\mathcal{E}}}\,|\,{\mathcal{L}}^{e,*}_{\frac12},L_{|G}^{\textnormal{in}},\sigma^{\textnormal{in}}_{|G},{L}^{*}_{x_*}=u\big)&=\tilde{\P}^L_{\tilde{\G}^*}\big(e\notin{\mathcal{E}}\,|\,{\mathcal{L}}^{e,*}_{\frac12},L_{|G}^{\textnormal{in}},{L}^{*}_{x_*}=u\big)
    \\&=p_e^{\G^*}\big(\sqrt{L^{\textnormal{in}}+{L}^{*}}\big) 1_{x_e\notin{\I_E^{\mathcal{L}}},{L}^{*}_{x_*}=u}
    \\&=p_{x_e}^{u,\G}\big(\sigma^{\textnormal{in}}\sqrt{2L^{\textnormal{in}}},{L}^{*}\big) 1_{x_e\notin{\I_E^{\mathcal{L}}}},
\end{split}
\end{equation}
where we used \eqref{defpe} and \eqref{defpef} in the last equality. Finally, if $x\in{{G}\setminus U_{\kappa}},$ then $\kappa_x=0,$ $x\notin{\I_E^u},$ and $\tilde{\P}_{\tilde{\G}}(x\notin{\mathcal{E}_u}\,|\,\hat{\omega}_u,(\phi_x)_{x\in{{G}}})=1=p_x^{u,\G}\big(\phi,\ell_{\cdot,u}\big).$ Therefore since $\I_E^{\mathcal{L}}$ is obtained from $\mathcal{L}_{\frac12}^{e,*}$ in the same way that $\I_E^u$ is obtained from $\hat{\omega}^u,$ we obtain \eqref{lawofEu} by \eqref{disloopsvsinter}, \eqref{equalitypeG} and \eqref{equalitypxG}.
\end{proof}

\bibliography{bibliographie}
\bibliographystyle{abbrv}
\end{document}